\documentclass{amsart}

\usepackage{xcolor}
\usepackage{amsfonts,mathrsfs,amssymb,amsthm,mathptm}
\usepackage{eucal}
\usepackage{amsmath,amscd}
\usepackage{amsmath}
\usepackage{mathptm,pslatex}
\usepackage{multirow}
\usepackage{color}
\usepackage[all]{xy}
\usepackage{url}
\usepackage{cite}
\usepackage{exscale}
\usepackage{relsize}
\usepackage{bbm}
\usepackage{enumerate}
\usepackage{paralist}
\usepackage{bm}
\usepackage{hyperref}
\hypersetup{colorlinks=true,citecolor=red,linkcolor=blue,CJKbookmarks=true}
\usepackage{amsrefs}
\usepackage{tikz-cd}
\usepackage{graphicx}
\xyoption{all}

\usepackage{paralist}
\newtheorem{Thm}{Theorem}[section]
\newtheorem{Def}[Thm]{Definition}
\newtheorem{Lem}[Thm]{Lemma}
\newtheorem{Cor}[Thm]{Corollary}
\newtheorem{Prop}[Thm]{Proposition}
\newtheorem{Eg}[Thm]{Example}
\newtheorem{Rem}[Thm]{Remark}

\newtheorem{Nota}[Thm]{Notation}

\newcommand{\Hom}{{\rm{Hom}}}

\newcommand{\Ext}{{\rm{Ext}}}
\newcommand{\End}{{\rm{End}}}
\newcommand{\modn}{{\rm{mod}}}
\newcommand{\Mod}{{\rm{Mod}}}

\newcommand{\ind}{{\rm{ind}}}
\renewcommand{\ker}{{\rm{Ker}}}

\newcommand{\image}{{\rm{Im}}}

\newcommand{\rad}{{\rm{rad~}}}
\newcommand{\soc}{{\rm{soc}}}
\renewcommand{\top}{{\rm{top}}}

\newcommand{\proj}{{\rm{proj }}}
\newcommand{\inj}{{\rm{inj }}}

\newcommand{\gldim}{{\rm{gldim}}}
\newcommand{\add}{{\rm{add}}}

\newcommand{\resdim}{{\rm{resdim}}}
\newcommand{\resLambdaA}{{\rm{^\Lambda_A\pi}}}
\newcommand{\resLambdaB}{{\rm{^\Lambda_B\pi}}}

\newcommand{\longmapsfrom}{\mathrel{\reflectbox{\ensuremath{\longmapsto}}}}

\newcommand{\projdim}{{\rm{proj.dim}}}

\numberwithin{equation}{section}
\allowdisplaybreaks[4]

\begin{document}
\title{$nd\mathbb{Z}$-Cluster tilting Subcategories of $d$-Nakayama Algebras}

\author{Wei Xing}
\address{Uppsala University, Uppsala 75106, Sweden}
\curraddr{}
\email{wei.xing@math.uu.se}
\thanks{}

\subjclass[2010]{Primary ; Secondary }

\date{}
\dedicatory{}

\renewcommand{\thefootnote}{\alph{footnote}}
\setcounter{footnote}{-1} 


\begin{abstract}
    Jasso and K\"{u}lshammer introduced the class of $d$-Nakayama algebras as a higher dimensional analogue of Nakayama algebras.
    In particular, they are endowed with a distinguished $d\mathbb{Z}$-cluster tilting subcategory.
    In this paper, we investigate which $d$-Nakayama algebras admit an $nd\mathbb{Z}$-cluster tilting subcategory for $n > 1$. The radical square zero case is already covered by results on classical Nakayama algebras due to Herschend–Kvamme–Vaso. For each remaining non-self-injective $d$-Nakayama algebra, we provide a complete classification of its $nd\mathbb{Z}$-cluster tilting subcategories.
    In fact, there exists at most one for a suitable integer $n$.
    A self-injective $d$-Nakayama algebra is determined by two positive integers $m$ and $\ell$. 
    We show that an $nd\mathbb{Z}$-cluster tilting subcategory is only possible if $n|m$ and $n|(\ell-2)$.
    In case $n=\ell-2$, we show that such subcategory does indeed exist by constructing an explicit example.
\end{abstract}

\maketitle

\tableofcontents

\section{Introduction}

Auslander-Reiten theory is  fundamental in the study of representation theory of finite dimensional algebras from the perspective of homological algebra. 
A higher dimensional analogue was introduced in \cite{Iya07a, Iya07b} and has since found  connections to algebraic geometry \cite{IW11,IW13,IW14,HIMO23}, combinatorics \cite{OT12} and symplectic geometry \cite{DJL21}. It also plays a key role in the proof of the Donovan-Weymss conjecture \cite{JKM22}.   
In higher Auslander-Reiten theory, the object of study is a category $\mathcal{A}$,
typically the module category of a finite dimensional algebra or its bounded derived category, 
equipped with a $d$-cluster tilting subcategory $\mathcal{M} \subseteq\mathcal{A}$,
possibly with some additional properties.
Depending on different settings, $d$-cluster tilting subcategories give rise to higher notions in homological algebra.
For example, if $\mathcal{A}$ is abelian and $\mathcal{M}$ is $d$-cluster tilting,
then $\mathcal{M}$ is a $d$-abelian category in the sense of Jasso \cite{Jas16}.
If $\mathcal{A}$ is triangulated, 
$\mathcal{M}\subseteq \mathcal{A}$ is $d$-cluster tilting with the additional property that it is closed under $d$-fold suspension, 
then $\mathcal{M}$ is $(d+2)$-angulated in the sense of Geiss-Keller-Oppermann \cite{GKO13}.

Let $A$ be a finite dimensional algebra and denote by $\modn A$ the category of finitely generated right $A$-modules. The algebra $A$ is called $d$-representation finite if $\gldim A = d$ and there exists a $d$-cluster tilting subcategory $\mathcal{M}\subseteq\modn A$. In this case, $\mathcal{M}$ is unique and canonically induces a $d$-cluster tilting subcategory $\mathcal{U}$ in $\mathcal{D}^b(A)$. More precisely, 
\begin{equation*}
    \mathcal{U} = \add \{M[di]\mid M\in \mathcal{M} \text{ and } i\in \mathbb{Z}\}.
\end{equation*}
This case was intensively studied by many authors, see for example \cite{Iya11, IO11, IO13, HI11a, HI11b}. 
For $A$ with $\gldim A > d$, 
much fewer results are known.
Certain families of such algebras which are not self-injective were found by Vaso \cite{Vas19, Vas20, Vas21, Vas23}.
Large classes of self-injective algebras with a $d$-cluster tilting subcategory were constructed  in \cite{DI20}, including many instances of $n$-fold trivial extension algebras and higher preprojective algebras.
Other notable families of such algebras include Nakayama algebras and higher Nakayama algebras \cite{JK19}.

The notion of $d\mathbb{Z}$-cluster tilting property was introduced by Iyama and Jasso \cite{IJ17}.
This property is significantly stronger than the  $d$-cluster tilting property.
In particular, if an algebra $A$ with $\gldim A< \infty$ is such that  $\modn A$ admits a $d\mathbb{Z}$-cluster tilting subcategory,
then $d$ must divide $\gldim A$.
In the special case where $\gldim A = d$, any $d$-cluster tilting subcategory of $\modn A$ is trivially $d\mathbb{Z}$-cluster tilting.
A key benefit of the $d\mathbb{Z}$-cluster tilting property, as observed by Kvamme \cite{Kva21}, 
is that a $d\mathbb{Z}$-cluster tilting subcategory in the module category canonically induces such a subcategory in its singularity category.
Furthermore, the Auslander-Iyama correspondence introduced by Jasso-Muro \cite{JKM22} established a one-to-one correspondence between triangulated categories with a $d\mathbb{Z}$-cluster tilting object and twisted periodic self-injective algebras up to equivalence.

For the reasons given above, we are motivated to search for algebras that admit $d\mathbb{Z}$-cluster tilting subcategories.
A complete classification of such subcategories for Nakayama algebras was provided in \cite{HKV25}.
As a natural continuation, we restrict our attention in this paper to $d$-Nakayama algebras, since they are homologically complex yet computationally accessible.
By construction, a $d$-Nakayama algebra $A$ has a distinguished $d\mathbb{Z}$-cluster tilting subcategory in its module category \cite{JK19},
and this property forces $\gldim A\in d\mathbb{Z}$.
We further restrict our search to $nd\mathbb{Z}$-cluster tilting subcategories for a positive integer $n>1$. 
More precisely, we give a classification of $nd\mathbb{Z}$-cluster tilting subcategories for non-self-injective $d$-Nakayama algebras, and we construct explicit examples of such subcategories for self-injective $d$-Nakayama algebras satisfying an additional condition.
A necessary condition for $s$-representation finite $d$-Nakayama algebras was given in \cite{Sen23} for some positive integer $s$. The same paper posed the question that whether this necessary condition is also sufficient. As a consequence of our classification result, we provide an affirmative answer to this question.

Certain methods used in this paper are of independent interest, in particular the sink-source gluing of two algebras. Such gluing constructions have been considered in \cite{IPTZ87, Lev08, BCW15, Vas20}. We show that the algebra obtained by the sink-source gluing of algebras admitting a $d$-cluster tilting subcategory (respectively, a $d\mathbb{Z}$-cluster tilting subcategory) itself possesses a canonical $d$-cluster tilting subcategory (respectively, $d\mathbb{Z}$-cluster tilting subcategory) in its module category, which is obtained by gluing the corresponding $d$-cluster tilting subcategories (respectively, $d\mathbb{Z}$-cluster tilting subcategories) of each piece. This technique can be applied to a broader family of algebras to obtain more algebras with $d$-cluster tilting or $d\mathbb{Z}$-cluster tilting subcategories.

Our classification result for higher Nakayama algebras generalizes the classification of $n\mathbb{Z}$-cluster tilting subcategories for Nakayama algebras\cite{HKV25}.
For a $d$-Nakayama algebra $A$ with $\rad^2 A = 0$,
we may view it as a classical Nakayama algebra and apply the classification reault obtained in \cite{HKV25}.
More generally, a characterization of $d$-cluster tilting subcategories and $d\mathbb{Z}$-cluster tilting subcategories for radical square zero algebras given by quivers with relations was given in \cite{Vas23}.
Otherwise, for each non self-injective $d$-Nakayama algebra, we give a classification of $nd\mathbb{Z}$-cluster tilting subcategories for $n>1$.
In this case, such an integer $n$ is unique and the $nd\mathbb{Z}$-cluster tilting subcategory is also unique.
We subdivide the classification result into three cases -- depending on the Kupisch series -- namely, 
acyclic homogeneous, 
acyclic non-homogeneous and cyclic non-homogeneous.

The classification for acyclic homogeneous $d$-Nakayama algebras is given in Theorem \ref{thm_classification_homo_acyclic}.
In this case, all such algebras with $nd\mathbb{Z}$-cluster tilting subcategories are $2$-subhomogeneous (see \cite[Definition 3.1]{Xin25}).
Equivalently, the indecomposable objects in such an $nd\mathbb{Z}$-cluster tilting subcategory consist of projective or injective objects.
The classification result for acyclic non-homogeneous $d$-Nakayama algebras is given in Theorem \ref{thm_classification_nonhomo_acyclic}.
All such algebras are obtained by
sink-source gluing of acyclic homogeneous ones with $nd\mathbb{Z}$-cluster tilting subcategories.
The classification result for cyclic non-homogeneous $d$-Nakayama algebras is given in Theorem \ref{thm_classification_non_selfinj_cyclic_case}.
Such algebras are obtained by self-gluing acyclic non-homogeneous ones with $nd\mathbb{Z}$-cluster tilting subcategories.

For a self-injective $d$-Nakayama algebra, such a classification result seems to be infeasible.
So instead we construct an explicit example for each such algebra with extra conditions.

\section{Preliminaries}
\subsection{Notations and conventions}
Throughout this paper,
we fix positive integers $d$ and $n$.
We work over an arbitrary field $\Bbbk$.
Unless stated otherwise,
all algebras are finite dimensional $\Bbbk$-algebras and all modules are finite dimensional right modules.
We denote by $D$ the $\Bbbk$-duality $\Hom_{\Bbbk}(-, \Bbbk)$.

All subcategories considered are supposed to be full.
Let $F: \mathcal{C}\rightarrow \mathcal{D}$ be a functor,
the essential image of $F$ is the full subcategory of $\mathcal{D}$ given by 
\begin{equation*}
    F\mathcal{C} = \{D\in \mathcal{D} \mid \exists  C\in \mathcal{C} \text{ such that } FC \cong D\}.
\end{equation*}

Let $A$ be a finite dimensional algebra over $\Bbbk$.
Denote by $\Mod A$ the category of right $A$-modules and 
$\modn A$ the subcategory consisting of finitely generated right $A$-modules.
Denote by $\ind A$ the set of isomorphism classes of indecomposable $A$-modules.
For a subcategory $\mathcal{C}\subseteq \modn A$, 
we denote by $\ind \mathcal{C}$ the set of isomorphism classes of indecomposable objects in $\mathcal{C}$. 
We use the notation $\mathcal{J}_{\mathcal{C}}$ to denote the ideal of $\mathcal{C}$ such that for $M, N\in \mathcal{C}$, 
we have 
\begin{equation*}
    \mathcal{J}_{\mathcal{C}}(M,N)=\{f:M\rightarrow N\mid \image f\subseteq \rad N\}.
\end{equation*}
We denote by $\underline{\modn}A$ the projectively stable module category of $A$,
that is the category with the same  objects as $\modn A$ and morphisms given by $\underline{\Hom}_A(M,N) = \Hom_A(M,N)/\mathcal{P}(M,N)$ where $\mathcal{P}(M, N)$ denotes the subspace of morphisms factoring through projective modules.
We denote by $\Omega: \underline{\modn}A\rightarrow \underline{\modn}A$ the syzygy functor defined by $\Omega(M)$ being the kernel of the projective cover $P(M)\twoheadrightarrow M$.
Let $\Omega^0(M) = M$ and $\Omega^{i+1}(M) = \Omega(\Omega^i(M))$ for $i\geq 0$.
The injectively stable module category $\overline{\modn}A$ of $A$ and the cosyzygy functor $\Omega^{-1}: \overline{\modn}A\rightarrow \overline{\modn}A$ are defined dually.
When $A$ is a self-injective algebra, $\modn A$ is Frobenius and thus $\underline{\modn} A$ has a triangulated category structure with the suspension functor $\Omega^{-1}$.
We refer to \cite{Hap88} for more details.

We consider the $d$-Auslander-Reiten translations $\tau_d: \underline{\modn}A\rightarrow \overline{\modn}A$    and $\tau_d^{-1}: \overline{\modn}A\rightarrow \underline{\modn}A$
defined by $\tau_d = \tau\Omega^{d-1}$ and $\tau_d^{-1} = \tau^{-1}\Omega^{-(d-1)}$
where $\tau$ and $\tau^{-1}$ denote the usual Auslander-Reiten translations.

\subsection{$d\mathbb{Z}$-cluster tilting subcategories.}

Let $\mathcal{M}$ be a subcategory of a category $\mathcal{C}$ and let $C\in \mathcal{C}$.
A right $\mathcal{M}$-approximation of $C$ is a morphism $f: M\rightarrow C$ with $M\in \mathcal{M}$ such that all morphisms $g: M'\rightarrow C$ with $M'\in \mathcal{M}$ factor through $f$. We say that 
$\mathcal{M}$ is contravariantly finite in $\mathcal{C}$ if every $C\in \mathcal{C}$ admits a right $\mathcal{M}$-approximation.
The notions of left $\mathcal{M}$-approximation and covariantly finite are defined dually.
We say that $\mathcal{M}$ is functorially finite in $\mathcal{C}$ if $\mathcal{M}$ is both contravariantly finite and covariantly finite.
In particular, if $M\in \modn A$, 
then $\add M$ is functorially finite.
Recall in the case when $\mathcal{C}$ is abelian, $\mathcal{M}$ is called a generating (resp. cogenerating) subcategory if for any object $C\in\mathcal{C}$, there exists an epimorphism $M\rightarrow C$ (resp. monomorphism $C \rightarrow M$) with $M\in \mathcal{M}$. In particular, $\add M\subseteq\modn A$ is generating if and only if $A\in \add M$ and cogenerating if and only if $DA\in \add M$.
In this case, we call $M$ a generator and cogenerator correspondingly.
And $M$ is a generator-cogenerator if it is both a generator and a cogenerator.

\begin{Def}[\cite{Iya11, IY08, IJ17}]
    Let $d$ be a positive integer.
    Let $\mathcal{C}$ be an abelian or a triangulated category,
    and $A$ a finite dimensional $\Bbbk$-algebra.
    \begin{itemize}
        \item [(a)] We call a subcategory $\mathcal{M}$ of $\mathcal{C}$ a $d$-cluster tilting subcategory if it is functorially finite, generating-cogenerating if $\mathcal{C}$ is abelian and 
        \begin{align*}
            \mathcal{M} & = \{ C\in \mathcal{C} \mid \Ext_{\mathcal{C}}^i(C, \mathcal{M}) = 0 \text{ for } 1\leq i\leq d-1\}\\
            & = \{ C\in \mathcal{C} \mid \Ext_{\mathcal{C}}^i(\mathcal{M}, C) = 0 \text{ for } 1\leq i\leq d-1\}.
        \end{align*}
        If moreover $\Ext_{\mathcal{C}}^i(\mathcal{M}, \mathcal{M}) \neq 0$ implies that $i\in d\mathbb{Z}$,
        then we call $\mathcal{M}$ a $d\mathbb{Z}$-cluster tilting subcategory.
        \item [(b)] A module $M\in \modn A$ is called a $d$-cluster tilting module (respectively $d\mathbb{Z}$-cluster tilting module) if $\add M$ is a $d$-cluster tilting subcategory (respectively $d\mathbb{Z}$-cluster tilting subcategory) of $\modn A$.
    \end{itemize}
\end{Def}

We collect certain homological properties of $d$-cluster tilting subcategories in the following proposition.

\begin{Prop}\cite[Theorem 2.8]{Iya08}\cite[Corollary 3.3]{Vas19}\label{prop_indec_equiv}
Let $A$ be an algebra and let $\mathcal{C}$ be a $d$-cluster tilting subcategory of $\modn A$.
Then the following statements hold.
\begin{itemize}
    \item [(a)] $A\in \mathcal{C}$ and $DA\in \mathcal{C}$.
    \item [(b)] Denote by $\mathcal{C}_P$ and $\mathcal{C}_I$ the sets of isomorphism classes of indecomposable non-projective respectively non-injective modules in $\mathcal{C}$. 
    Then $\tau_{d}$ and $\tau_{d}^{-1}$ induce mutually inverse bijections
    \begin{equation*}
        \xymatrix{\mathcal{C}_P\ar@/^/[r]^{\tau_{d}} & \mathcal{C}_I.\ar@/^/[l]^{\tau_{d}^{-1}}\\}
    \end{equation*}
    \item [(c)] $\Omega^iM$ is indecomposable for all $M\in \mathcal{C}_P$ and $0<i<d$. Dually, $\Omega^{-i}N$ is indecomposable for all $N\in \mathcal{C}_I$ and $0<i<d$.
    \item [(d)] The following statements are equivalent.
    \begin{itemize}
        \item [(i)] $\mathcal{C}$ is $d\mathbb{Z}$-cluster tilting.
        \item [(ii)] $\Omega^d(\mathcal{C})\subseteq \mathcal{C}$.
        \item [(iii)] $\Omega^{-d}(\mathcal{C}) \subseteq \mathcal{C}$.
    \end{itemize}
\end{itemize}
\end{Prop}

Recall that $M\in \modn A$ is $d$-rigid for some positive integer $d$ if 
\begin{equation*}
    \Ext_A^i(M, M) = 0, \text{ for all } 0<i<d.
\end{equation*}
The following characterization of $d$-cluster tilting modules from \cite[Lemma 2.12]{Iya08} will be useful for us. 

\begin{Lem}(\cite[Lemma 2.12]{Iya08})\label{lem_characterisation_of_d_ct_module}
    Let $A$ be an algebra and let $M\in \modn A$ be  a $d$-rigid generator-cogenerator.
    Then the following conditions are equivalent.
    \begin{itemize}

        \item [(a)] $M$ is a $d$-cluster tilting module.
        \item [(b)] $\gldim \End_A(M)\leq d+1$.
        \item [(c)] For all indecomposable object $X\in \add M$,
        there exists an exact sequence 
        \begin{equation*}
            \xymatrix@C=1em@R=1em{
            0\ar[r] & M_{d+1}\ar[r]^{f_{d+1}} & M_d\ar[r]^{f_d} & \cdots\ar[r]^{f_1} & M_1\ar[r]^{f_1} & X 
            }
        \end{equation*}
        with $M_i\in \add M$ for all $1\leq i\leq d+1$ such that the following sequence is exact.
        \begin{equation*}
            \xymatrix@C=1em@R=1em{
            0\ar[r] & \Hom_{A}(M, M_{d+1})\ar[r]^-{f_{d+1}^{\ast}} & \Hom_{A}(M, M_{d})\ar[r]^-{f_{d}^{\ast}} & \cdots\ar[r]^-{f_2^{\ast}} & \Hom_{A}(M, M_{1})\ar[r]^-{f_{1}^{\ast}} & \mathcal{J}_A(M, X)\ar[r] & 0.
            }
        \end{equation*}
    \end{itemize}
\end{Lem}

Let $\Gamma:=\End_A(M)$ be the endomorphism algebra of $M$.
The global dimension of $\Gamma$, denoted by $\gldim \Gamma$,
by definition is the supremum of projective dimensions of all simple $\Gamma$-modules.
To calculate $\gldim \Gamma$, we recall the following definition.

\begin{Def}
    Let $M\in \modn A$.
    For each $X\in \modn A$, let  $f_X: M_X\rightarrow X$ be the minimal right $\add M$-approximation of $X$.
    We call $\Omega_M(X) := \ker f_X$ the $M$-syzygy of $X$. 

    For each $n\in \mathbb{N}$, 
    define $\Omega_M^n(X) := \Omega_M(\Omega_M^{n-1}(X))$ by iteration,
    where $\Omega_M^0(X) = X$. 
    The $M$-resolution dimension of $X$, denoted by $\resdim_M(X)$, is the minimal natural number $n$ or $\infty$,
    such that $\Omega_M^n(X)\in \add M$.
\end{Def}

Further, we recall the following Lemma which will be used to calculate $\gldim \Gamma$ later.

\begin{Lem}\cite[Lemma 2.3]{CX22}\label{lem_glodim_endomorphism_algebra}
    Let $M = \bigoplus_{1\leq i\leq n}M_i$ with $M_i$ mutually non-isomorphic indecomposable direct summands.
    For each $1\leq i\leq n$, 
    set $S_i=\top \Hom_A(M, M_i)$
    and $M_{\hat{i}}:=\bigoplus_{1\leq j\leq n, j\neq i}M_j$.
    Let $f_i: U_i\rightarrow M_i$ be the minimal right $\add M_{\hat{i}}$-approximation of $M_i$. 
    Then the following statements hold true.
    \begin{itemize}
        \item [(a)] If $M_i$ is projective, 
        then 
        \begin{equation*}
            \projdim(_\Gamma S_i) = \left\{
            \begin{array}{cc}
               0  & \text{ if } \Hom_A(M, \rad M_i) = 0  \\
                \resdim_M(\rad(M_i))+1 & \text{ otherwise. } 
            \end{array}\right.
        \end{equation*}
        \item [(b)] $\Ext_{\Gamma}^1(S_i, S_i)=0$ if and only if $\image(\Hom_A(M_i, f_i)) = \rad (\End_A(M_i)).$
        In this case, 
        \begin{equation*}
            \projdim(_\Gamma S_i) =\left\{
            \begin{array}{cc}
                0 & \text{ if } U_i = 0; \\
                1 & \text{ if } U_i\neq 0, \ker f_i=0;\\
                \resdim_M(\ker f_i) + 2 & \text{ otherwise.}
            \end{array}\right.
        \end{equation*}
    \end{itemize}
\end{Lem}

We introduce the definition of partial $d$-cluster tilting subcategories.
They are $d$-rigid but not 'maximal' compared to $d$-cluster tilting subcategories.
Additionally they interact with $\tau_d^{\pm 1}$ in a nice way which will be useful for our classification.

\begin{Def}\label{def_weakly_d_ct}
    Let $A$ be an algebra and let $\mathcal{C}\subseteq \modn A$ be a subcategory. 
    We call $\mathcal{C}$ a partial $d$-cluster tilting subcategory if the following conditions hold.
    \begin{itemize}
        \item [(a)] $A, DA\in \mathcal{C}$.
        \item [(b)] $\Ext_A^i(\mathcal{C}, \mathcal{C})=0$ for all $0<i<d$.
        \item [(c)] Denote by $\mathcal{C}_p$ and $\mathcal{C}_I$ the sets of isomorphism classes of indecomposable non-projective respectively non-injective modules in $\mathcal{C}$. 
        Then $\tau_d$ and $\tau_d^{-1}$ induce mutually inverse bijections
        \begin{equation*}
            \xymatrix{\mathcal{C}_P\ar@/^/[r]^{\tau_{d}} & \mathcal{C}_I.\ar@/^/[l]^{\tau_{d}^{-1}}\\}
        \end{equation*}
        \item [(d)] $\Omega^i M$ is indecomposable for all $M\in \mathcal{C}_P$ and $0<i<d$. 
        Dually, $\Omega^{-i}N$ is indecomposable for all $N\in \mathcal{C}_I$ and $0<i<d$.
    \end{itemize}
    We call $\mathcal{C}$ a partial $d\mathbb{Z}$-cluster tilting subcategory if additionally, 
    \begin{itemize}
        \item [(e)] $\Omega^{d}(\mathcal{C}) \subseteq \mathcal{C}$ and $\Omega^{-d}(\mathcal{C}) \subseteq \mathcal{C}$.
    \end{itemize} 
\end{Def}

\begin{Rem}\label{rem_partial_dZ_ct_def_implies_dZ_rigidity}
    Assume $\mathcal{C}$ is a partial $d\mathbb{Z}$-cluster tilting subcategory. 
    Then Definition \ref{def_weakly_d_ct} (e) implies the following.
    \begin{itemize}
        \item [(e')] $\Ext_A^i(\mathcal{C}, \mathcal{C})\neq 0$ implies $i\in d\mathbb{Z}$.
    \end{itemize}
    Notice that when $\mathcal{C}$ is $d\mathbb{Z}$-cluster tilting, (e) and (e') are equivalent. 
    However, this fails for partial $d\mathbb{Z}$-cluster tilting subcategories as they lack the maximality property.
\end{Rem}

\begin{Rem}
    In \cite{Vas19}, it was shown that $d$-cluster tilting subcategories and partial $d$-cluster tilting subcategories coincide for representation-directed algebras.
    In general, $d$-cluster tilting implies partial $d$-cluster tilting, which can be seen directly from Proposition \ref{prop_indec_equiv}.
    However partial $d$-cluster tilting subcategories are not necessarily $d$-cluster tilting subcategories. For example, if we take $A$ to be a self-injective algebra,
    then $\add A$ is trivially partial $d$-cluster tilting for any $d>0$ but it is not $d$-cluster tilting unless $A$ is semisimple.
\end{Rem}

\begin{Prop}\label{prop_d_ct_implies_weakly_d_ct}
    Let $A$ be an algebra and let $\mathcal{C}\subseteq \modn A$ be a $d$-cluster tilting   subcategory. 
    Then $\mathcal{C}$ is partial $d$-cluster tilting.
\end{Prop}

Since we are interested in $nd\mathbb{Z}$-cluster tilting subcategories of an algebra which admits a $d\mathbb{Z}$-cluster tilting subcategory, 
we compare the functors $\tau_d^{\pm}$, $\tau_{nd}^{\pm}$ and $\Omega^{\pm d}$.

\begin{Prop}\label{prop_tau_nd_and_tau_d} 
    We have $\tau_{nd} = \tau_d \Omega^{(n-1)d}$ and $\tau_{nd}^{-1} = \tau_d^{-1}\Omega^{-(n-1)d}$.
\end{Prop}

The following Proposition indicates that for an algebra $A$  with a $d\mathbb{Z}$-cluster tilting subcategory $\mathcal{M}$  and a partial $nd\mathbb{Z}$-cluster tilting subcategory  $\mathcal{C}$, 
we can always find a partial $nd\mathbb{Z}$-cluster tilting subcategory of $\modn A$ inside $\mathcal{M}$. 
This is critical for our strategy of classification as it will allow us to focus on a specific $\mathcal{M}$ which is much more accessible than all of $\modn A$.

\begin{Prop}\label{prop_C_intersecting_M_gives_weakly_ndZ_ct}
    Let $A $ be an algebra.
    Assume $\mathcal{M}\subseteq \modn A$ is a $d\mathbb{Z}$-cluster tilting subcategory 
    and 
    $\mathcal{C}\subseteq \modn A$ is a partial  $nd\mathbb{Z}$-cluster tilting subcategory.
    Then $\mathcal{C}^{\mathcal{M}} = \mathcal{C}\cap \mathcal{M}\subseteq \modn A$ is partial $nd\mathbb{Z}$-cluster tilting.
\end{Prop}

\begin{proof}
    We verify  condition (c) in Definition \ref{def_weakly_d_ct} since condition (a), (b) and (d) are immediate.
    It suffices to show that $\tau_{nd}\mathcal{C}^{\mathcal{M}}_P \subseteq \mathcal{C}^{\mathcal{M}}_I$ and $\tau_{nd}^{-1}\mathcal{C}^{\mathcal{M}}_I \subseteq \mathcal{C}^{\mathcal{M}}_P$.

    Let $M\in \mathcal{C}^{\mathcal{M}}_P = \mathcal{C}_P\cap \mathcal{M}_P$. 
    We have $\tau_{nd}M =  \tau_d\Omega^{(n-1)d}M\in \mathcal{C}_I$.
    Moreover $\Omega^{(n-1)d}M$ is indecomposable and non-projective by Definition \ref{def_weakly_d_ct} (d).
    So $\Omega^{(n-1)d}M\in \mathcal{M}_P$ as $M\in \mathcal{M}$ and $\Omega^d\mathcal{M}\subseteq \mathcal{M}$.
    Hence $\tau_d\Omega^{(n-1)d}M\in \mathcal{M}_I$.
    Thus we have $\tau_{nd}M\in \mathcal{C}_I\cap \mathcal{M}_I = \mathcal{C}^{\mathcal{M}}_I$.

    Arguing similarly we have $\tau_{nd}^{-1}\mathcal{C}^{\mathcal{M}}_I\subseteq \mathcal{C}^{\mathcal{M}}_P$.

    So the inverse bijections 
    \begin{equation*}
            \xymatrix{\mathcal{C}_P\ar@/^/[r]^{\tau_{nd}} & \mathcal{C}_I.\ar@/^/[l]^{\tau_{nd}^{-1}}\\}
    \end{equation*}
    restricts to inverse  bijections 
    \begin{equation*}
            \xymatrix{\mathcal{C}^{\mathcal{M}}_P\ar@/^/[r]^{\tau_{nd}} & \mathcal{C}^{\mathcal{M}}_I.\ar@/^/[l]^{\tau_{nd}^{-1}}\\}
    \end{equation*}
\end{proof}

\begin{Prop}\label{prop_more_modules_from_tau_d_Omega_-1}
    Let $A$ be an algebra.
    Assume $\mathcal{M}\subseteq \modn A$ is a  $d\mathbb{Z}$-cluster tilting subcategory and $\mathcal{C}\subseteq \modn A$ is a partial  $nd\mathbb{Z}$-cluster tilting subcategory.
    Denote by $\mathcal{C}^{\mathcal{M}} = \mathcal{C}\cap \mathcal{M}$ the induced partial  $nd\mathbb{Z}$-cluster tilting.
    Then the following statements hold.
    \begin{itemize}
        \item [(a)] If $M\in \mathcal{C}^{\mathcal{M}}_P$,
        then $\Omega^{-d}\tau_d M\in \mathcal{C}$.
        \item [(b)] If $N\in \mathcal{C}^{\mathcal{M}}_I$,
        then $\Omega^d\tau_d^{-1}N\in \mathcal{C}$.
    \end{itemize}
\end{Prop}

\begin{proof}
    We prove (a) as (b) is similar.
    As $M\in \mathcal{C}^{\mathcal{M}}_P$, there exists $N\in \mathcal{C}^{\mathcal{M}}_I$ such that 
    \begin{equation*}
        M = \tau_{nd}^{-1}N = \tau_d^{-1}\Omega^{-(n-1)d}N.
    \end{equation*}
    As $N\in \mathcal{C}_I$, $\Omega^{-(n-1)d}N$ is indecomposable non-injective by Definition \ref{def_weakly_d_ct} (d).
    Thus $\Omega^{-(n-1)d}N\in \mathcal{M}_I$ since $\Omega^{-d}\mathcal{M}\subseteq \mathcal{M}$ and $N\in \mathcal{M}_I$.
    So 
    \begin{equation*}
        \tau_d \tau_d^{-1}\Omega^{-(n-1)d}N\cong \Omega^{-(n-1)d}N.
    \end{equation*}
    Thus 
    \begin{equation*}
        \Omega^{-d}\tau_d M = \Omega^{-d}\tau_d\tau_d^{-1}\Omega^{-(n-1)d}N\cong \Omega^{-nd}N\in \mathcal{C}
    \end{equation*}
    as $\Omega^{-nd}\mathcal{C}\subseteq \mathcal{C}$.
\end{proof}

\subsection{The sink-source gluing}\label{preliminary_gluing}

In this section, we introduce the gluing of algebras given by quivers with relations along a sink and a source, which will be called a sink-source gluing. 
Such gluings or pullbacks were considered in \cite{IPTZ87, Lev08, BCW15}.
In \cite{Vas20}, the gluing of representation-directed algebras along left and right abutments of height $h$ for a positive integer $h$ was introduced. 
The sink-source gluing is a specialization of gluing along left and right abutments of height $0$. 
The difference here is that we don't restrict to representation-directed algebras.
To establish notation as well as make the article self-contained,
we give full proof of some results which may be already considered in \cite{IPTZ87, Lev08, BCW15, Vas20}.

\begin{Def}\label{def_glue_two_algs}
    Let $A= \Bbbk Q_A/\mathcal{R}_A$ be an algebra in which there is a sink vertex $a$ in $Q_A$ and $B=\Bbbk Q_B/\mathcal{R}_B$ be an algebra in which there is a source vertex $b$ in $Q_B$.
    The sink-source gluing of $A$ and $B$, denoted by $\Lambda := A\Delta B$, is defined to be the algebra $\Lambda = \Bbbk Q_{\Lambda}/\mathcal{R}_{\Lambda}$ given by a quiver with relations where 
    \begin{align*}
        (Q_{\Lambda})_0 & = ((Q_A)_0 \cup (Q_B)_0) /(a\sim b) \\
        (Q_{\Lambda})_1 & = (Q_A)_1 
        \cup (Q_B)_1.
    \end{align*}
    Note that $Q_A$ and $Q_B$ are full subquivers of $Q_{\Lambda}$. 
    We set $e_A := \sum_{i\in (Q_A)_0}e_i$ for the sum of all primitive idempotents of $\Lambda$ corresponding to the vertices of $Q_A$, and similarly we set $e_B :=\sum_{i\in (Q_B)_0}e_i$.
    The relations $\mathcal{R}_{\Lambda}$ are given as follows.
    \begin{equation*}
        \mathcal{R}_{\Lambda} = \langle \mathcal{R}_A + \mathcal{R}_B + (1-e_A)\Bbbk Q_{\Lambda}(1-e_B)\rangle.
    \end{equation*}
\end{Def}

Let $\Lambda = A\Delta B$. We have a pull-back diagram as follows.
\begin{equation*}
    \xymatrix{
        \Lambda\ar[r]^{\resLambdaB}\ar[d]_\resLambdaA & B\ar[d]\\
            A\ar[r] & \Bbbk. }
\end{equation*} 
Here
\begin{equation*}
    \resLambdaA: \Lambda\rightarrow \Lambda/\langle 1-e_A \rangle \cong A \text{ and }
    \resLambdaB: \Lambda\rightarrow \Lambda/\langle 1-e_B \rangle \cong B 
\end{equation*}
are ring epimorphisms.
Indeed we have the following recollement of module categories induced by the idempotent $1-e_A$  (see \cite{Chr14, CJ14} for more details).
\begin{equation}\label{recollement_modA_modLambda}
    \xymatrix@C=5em{
        \modn A\ar[r]^{^\Lambda_A\pi_{\ast}} & 
        \modn \Lambda \ar@/^2em/[l]^{\resLambdaA^{!}}\ar@/_2em/[l]_{\resLambdaA^{\ast}}\ar[r] & \modn (1-e_A)\Lambda(1-e_A)\ar@/^2em/[l]\ar@/_2em/[l]
        }.
\end{equation}
Here $\resLambdaA_{\ast}$ is a fully faithful embedding which is also exact. 
The left and right adjoints of $\resLambdaA$ are 
\begin{equation*}
    \resLambdaA^{\ast} = -\otimes_{\Lambda}\Lambda/\langle 1-e_A\rangle 
    \text{ and } 
    \resLambdaA^{!} = \Hom_{\Lambda}(A, -).
\end{equation*}

We have a similar recollement induced by the idempotent $1-e_B$.

\begin{Rem}
    Denote by $S_a\in \modn A$ the simple injective $A$-module at vertex $a$ and $S_b\in \modn B$ the simple projective $B$-module at vertex $b$.
    Let $S$ be the simple $\Lambda$-module at vertex $a$, or equivalently $b$. 
    We identify both $\resLambdaA_{\ast}S_a$ and $\resLambdaB_{\ast}S_b$ with $S$ along the isomorphisms
    \begin{equation*}
        S \cong  \resLambdaA_{\ast}S_a \cong  \resLambdaB_{\ast}S_b.
    \end{equation*}
    We call $S$ the simple bridge of the sink-source gluing 
    and we use the notation $ A\Delta_S B = A\Delta B$ to emphasize $S$.
\end{Rem}

In the following Proposition, we study the image of projective objects under the adjoint pairs $(\resLambdaA^{\ast}, \resLambdaA_{\ast})$ and $(\resLambdaB^{\ast}, \resLambdaB_{\ast})$. 
A dual statement is obtained for the injective objects in a similar way. 
As a Corollary, we show the compatibility between the Nakayama functors and the functors $\resLambdaA_{\ast}, \resLambdaB_{\ast}$.

\begin{Prop}\label{prop_image_of_projects_lambda_A_B}
    Let $\Lambda = A\Delta_S B$.
    Then the following statements hold true.
    \begin{itemize}
        \item [(a)] The functors 
        $\resLambdaA^{\ast}$ and $\resLambdaB^{\ast}$ preserve projective objects. 
        More precisely,
        \begin{equation*}
            \resLambdaA^{\ast}(e_i\Lambda)  \cong 
            \left\{
        \begin{array}{cc}
            e_i A & i\in (Q_A)_0 \\
            0 & \text{ otherwise. } 
        \end{array}
        \right.
        \text{ and~~~~ }
        \resLambdaB^{\ast}(e_i\Lambda)  \cong 
        \left\{
        \begin{array}{cc}
            e_i B & i\in (Q_B)_0 \\
            0 & \text{ otherwise. } 
        \end{array}
        \right.
        \end{equation*}
        \item [(b)]
        We have $\resLambdaA_{\ast}(e_iA) \cong e_i\Lambda$ and 
        \begin{equation*}
            \resLambdaB_{\ast}(e_iB) \cong 
            \left\{
        \begin{array}{cc}
            e_i \Lambda & i\neq b \\
            S & i = b 
        \end{array}
        \right.
        \end{equation*}
        \item [(a')] The functors $\resLambdaA^!$ and $\resLambdaB^!$ preserve injective objects. 
        More recisely,
        \begin{equation*}
            \resLambdaA^{!}(D\Lambda e_i)  \cong 
            \left\{
        \begin{array}{cc}
            DAe_i & i\in (Q_A)_0 \\
            0 & \text{ otherwise. } 
        \end{array}
        \right.
        \text{ and~~~~ }
        \resLambdaB^{!}(D\Lambda e_i)  \cong 
        \left\{
        \begin{array}{cc}
            DBe_i & i\in (Q_B)_0 \\
            0 & \text{ otherwise. } 
        \end{array}
        \right.
        \end{equation*}
        \item [(b')]
        We have $\resLambdaB_{\ast}(DBe_i) \cong D\Lambda e_i$ and 
        \begin{equation*}
            \resLambdaA_{\ast}(DAe_i) \cong 
            \left\{
        \begin{array}{cc}
            D\Lambda e_i & i\neq a \\
            S & i = a 
        \end{array}
        \right.
        \end{equation*}
    \end{itemize}
\end{Prop}

\begin{proof}
    Part (a) follows from the properties of the recollement. 
    We prove (b).

    \begin{equation*}
        \resLambdaA_{\ast}(e_iA) \cong (e_i\Lambda)/(e_i\Lambda(1-e_A)\Lambda) \cong e_i\Lambda
    \end{equation*}
    as $e_i\Lambda (1-e_A) = 0$ for $i\in (Q_A)_0$ by definition. 
    Similarly, 
    \begin{equation*}
        \resLambdaB_{\ast}(e_iB) \cong (e_i\Lambda)/(e_i\Lambda(1-e_B)\Lambda) \cong 
        \left\{
        \begin{array}{cc}
           e_i\Lambda  & i\neq b \\
           S  & i=b. 
        \end{array}
        \right.
    \end{equation*}
    Dually, statement (a') and (b') can be proved similarly.
\end{proof}

\begin{Cor}\label{cor_nakayama_functor_lambda_A_B}
    Let $\Lambda = A\Delta_S B$. We have 
    \begin{equation*}
        \nu_{\Lambda}(\resLambdaA_{\ast}(e_i A)) \cong \resLambdaA_{\ast}\nu_A(e_i A), i\in (Q_A)_0\backslash \{a\},
    \end{equation*}
    and 
    \begin{equation*}
        \nu_{\Lambda}(\resLambdaB_{\ast}(e_i B)) \cong \resLambdaB_{\ast}\nu_B(e_i B), i\in (Q_B)_0\backslash \{b\}.
    \end{equation*}
    Moreover, 
    \begin{align*}
        \resLambdaA_{\ast}\nu_{A}(e_a A) & \cong S,\\
        \nu_{\Lambda}(\resLambdaB_{\ast}(e_bB)) & \cong \nu_{\Lambda}(S),\\
        \nu_{\Lambda}(\resLambdaA_{\ast}(e_aA)) & \cong \resLambdaB_{\ast}\nu_B(e_b B).
    \end{align*}
\end{Cor}

\begin{proof}
    Since the Nakayama functor maps the indecomposable projective module at each vertex to the corresponding indecomposable injective module,
    the statement follows by applying Proposition \ref{prop_image_of_projects_lambda_A_B}.
\end{proof}

We describe the morphism spaces between objects in $\resLambdaA_{\ast}(\modn A)$ and $\resLambdaB_{\ast}(\modn B)$. 
As a result of sink-source gluing, all such morphisms factor through the simple bridge $S$.
We formulate this fact via left or right $\add S$-approximations.

\begin{Prop}\label{prop_hom_space_of_Lambda}
    Let $\Lambda = A\Delta_S B$. 
    Let $M\in \modn A$ and $N\in \modn B$.
    The following statements hold true.
    \begin{itemize}
        \item [(a)] We have 
        \begin{equation*}
        \Hom_{\Lambda}(\resLambdaB_{\ast}N,\resLambdaA_{\ast}M) \cong \Hom_B(N, S)\otimes_\Bbbk\Hom_A(S, M).
        \end{equation*}
        In particular, 
        $\Hom_{\Lambda}(^{\Lambda}_B\pi_{\ast}N,\resLambdaA_{\ast}M) = 0$ if $S\notin \add M$ or $S\notin \add N$.
        \item [(b)] Similarly,
        \begin{equation*}
        \Hom_{\Lambda}(\resLambdaA_{\ast}M, \resLambdaB_{\ast}N) \cong \Hom_A(M, S)\otimes_\Bbbk\Hom_B(S, N).
        \end{equation*}
        More precisely, 
        \begin{equation*}
            \Hom_{\Lambda}(\resLambdaA_{\ast}M, \resLambdaB_{\ast}N) \cong \Bbbk^{\ell_M}\otimes_{\Bbbk} \Hom_B(S, N) \cong \Hom_A(M, S)\otimes_{\Bbbk} \Bbbk^{r_N}
        \end{equation*}
        where $f_M: M\rightarrow S^{\ell_M}$ is the minimal left $\add S$-approximation of $M$ and $g_N: S^{r_N}\rightarrow N$ is the minimal right $\add S$-approximation of $N$.
        In particular, 
        \begin{equation*}
            \Hom_{\Lambda}(\resLambdaA_{\ast}M, \resLambdaB_{\ast}N) = 0
        \end{equation*}
        if $f_M=0$ or $g_N=0$.
    \end{itemize}
\end{Prop}

\begin{proof}
    As $\resLambdaA_{\ast}(\modn A)\cap \resLambdaB_{\ast}(\modn B) = \add S$,
    all morphisms from $\resLambdaB_{\ast}N$ to $\resLambdaA_{\ast}M$ factor through $S$.
    So we have  
    \begin{equation*}
        \Hom_{\Lambda}(\resLambdaB_{\ast}N, S)\otimes_{\Bbbk} \Hom_{\Lambda}(S, \resLambdaA_{\ast}M)\twoheadrightarrow 
        \Hom_{\Lambda}(\resLambdaB_{\ast}N, \resLambdaA_{\ast}M).
    \end{equation*}
    Since $\dim_{\Bbbk} S=1$, 
    the isomorphism 
    \begin{equation*}
        \Hom_{\Bbbk}(\resLambdaB_{\ast}N, S)\otimes_{\Bbbk} \Hom_{\Bbbk}(S, \resLambdaA_{\ast}M)\cong 
        \Hom_{\Bbbk}(\resLambdaB_{\ast}N, \resLambdaA_{\ast}M)
    \end{equation*}
    restricts to an injective map 
    \begin{equation*}
        \Hom_{\Lambda}(\resLambdaB_{\ast}N, S)\otimes_{\Bbbk} \Hom_{\Lambda}(S, \resLambdaA_{\ast}M)\hookrightarrow 
        \Hom_{\Lambda}(\resLambdaB_{\ast}N, \resLambdaA_{\ast}M).
    \end{equation*}
    Thus
    \begin{align*}
        \Hom_{\Lambda}(\resLambdaB_{\ast}N, \resLambdaA_{\ast}M) & \cong \Hom_{\Lambda}(\resLambdaB_{\ast}N, S)\otimes_{\Bbbk} \Hom_{\Lambda}(S, \resLambdaA_{\ast}M) \\
        & \cong \Hom_B(N, S_b)\otimes_{\Bbbk} \Hom_A(S_a, M)
    \end{align*}
    as $\resLambdaA_{\ast}$ and $\resLambdaB_{\ast}$ are fully faithful embeddings.
    For a similar reason, we have 
    \begin{equation*}
        \Hom_{\Lambda}(\resLambdaA_{\ast}M, \resLambdaB_{\ast}N) \cong \Hom_A(M, S_a)\otimes_{\Bbbk}\Hom_B(S_b, N).
    \end{equation*}
    
    Since $S_b\in \modn B$ is projective,
    $\Hom_B(N, S_b)\neq 0$ if and only if $S_b\in \add N$.
    Similarly, $\Hom_A(S_a, M)\neq 0$ if and only if $S_a\in \add M$ as $S_a\in \modn A$ is injective.
    Thus 
    \begin{equation*}
        \Hom_{\Lambda}(\resLambdaB_{\ast}N, \resLambdaA_{\ast}M) = 0
    \end{equation*}
    if $S_a\notin \add M$ or $S_b\notin \add N$.

    On the other hand, 
    \begin{equation*}
        \Hom_A(M, S_a)\cong \Hom_A(S_a^{\ell_M}, S_a) \cong \Bbbk^{\ell_M}
    \end{equation*}
    as $f_M:M\rightarrow S_a^{\ell_M}$ is the minimal left $\add S_a$-approximation of $M$.
    Similarly, we have 
    \begin{equation*}
        \Hom_B(S_b, N) \cong \Hom_B(S_b, S_b^{r_N}) \cong \Bbbk^{r_N}.
    \end{equation*}
    Thus 
    \begin{equation*}
        \Hom_{\Lambda}(\resLambdaA_{\ast}M, \resLambdaB_{\ast}N) \cong \Bbbk^{\ell_M}\otimes_{\Bbbk} \Hom_B(S_b, N) \cong \Hom_A(M, S_a)\otimes_{\Bbbk} \Bbbk^{r_N}.
    \end{equation*}
    In particular, 
    \begin{equation*}
        \Hom_{\Lambda}(\resLambdaA_{\ast}M, \resLambdaB_{\ast}N) = 0
    \end{equation*}
    if $f_M=0$ or $g_N=0$.

\end{proof}

The following Proposition shows that there are no non-trivial extensions between objects in $\resLambdaA_{\ast}(\modn A)$ and objects in $\resLambdaB_{\ast}(\modn B)$,
implying a certain directedness of the sink-source gluing. 
As a consequence, the objects in $\ind \Lambda$ are obtained as the union of objects in $\resLambdaA_{\ast}(\ind A)$ and that of $\resLambdaB_{\ast}(\ind B)$ with $\resLambdaA_{\ast}(S_a)$ and $\resLambdaB_{\ast}(S_b)$ being identifies with the simple bridge $S$.

\begin{Prop}\label{prop_ext_Amod_Bmod_is_zero}
    Let $\Lambda = A\Delta_S B$. 
    Let $M\in \modn A$ and $N\in \modn B$.
    We have 
    \begin{equation*}
        \Ext_{\Lambda}^i(\resLambdaA_{\ast}M, \resLambdaB_{\ast}N) = 0, \text{ for all } i\geq 1.
    \end{equation*}
    If moreover we assume that $\Omega^j_B(N)$ is indecomposable non-projective for all $0< j<k$ for some integer $k>0$,
    then 
    \begin{equation*}
        \Ext_{\Lambda}^i(\resLambdaB_{\ast}N, \resLambdaA_{\ast}M) = 0, \text{ for all } 1\leq j\leq k.
    \end{equation*}
\end{Prop}
\begin{proof}
    Firstly we claim that $S_a$ does not appear as a composition factor of $\rad X$ for any $X\in \modn A$.
    If $S_a$ appears as a submodule of $\rad X/X'$,
    then it is a submodule of $X/X'$ which is a direct summand.
    This gives a split epimorphism $\rad X\twoheadrightarrow S_a$, 
    which factors through $X$,
    a contradiction.
    
    Let $P_M^{\bullet}\rightarrow M$ be the minimal projective resolution of $M\in \modn A$.
    By Proposition \ref{prop_image_of_projects_lambda_A_B} (b) together with the fact that $\resLambdaA_{\ast}$ is exact, 
    $\resLambdaA_{\ast}P_M^{\bullet}\rightarrow \resLambdaA_{\ast}M$ is the minimal projective resolution of $\resLambdaA_{\ast}M\in \modn \Lambda$.

    Following the claim above,  
    $S_a$ doesn't appear as a composition factor of $P_M^i$ for $i\neq 0$ as otherwise $S_a$ would appear as a composition factor of $\rad P_M^{i-1}$.
    This implies that $ P_M^i\rightarrow 0$ is the minimal left $\add S_a$-approximation of $P_M^i$ for $i\neq 0$.
    By Proposition \ref{prop_hom_space_of_Lambda} (b), 
    \begin{equation*}
        \Hom_{\Lambda}(\resLambdaA_{\ast}P_M^i, \resLambdaB_{\ast}N) = 0
    \end{equation*}
    for $i\neq 0$.
    Thus $\Ext_{\Lambda}^i(\resLambdaA_{\ast}M, \resLambdaB_{\ast}N) = 0$ for $i\geq 1$.

    Let $Q_N^{\bullet}\rightarrow N$ be the minimal projective resolution of $N\in \modn B$.
    As $\Omega_B^j(N)$ is indecomposable non-projective for $0< j < k$, $\Omega_B^j(N)\ncong S_b$ which implies $S_b\notin \add Q_N^i$ for $0< i\leq k$.
    By Proposition \ref{prop_hom_space_of_Lambda}, 
    $\Hom_{\Lambda}(\resLambdaB_{\ast}Q_N^i, \resLambdaA_{\ast}M) = 0$ for all $1\leq i\leq k$.
    Therefore 
    \begin{equation*}
        \Ext_{\Lambda}^i(\resLambdaB_{\ast}N, \resLambdaA_{\ast}M) = 0, \text{ for all } 1\leq j\leq k.
    \end{equation*}
\end{proof}

\begin{Prop}\label{prop_ind_lambda}
    Let $\Lambda = A\Delta_S B$.
    We have $\ind \Lambda = \ind A \cup \ind B / \langle S_a \sim S_b\rangle$.
\end{Prop}
\begin{proof}
    As $\resLambdaA_{\ast}$ and $\resLambdaB_{\ast}$ are fully faithful embeddings, 
    they preserve indecomposability.
    For $M\in \ind A$ and $N\in \ind B$, 
    $\resLambdaA_{\ast}(M)\cong \resLambdaB_{\ast}(N)$ if and only if $M\cong S_a$ and $N\cong S_b$ by Proposition \ref{prop_hom_space_of_Lambda} (a).
    So we have an injective map $\ind A\cup \ind B/\langle S_a\sim S_b\rangle \hookrightarrow \ind \Lambda$.
    
    To see the surjectivity, consider $L\in \ind \Lambda$ and the following exact sequence.
    \begin{equation*}
        \xymatrix@C=1em@R=1em{
        0\ar[r] & L\langle 1-e_A \rangle \ar[r] & L\ar[r] & L/(L\langle 1-e_A\rangle)\ar[r] & 0.
        }
    \end{equation*}
    We have $ L/(L\langle 1-e_A\rangle)\in \resLambdaA_{\ast}(\modn A)$ as it is supported on $(Q_A)_0$ and $L\langle 1-e_A \rangle\in \resLambdaB_{\ast}(\modn B)$ as it is supported on $(Q_B)_0$. 
    But $\Ext_{\Lambda}^1(L/(L\langle 1-e_A\rangle, L\langle 1-e_A \rangle)=0$ by Proposition \ref{prop_ext_Amod_Bmod_is_zero}. 
    So the above exact sequence splits.
    Thus either $L\cong L/(L\langle 1-e_A\rangle\in \resLambdaA_{\ast}(\ind A)$ or $L\cong L\langle 1-e_A \rangle\in \resLambdaB_{\ast}(\ind B)$ as $L$ is indecomposable.
    
\end{proof}

Now we study the compatibility between the syzygy functors and the embeddings $\resLambdaA_{\ast}, \resLambdaB_{\ast}$.
Moreover, we show that $\resLambdaA_{\ast}$ and $\resLambdaB_{\ast}$ are indeed homological embeddings. 
Recall an exact functor $\iota: \mathcal{A}\rightarrow \mathcal{B}$ between abelian categories is called a homological embedding if 
\begin{equation*}
    \Ext_{\mathcal{A}}^k(X, Y) \xrightarrow[]{\sim} \Ext_{\mathcal{B}}^k(X, Y)
\end{equation*}
for all $k\geq 0$ and $X, Y\in \mathcal{A}$.
We refer to \cite{Chr14, CJ14} for more discussions in this concept.

\begin{Prop}\label{prop_syzygy_A_B_syzygy_lambda_and_homoligical_embeddings}
    Let $\Lambda = A\Delta_S B$.
    Let $M\in \modn A$ and $N\in \modn B$.
    \begin{itemize}
        \item [(a)] 
        We have 
        \begin{equation*}
        \Omega_{\Lambda}^i(\resLambdaA_{\ast}M) \cong \resLambdaA_{\ast}\Omega_A^i(M), \text{ for all } i>0.
        \end{equation*}
        \item [(b)] Write $\Omega_B^i(N) = S_b^{t_i}\oplus N_i$ such that $S_b\notin \add N_i$ and $t_i\geq 0$.
        We have 
        \begin{equation*}
            \Omega_{\Lambda}^{i+1}(\resLambdaB_{\ast} N) = \resLambdaA_{\ast}\Omega_A(S_a)^{t_i} \oplus \resLambdaB_{\ast}\Omega_B(N_i).
        \end{equation*}
        In particular, if $\Omega_B^i(N)$ is indecomposable non-projective for $0<i<k$, 
        then 
        \begin{equation*}
            \Omega_{\Lambda}^j(\resLambdaB_{\ast}N) \cong \resLambdaB_{\ast} \Omega_B^j(N), \text{ for all } 1\leq j\leq k. 
        \end{equation*}
        \item [(c)] 
        The functors  
        $^\Lambda_A\pi_{\ast} : \modn A\rightarrow \modn \Lambda$ and $\resLambdaB_{\ast}: \modn B\rightarrow \modn \Lambda$ are homological embeddings.
    \end{itemize}
\end{Prop}

\begin{proof}
    Let $P_M^{\bullet}\rightarrow M$ be the minimal projective resolution of $M\in \modn A$.
    As $\resLambdaA_{\ast}$ is exact and preserves projective objects by Proposition \ref{prop_image_of_projects_lambda_A_B} (b), 
    $\resLambdaA_{\ast}P_M^{\bullet}\rightarrow \resLambdaA_{\ast}M$ is the minimal projective resolution of $M$.  
    So (a) follows.

    Let $M'\in \modn A$. For a similar reason together with the fact that $\resLambdaA_{\ast}$ is fully faithful,
    we have 
    \begin{equation*}
        \Ext_A^k(M, M') \cong \Ext_{\Lambda}^k(\resLambdaA_{\ast}M, \resLambdaA_{\ast}M'), \forall i\geq 1,
    \end{equation*}
    which implies that $\resLambdaA_{\ast}$ is a homological embedding.
    
    Now we prove (b).
    Note that $\Omega_{\Lambda}(\resLambdaB_{\ast}N)\cong \resLambdaB_{\ast}\Omega_B(N)$ if $S_b\notin \add N$ by Proposition \ref{prop_image_of_projects_lambda_A_B} (b) together with the fact that $\resLambdaB_{\ast}$ is an exact fully faithful embedding.  
    If moreover we assume $\Omega^i_B(N)$ is indecomposable non-projective for $0<i<k$, then $S_b\notin \add \Omega_B^i(N)$ for all $0<i<k$.
    By induction, we have 
    \begin{equation*}
        \Omega_{\Lambda}^j(\resLambdaB_{\ast}N)\cong \resLambdaB_{\ast}\Omega_B^j(N), \text{ for all } 1\leq j\leq k.
    \end{equation*}

    In general, we may write 
    \begin{equation*}
        \Omega^i_B(N) = S_b^{t_i}\oplus N_i
    \end{equation*}
    for some $t_i\geq 0$ such that $S_b\notin \add N_i$.
    Thus
    \begin{align*}
        \Omega_{\Lambda}^{i+1}(\resLambdaB_{\ast} N) & =\Omega_{\Lambda}(\resLambdaB_{\ast}S_b^{t_i}\oplus \resLambdaB_{\ast}N_i)\\
        & = \Omega_{\Lambda}(S)^{t_i} \oplus 
        \Omega_{\Lambda}(\resLambdaB_{\ast}N_i)\\
        & \cong \resLambdaA_{\ast}\Omega_A(S_a)^{t_i} \oplus \resLambdaB_{\ast} \Omega_B(N_i).
    \end{align*}

    Let $N\rightarrow I_N^{\bullet}$ be the minimal injective resolution of $N\in \modn B$. 
    As $\resLambdaB_{\ast}$ is exact and preserves injective objects by Proposition \ref{prop_image_of_projects_lambda_A_B} (b'),
    we have that $\resLambdaB_{\ast}N\rightarrow \resLambdaB_{\ast}I_N^{\bullet}$ is the minimal injective resolution of $\resLambdaB_{\ast}N\in \modn \Lambda$.
    Moreover $\resLambdaB_{\ast}$ is fully faithful, for $N'\in \modn B$ we have 
    \begin{equation*}
        \Ext_B^k(N', N) \cong \Ext_{\Lambda}^k(\resLambdaB_{\ast}N', \resLambdaB_{\ast}N), \forall i\geq 1,
    \end{equation*}
    which implies that $\resLambdaB_{\ast}$ is a homological embedding. 
\end{proof}

From now on, we identify $\modn A$ with the subcategory $\resLambdaA_{\ast}(\modn A)$ of $\modn \Lambda$.
Similarly, we identify $\modn B$ with $\resLambdaB_{\ast}(\modn B)$.

We show in the following Proposition that the embeddings $\resLambdaA_{\ast}, \resLambdaB_{\ast}$ are compatible with the functors $\tau_{d,A}^{\pm 1}, \tau_{d,B}^{\pm 1}$ and $\tau_{d, \Lambda}^{\pm 1}$.

\begin{Prop}\label{prop_compatibility_embedding_tau_d}
    Let $\Lambda = A\Delta_S B$. 
    \begin{itemize}
        \item [(a)] 
        Let $M\in \underline{\modn} A$ and $N\in \underline{\modn} B$. 
        Assume $\Omega_A^i(M)$ and $\Omega_B^i(N)$ are indecomposable non-projective for all $0<i<d$.
        Then 
        \begin{equation*}
            \tau_{d,\Lambda}(\resLambdaA_{\ast}M)\cong \resLambdaA_{\ast}\tau_{d,A}M,
        \end{equation*}
        and 
        \begin{equation*}
            \tau_{d,\Lambda}(\resLambdaB_{\ast}N)\cong \resLambdaB_{\ast}\tau_{d,B}N.
        \end{equation*}
        \item [(b)] Dually, let $M\in \overline{\modn } A$ and $N\in \overline{\modn } B$.
        Assume that $\Omega_A^{-i}(M)$ and $\Omega_B^{-i}(N)$ are indecomposable non-injective for all $0<i<d$.
        Then 
        \begin{equation*}
            \tau_{d,\Lambda}^{-1}(\resLambdaA_{\ast}M) \cong \resLambdaA_{\ast}\tau_{d, A}M,
        \end{equation*}
        and 
        \begin{equation*}
            \tau_{d,\Lambda}^{-1}(\resLambdaB_{\ast}N) \cong \resLambdaB_{\ast}\tau_{d, B}N.
        \end{equation*}
    \end{itemize}
    
\end{Prop}
\begin{proof}
    We only prove (a) as the proof for (b) is similar.
    Let 
    \begin{equation*}
        P_M^{\bullet}: 
        \xymatrix@C=1em@R=1em{
        P_M^d\ar[r] & P_M^{d-1}\ar[r] & \cdots\ar[r] & P_M^0\ar[r] & 0
        }
    \end{equation*}
    such that $P_M^{\bullet}\rightarrow M$ is the beginning of the minimal projective resolution of $M\in \modn A$.
    We claim that $e_aA\notin \add P_M^i$ for all $0<i<d$.
    Otherwise $\Omega_A^i(M)\cong S_a$ for some $i$ and $S_a\hookrightarrow P_M^{i-1}$ which implies $S_a\in \add P_M^{i-1}$, a contradiction.
    Thus for $d\geq 1$ and $M\ncong S_a$ or $d\geq 2$,
    we have the following commutative diagram by Corollary \ref{cor_nakayama_functor_lambda_A_B}.
    \begin{equation*}
        \xymatrix{
        \resLambdaA_{\ast} \nu_A P_M^d\ar[r]\ar[d]^{\cong} & \resLambdaA_{\ast} \nu_A P_M^{d-1}\ar[d]^{\cong}\\
        \nu_{\Lambda}\resLambdaA_{\ast}P_M^d\ar[r] & \nu_{\Lambda}\resLambdaA_{\ast}P_M^{d-1}
        }
    \end{equation*}
    So $\tau_{d, \Lambda}(\resLambdaA_{\ast}M) \cong \resLambdaA_{\ast}\tau_{d,A}M$.

    If $M=S_a$ and $d=1$. 
    Consider the minimal projective presentation of $S_a\in \modn A$:
    \begin{equation*}
        \xymatrix@C=1em@R=1em{
        Q\ar[r] & e_aA\ar[r] & S_a\ar[r] & 0.
        }
    \end{equation*}
    We have $\resLambdaA_{\ast}\tau_A(S_a) = \ker(\resLambdaA_{\ast}\nu_AQ\rightarrow S)$. 
    On the other hand, we have 
    \begin{equation*}
        \nu_{\Lambda}\resLambdaA_{\ast}Q\cong \resLambdaA_{\ast}\nu_A Q \text{ and } \nu_{\Lambda}\resLambdaA_{\ast} e_aA \cong \resLambdaB_{\ast}\nu_B(e_bB).
    \end{equation*}
    So $\tau_{\Lambda}\resLambdaA_{\ast}S_a = \ker(\resLambdaA_{\ast}\nu_A Q\xrightarrow {f}\resLambdaB_{\ast}\nu_B(e_bB))$.
    Note that by Proposition \ref{prop_hom_space_of_Lambda} (b), $\image f = S$.
    Thus $\resLambdaA_{\ast}\tau_A(S_a) \cong \tau_{\Lambda}\resLambdaA_{\ast}S_a$.

    Arguing similarly, 
    it can be shown that $\resLambdaB_{\ast}\tau_B(N) \cong \tau_{\Lambda}\resLambdaB_{\ast}N$.
\end{proof}

Now assume $\mathcal{A}\subseteq \modn A$ and $\mathcal{B}\subseteq \modn B$ are subcategories such that $S\in \mathcal{A}$ and $S\in \mathcal{B}$. 
We introduce the gluing of $\mathcal{A}$ and $\mathcal{B}$ in $\modn \Lambda$.
A priori the glued category is only an additive subcategory of $\modn \Lambda$.
We will later on restrict to the case when $\mathcal{A}\subseteq \modn A$ and $\mathcal{B}\subseteq \modn B$ are $d$-cluster tilting subcategories so that the glued category is also $d$-cluster tilting. 
This may be viewed as a higher dimensional version of   \cite[Definition 3.1]{Vas21}.

\begin{Def}
    We define the gluing of $\mathcal{A}$ and $\mathcal{B}$ as $\mathcal{C} = \mathcal{A}\Delta_S \mathcal{B} = \add\{\mathcal{A}, \mathcal{B}\}\subseteq \modn \Lambda$.
\end{Def}

\begin{Prop}\label{prop_glue_d_ct_2_d_ct}
    Let $\Lambda = A\Delta_S B$.
    Assume 
    $\mathcal{M}_A = \add
    M_A\subseteq \modn A$ and 
    $\mathcal{M}_B = \add M_B\subseteq \modn B$ are $d$-cluster tilting subcategories. 
    Then  $\mathcal{M} = \mathcal{M}_A\Delta_S \mathcal{M}_B\subseteq \modn \Lambda$ is a $d$-cluster tilting subcategory.
    If moreover, $\mathcal{M}_A$ and $\mathcal{M}_B$ are $d\mathbb{Z}$-cluster tilting, then $\mathcal{M}$ is also $d\mathbb{Z}$-cluster tilting.
\end{Prop}

\begin{proof}
    Firstly note that $\Lambda, D\Lambda\in \mathcal{M}$ by Proposition \ref{prop_image_of_projects_lambda_A_B}.
    
    As $\resLambdaA_{\ast}$ and $\resLambdaB_{\ast}$ are homological embeddings by Proposition \ref{prop_syzygy_A_B_syzygy_lambda_and_homoligical_embeddings} (c), 
    \begin{equation*}
        \Ext_{\Lambda}^i(\mathcal{M}_A, \mathcal{M}_A) = 0 
        \text{ and } 
        \Ext_{\Lambda}^i(\mathcal{M}_B, \mathcal{M}_B) = 0, 
        \text{ for } 0<i<d.
    \end{equation*}

    By Proposition \ref{prop_ext_Amod_Bmod_is_zero},  $\Ext_{\Lambda}^i(\mathcal{M}_A, \mathcal{M}_B) = 0$ for $0<i<d$.

    Let $N\in (\mathcal{M}_B)_P$. 
    Then $\Omega^i_B(N)$ is indecomposable non-projective by Proposition \ref{prop_indec_equiv} (c) for $0<i<d$,  thus 
    \begin{equation*}
        \Ext_{\Lambda}^i(N, \mathcal{M}_A) = 0, \text{ for all } 0<i<d,
    \end{equation*}
    by Proposition \ref{prop_ext_Amod_Bmod_is_zero}.
    This implies that 
    $\Ext_{\Lambda}^i(\mathcal{M}_B, \mathcal{M}_A)=0$ for $0<i<d$.

    Thus $\Ext_{\Lambda}^i(\mathcal{M}, \mathcal{M})=0$ for $0<i<d$.

    Next we show that for any $M\in \ind  \mathcal{M}$, there exists an exact sequence 
    \begin{align*}
        \xymatrix@C=1em@R=1em{
        0\ar[r] & M_{d+1}\ar[r] & \cdots\ar[r] & M_1\ar[r] & M
        } & (\dag)
    \end{align*}
    such that 
    \begin{align*}
        \xymatrix@C=1em@R=1em{
        0\ar[r] & \Hom_{\Lambda}(-,M_{d+1})|_{\mathcal{M}}\ar[r] & \cdots\ar[r] & \Hom_{\Lambda}(-,M_{1})|_{\mathcal{M}}\ar[r] & 
        \mathcal{J}_{\Lambda}(-,M)|_{\mathcal{M}}\ar[r] & 0
        } && (\ddag)
    \end{align*}
    is exact.
    Notice that $\mathcal{M}$ admits an additive generator.
    So by applying Lemma \ref{lem_characterisation_of_d_ct_module}, we conclude that $\mathcal{M}$ is $d$-cluster tilting.

    Assume $M\in \ind \mathcal{M}_A$. 
    We take $(\dag)$ in $ \mathcal{M}_A$ to be the $d$-almost split sequence that ends with $M$.
    It suffices to show 
    \begin{align*}
        \xymatrix@C=1em@R=1em{
        0\ar[r] & \Hom_{\Lambda}(N,M_{d+1})\ar[r] & \cdots\ar[r] & \Hom_{\Lambda}(N,M_{1})\ar[r] & 
        \mathcal{J}_{\Lambda}(N,M)\ar[r] & 0
        } 
    \end{align*}
    is exact for $N\in \ind \mathcal{M}_B\backslash \ind \mathcal{M}_A$.
    This is trivial by Proposition \ref{prop_hom_space_of_Lambda} (a).

    Assume $M\in \ind \mathcal{M}_B$ and take $(\dag)$ in $\mathcal{M}_B$ to be the $d$-almost split sequence that ends with $M$. 
    We need to show that 
    \begin{align*}
        \xymatrix@C=1em@R=1em{
        0\ar[r] & \Hom_{\Lambda}(L,M_{d+1})\ar[r] & \cdots\ar[r] & \Hom_{\Lambda}(L,M_{1})\ar[r] & 
        \mathcal{J}_{\Lambda}(L,M)\ar[r] & 0
        } 
    \end{align*}
    is exact for $L\in \ind \mathcal{M}_A$.

    It holds true if $L\cong S$ since $S \cong S_b\in \mathcal{M}_B$.
    Otherwise, let $f_L: L\rightarrow S^{\ell_L}$ be the minimal left $\add S$-approximation. 
    The statement follows as  
    \begin{equation*}
        \Hom_{\Lambda}(L, M_j) \cong \Hom_{\Lambda}(S^{\ell_L}, M_j)
    \end{equation*}
    for $0\leq j\leq d+1$ where $M_0=M$. 

    Therefore $\mathcal{M}\subseteq \modn \Lambda$ is $d$-cluster tilting.

    If additionally $\mathcal{M}_A\subseteq \modn A$ and $\mathcal{M}_B\subseteq \modn B$ are $d\mathbb{Z}$-cluster tilting.
    We have 
    \begin{equation*}
        \Omega_{\Lambda}^d(\mathcal{M}_A) = \Omega_A^d(\mathcal{M}_A)\subseteq \mathcal{M}_A
    \end{equation*}
    by Proposition \ref{prop_syzygy_A_B_syzygy_lambda_and_homoligical_embeddings} (a).

    Moreover as we have seen before, for $N\in (\mathcal{M}_B)_P$ and $1\leq i\leq d-1$,
    \begin{equation*}
        \Omega_{\Lambda}^i(N) \cong \Omega_B^i(N) \text{ and } 
        \Omega^{d-1}_B(N)\ncong S_b.
    \end{equation*}
    Thus $\Omega^d_{\Lambda}(N)\cong \Omega_B^d(N)\in \mathcal{M}_B \subseteq \mathcal{M}$ by Proposition \ref{prop_syzygy_A_B_syzygy_lambda_and_homoligical_embeddings} (b).
    So we conclude that $\mathcal{M}$ is $d\mathbb{Z}$-cluster tilting.

\end{proof}

\begin{Prop}\label{prop_d_ct_of_lambda_obtained_from_gluing}
    Let $\Lambda = A\Delta_S B$. 
    Suppose $\mathcal{M}\subseteq \modn \Lambda$ is a $d$-cluster tilting subcategory and $S\in \mathcal{M}$. 
    We have that $\mathcal{M} = \mathcal{M}_A\Delta \mathcal{M}_B$ where $\mathcal{M}_A = \resLambdaA^{\ast}\mathcal{M} \subseteq \modn A$ 
    and $\mathcal{M}_B = ^{\Lambda}_B\pi^{\ast}\mathcal{M} \subseteq \modn B$
    are $d$-cluster tilting subcategories.
     If additionally $\mathcal{M}\subseteq \modn \Lambda$ is $d\mathbb{Z}$-cluster tilting, then $\mathcal{M}_A\subseteq \modn A$ and $\mathcal{M}_B\subseteq \modn B$ are $d\mathbb{Z}$-cluster tilting.
\end{Prop}

\begin{proof}
    We prove $\mathcal{M}_A\subseteq \modn A$ is $d$-cluster tilting. 

    Since $\resLambdaA_{\ast}$ is a homological embedding by Proposition \ref{prop_syzygy_A_B_syzygy_lambda_and_homoligical_embeddings} (c), 
    \begin{equation*}
        \Ext_A^i(M, N) \cong \Ext_{\Lambda}^i(M, N)
    \end{equation*}
    for $M, N\in \mathcal{M}_A$.
    Thus $\Ext_A^i(\mathcal{M}_A, \mathcal{M}_A)=0$ for $0<i<d$.

    Let $X\in \ind A$ and $Y\in \ind \mathcal{M}\backslash \ind\mathcal{M}_A$.
    
    Suppose $\Ext_A^i(X, \mathcal{M}_A)=0$ for $0<i<d$.
    Then 
    \begin{equation*}
        \Ext_{\Lambda}^i(X, \mathcal{M}_A) = 0, \text{ for all } 0<i<d.
    \end{equation*}
    We have 
    \begin{equation*}
        \Ext_{\Lambda}^i(X, Y)=0, \text{ for all }i>0
    \end{equation*}
    by Proposition \ref{prop_ext_Amod_Bmod_is_zero}.
    Since $\mathcal{M}\subseteq \modn \Lambda$ is $d$-cluster tilting, 
    we have $X\in \mathcal{M}$ and thus 
    $X\in \mathcal{M}_A$.

    On the other hand, assume $\Ext_A^i( \mathcal{M}_A, X)=0$ for $0<i<d$. 
    We claim that 
    \begin{equation*}
        \Ext_{\Lambda}^i(Y, X) = 0, \text{ for all } 0<i<d.
    \end{equation*}
    This is true if $X\cong S_a$ as $S\in \mathcal{M}$.
    Otherwise,
    let $P_Y^{\bullet}\rightarrow Y$ be the minimal projective resolution of $Y\in \modn B$.
    By Proposition \ref{prop_hom_space_of_Lambda} (a), We have $\Hom_{\Lambda}(P_Y^i, X)=0$.
    Thus our claim follows.
    We have
    \begin{equation*}
        \Ext_{\Lambda}^i(\mathcal{M}, X)=0, \text{ for all }  0<i<d.
    \end{equation*}
    Thus $X\in \mathcal{M}$ which implies $X\in \mathcal{M}_A$.
    Therefore $\mathcal{M}_A\subseteq \modn A$ is $d$-cluster tilting.

    If additionally $\mathcal{M}\subseteq \modn \Lambda$ is $d\mathbb{Z}$-cluster tilting, 
    we have $\Omega_{\Lambda}^d(\mathcal{M})\subseteq \mathcal{M}$.
    By Proposition \ref{prop_syzygy_A_B_syzygy_lambda_and_homoligical_embeddings} (a), for $X\in \mathcal{M}_A$, we have 
    \begin{equation*}
        \Omega_{A}^d(X) \cong \Omega_{\Lambda}^d(X)\subseteq \mathcal{M} \cap \modn A = \mathcal{M}_A.
    \end{equation*}
    Thus $\mathcal{M}_A\subseteq \modn A$ is $d\mathbb{Z}$-cluster tilting. 

    Arguing similarly we can prove the statement for $\mathcal{M}_B$.

    Lastly as $\mathcal{M}_A\Delta \mathcal{M}_B$ and $\mathcal{M}$ are both $d$-cluster tilting subcategories of $\modn \Lambda$ and $\mathcal{M}_A\Delta \mathcal{M}_B\subseteq \mathcal{M}$,
    we conclude that $\mathcal{M} = \mathcal{M}_A\Delta \mathcal{M}_B$.
\end{proof}

Let  $\Lambda = A\Delta_S B$. 
Denote by 
\begin{equation*}
    d\text{-}CT(A) = \{\mathcal{M}_A=\add M_A\subseteq \modn A : d\text{-cluster tilting}\}
\end{equation*}
the set of $d$-cluster tilting subcategories of $\modn A$ which admit an additive generator. 
Respectively we use $d\text{-}CT(B)$ for $B$.
Let 
\begin{equation*}
    d\text{-}CT_S(\Lambda) = \{\mathcal{M}=\add M\subseteq \modn \Lambda : d\text{-cluster tilting and } S\in \mathcal{M}\}.
\end{equation*}
Combining Proposition \ref{prop_glue_d_ct_2_d_ct} and Proposition \ref{prop_d_ct_of_lambda_obtained_from_gluing},
we obtain the following Corollary.

\begin{Cor}\label{cor_bijection_d_ct_gluing}
    Let $\Lambda = A\Delta_S B$. 
    Then there is a bijection 
    \begin{align*}
        d\text{-}CT(A)\times d\text{-}CT(B) & \longleftrightarrow d\text{-}CT_S(\Lambda)\\
        (\mathcal{M}_A, \mathcal{M}_B) & \longmapsto \mathcal{M}_A\Delta \mathcal{M}_B \\
        (\mathcal{M}\cap \modn A, \mathcal{M}\cap \modn B) & \longmapsfrom \mathcal{M}
    \end{align*}
    Moreover, replacing $d$-CT by $d\mathbb{Z}$-CT, 
    the above bijiection restricts.
\end{Cor}

\begin{Prop}\label{prop_glue_parial_d_ct_obtain_partial_d_ct}
    Let $\Lambda = A\Delta_S B$.
    Suppose $\mathcal{C}^A\subseteq\modn A$ and $\mathcal{C}^B\subseteq\modn B$ are partial $d$-cluster tilting subcategories. 
    Then $\mathcal{C} = \mathcal{C}^A\Delta_S\mathcal{C}^B \subseteq \modn \Lambda$ is partial $d$-cluster tilting. 
    If moreover $\mathcal{C}^A$ and $\mathcal{C}^B$ are partial $d\mathbb{Z}$-cluster tilting, then $\mathcal{C}$ is also partial $d\mathbb{Z}$-cluster tilting.
\end{Prop}

\begin{proof}
    Firstly 
    \begin{align*}
        \add \Lambda & = \add \{A, B/S_b\} \subseteq \mathcal{C}\\
        \add D\Lambda & = \add \{DA/S_a, DB\} \subseteq \mathcal{C}
    \end{align*}
    by Proposition \ref{prop_image_of_projects_lambda_A_B}.

    We have 
    \begin{equation*}
        \Ext_{\Lambda}^i(\mathcal{C}, \mathcal{C}) = 0, \text{ for all } 0<i<d
    \end{equation*}
    as $\resLambdaA_{\ast}$ and $^{\Lambda}_B\pi_{\ast}$ are homological embeddings.

    Let $X\in \mathcal{C}_P$.
    If $X\in  \mathcal{C}^A_P$,
    then $\Omega_{\Lambda}^iX\cong \Omega_A^iX$ by Proposition \ref{prop_syzygy_A_B_syzygy_lambda_and_homoligical_embeddings} (a).
    Thus $\Omega_{\Lambda}^iX$ is indecomposable for $0<i<d$.

    If $X\in \mathcal{C}^B_P$, then $\Omega_B^i(X)$ is indecomposable non-projective for $0<i<d$.
    Thus 
    \begin{equation*}
        \Omega_{\Lambda}^iX\cong \Omega_B^iX, \text{ for all } 0<i<d
    \end{equation*}
    by Proposition \ref{prop_syzygy_A_B_syzygy_lambda_and_homoligical_embeddings} (b) which implies that $\Omega_{\Lambda}^iX$ is indecomposable for $0<i<d$.

    Notice that 
    \begin{equation*}
        \mathcal{C}_P = \mathcal{C}^A_P\cup \mathcal{C}^B_P
        \text{ and } 
        \mathcal{C}_I = 
        \mathcal{C}^A_I\cup \mathcal{C}^B_I.
    \end{equation*}
    By Proposition \ref{prop_compatibility_embedding_tau_d},
    we have the inverse bijections 
    \begin{equation*}
            \xymatrix{\mathcal{C}_P\ar@/^/[r]^{\tau_{d}} & \mathcal{C}_I.\ar@/^/[l]^{\tau_{d}^{-1}}\\}
        \end{equation*}
    Therefore $\mathcal{C}$ is a partial $d$-cluster tilting.

    If additionally $\mathcal{C}^A$ and $\mathcal{C}^B$ are partial $d\mathbb{Z}$-cluster tilting, 
    $\mathcal{C}$ is also partial $d\mathbb{Z}$-cluster tilting by Proposition \ref{prop_syzygy_A_B_syzygy_lambda_and_homoligical_embeddings}.
\end{proof}

\begin{Prop}\label{prop_partial_d_ct_of_lambda_obtained_from_gluing}
    Let $\Lambda=A\Delta_S B$.
    Assume $\mathcal{C}\subseteq \modn \Lambda$ is partial $d$-cluster tilting such that $S\in \mathcal{C}$.
    Then $\mathcal{C}^A = \mathcal{C}\cap \modn A$ and 
    $\mathcal{C}^B = \mathcal{C}\cap \modn B$ are partial $d$-cluster tilting subcategories.
    If moreover $\mathcal{C}$ is partial $d\mathbb{Z}$-cluster tilting, 
    then $\mathcal{C}^A$ and $\mathcal{C}^B$ are also partial $d\mathbb{Z}$-cluster tilting.
\end{Prop}

\begin{proof}
    We show $\mathcal{C}^A$ is partial $d$-cluster tilting.
    Firstly 
    \begin{align*}
        \add A & = \add \Lambda \cap \modn A  \subseteq \mathcal{C}^A \\
        \add DA & = \add ( (\add D\Lambda \cap \modn A) \cup \{S\} ) \subseteq \mathcal{C}^A.
    \end{align*}

    Next we have 
    \begin{equation*}
        \Ext_A^i(\mathcal{C}^A, \mathcal{C}^A) = 0, \text{ for all } 0<i<d, 
    \end{equation*}
    as $\resLambdaA_{\ast}$ is a homological embedding.

    Let $M\in \mathcal{C}^A_P$. 
    We have that $\Omega^i_A(M)$ is indecomposable for $0<i<d$ as $\Omega^i_A(M)\cong \Omega_{\Lambda}^i(M)$ by Proposition \ref{prop_syzygy_A_B_syzygy_lambda_and_homoligical_embeddings} (a).

    As $\tau_d^{\pm 1}(\modn A) \subseteq \modn A$, 
    we have the following inverse bijections.
    \begin{equation*}
            \xymatrix{\mathcal{C}^A_P\ar@/^/[r]^{\tau_{d}} & \mathcal{C}^A_I.\ar@/^/[l]^{\tau_{d}^{-1}}\\}
    \end{equation*}

    So $\mathcal{C}^A$ is partial $d$-cluster tilting.
    Arguing similarly, we have $\mathcal{C}^B$ is partial $d$-cluster tilting.

    If additionally $\mathcal{C}$ is partial $d\mathbb{Z}$-cluster tilting, 
    then $\mathcal{C}^A$ and $\mathcal{C}^B$ are also partial $d\mathbb{Z}$-cluster tilting by Proposition \ref{prop_syzygy_A_B_syzygy_lambda_and_homoligical_embeddings}.
\end{proof}

\begin{Cor}\label{cor_bijection_partial_d_ct_gluing}
    Let $\Lambda = A\Delta_S B$. 
    We have the following bijection.
     \begin{align*}
        \{(\mathcal{C}^A, \mathcal{C}^B)\mid \mathcal{C}^A\subseteq \modn A, \mathcal{C}^B\subseteq \modn B : \text{partial }
        d\text{-CT}\} & \longleftrightarrow \{\mathcal{C}\subseteq \modn \Lambda : \text{partial } d\text{-CT and } S\in \mathcal{C}\}\\
        (\mathcal{C}^A, \mathcal{C}^B) & \longmapsto \mathcal{C}^A\Delta \mathcal{C}^B \\
        (\mathcal{C}\cap \modn A, \mathcal{C}\cap \modn B) & \longmapsfrom \mathcal{C}
    \end{align*}
    Moreover, replacing partial  $d$-CT by partial $d\mathbb{Z}$-CT,
    the above bijection restricts.
\end{Cor}

\subsection{The orbit construction}\label{preliminary_orbit_construction}

We recall background on orbit categories following \cite{DI20}.

Let $\mathcal{A}$ be a locally bounded  $\Bbbk$-linear Krull-Schmidt category and let $G$ be a group.
A $\Bbbk$-linear $G$-action on the category $\mathcal{A}$ is an assignment $g\mapsto F_g$ of a $\Bbbk$-linear automorphism $F_g:\mathcal{A}\rightarrow \mathcal{A}$ such that $F_g\circ F_h = F_{gh}$ for all $g,h\in G$.
To simplify notations, we write $g(x):= F_g(x)$ for $g\in G$ and $x\in \mathcal{A}$.
A $G$-action is called admissible if $g(x)\ncong x$ for all $x\in \ind{\mathcal{A}}$ and $g\in G\backslash \{1\}$.

Let $\mathcal{A}$ be a locally bounded  $\Bbbk$-linear Krull-Schmidt category and $G$ be a group acting admissibly on $\mathcal{A}$.
The orbit category $\mathcal{A}/G$ is also locally bounded $\Bbbk$-linear Krull-Schmidt which is given by the following data.
\begin{itemize}
    \item [$\bullet$] The objects of $\mathcal{A}/G$ are the objects of $\mathcal{A}$.
    \item [$\bullet$] For $x, y\in \mathcal{A}/G$, 
    \begin{equation*}
        \Hom_{\mathcal{A}/G}(x,y)=\bigoplus_{g\in G}\Hom_{\mathcal{A}}(x, g(y)).
    \end{equation*}
    \item [$\bullet$] For $(a_g)_{g\in G}\in \Hom_{\mathcal{A}/G}(x, y)$ and $(b_g)_{g\in G}\in \Hom_{\mathcal{A}/G}(y,z)$, set 
    \begin{equation*}
        (ba)_g=\sum_{h\in G}h(b_{h^{-1}g})a_h: x\rightarrow g(z),
    \end{equation*}
    so $ba\in \Hom_{\mathcal{A}/G}(x, z)$.
\end{itemize}

The natural functor $F: \mathcal{A}\rightarrow \mathcal{A}/G$ induces an exact functor 
\begin{equation*}
    F^{\ast}: \Mod \mathcal{A}/G \rightarrow \Mod \mathcal{A}
\end{equation*}
called pull-up.
This functor has a left adjoint 
\begin{equation*}
    F_{\ast}: \Mod \mathcal{A}\rightarrow \Mod \mathcal{A}/G
\end{equation*}
called the push-down, which is also exact.
Moreover, $F_{\ast}$ induces a functor $F_{\ast}: \modn \mathcal{A}\rightarrow \modn \mathcal{A}/G$.

Notice that a $G$-action on $\mathcal{A}$ induces a $G$-action on $\modn \mathcal{A}$ defined by $g(M) = M\circ g$ for $M\in \modn \mathcal{A}$ and $g\in G$.
A subcategory $\mathcal{U}\subseteq \modn \mathcal{A}$ is called $G$-equivariant if $g(\mathcal{U})\subseteq \mathcal{U}$ for all $g\in G$.

We recall one of the main results in \cite{DI20}.

\begin{Thm}\cite[Corollary 5.2]{DI20}\label{thm_DI_orbit_construction}
    Let $\mathcal{A}$ be a locally bounded  $\Bbbk$-linear Krull-Schmidt category and $G$ a free abelian group of finite rank acting admissibly on $\mathcal{A}$.
    If $\mathcal{U}\subseteq \modn \mathcal{A}$ is a locally bounded $G$-equivariant $d$-cluster tilting subcategory, 
    then $F_{\ast}(\mathcal{U})\subseteq \modn \mathcal{A}/G$ is a locally bounded $d$-cluster tilting subcategory.
\end{Thm}

\subsection{Higher Nakayama algebras.} \label{preliminaries_higher_NakAlg}
We recall definitions and basic facts about higher Nakayama algebras constructed by Jasso-K\"{u}lshammer \cite{JK19}.

A tuple of non-negative integers $\ell_{\infty} = (\ldots, \ell_{-1}, \ell_0, \ell_{1}, \ldots)$ is called a Kupisch series if $\ell_i\leq \ell_{i-1}+1$ for all $i\in \mathbb{Z}$. 
\begin{itemize}
    \item [$\bullet$] $\ell_{\infty}$ is $\ell$-bounded if $\ell = \max\{\ell_i\mid i\in\mathbb{Z}\}$.
    \item [$\bullet$] $\ell_{\infty}$ is 
    connected if either 
    \begin{itemize}
        \item [(a)] $\ell_r=1$ for a unique $r\in\mathbb{Z}$ and $\ell_i=0$ for $i<r$ or
        \item [(b)] $\ell_i\geq 2$ for all $i\in\mathbb{Z}$.
    \end{itemize}
    \item [$\bullet$] $\ell_{\infty}$ is of width $m$ if $|\{\ell_i \neq 0\mid i\in \mathbb{Z}\}| = m$.
    \item [$\bullet$] $\ell_{\infty}$ is $m$-periodic if $\ell_i = \ell_{i+m}$ for all $i\in \mathbb{Z}$. 
\end{itemize}

Throughout the article, we always assume that
\begin{itemize}
    \item [(i)] $\ell_{\infty}$ is $\ell$-bounded for some integer $\ell$.
    \item [(ii)] $\ell_{\infty}$ is connected of width $m$ or $\ell_{\infty}$ is connected and  $m$-periodic.
\end{itemize}

We recall the definition of ordered sequences $(os_{\ell_{\infty}}^d, \preccurlyeq)$ from \cite{JK19}.
Define 
\begin{equation*}
    os_{\ell_{\infty}}^d :=\{x=(x_1, x_2, \ldots, x_d)\in \mathbb{Z}^d\mid x_1< \cdots < x_d \text{ and } x_d-x_1+1\leq \ell_{x_d-d+1}+d-1\}
\end{equation*}
to be the ordered sequence determined by $\ell_{\infty}$.
Additionally $os_{\ell_{\infty}}^d$ is endowed 
with the relation $\preccurlyeq$ defined as
\begin{equation*}
    x \preccurlyeq y \Longleftrightarrow x_1 \leq y_1 < x_2\leq y_2 < \cdots < x_d \leq y_d
\end{equation*}
for $x=(x_1, \ldots, x_d), y = (y_1, \ldots, y_d) \in os_{\ell_{\infty}}^d$. 

Now we recall the $\Bbbk$-linear category $\mathcal{A}_{\ell_{\infty}}^d=\Bbbk Q_{\ell_{\infty}}^d/I$ defined by $\ell_{\infty}$ and the integer $d$ via quivers with relations.
Let $\{e_i\mid 1\leq i\leq d\}$ be the standard basis of $\mathbb{Z}^d$.
\begin{itemize}
    \item [$\bullet$] The vertices set of the quiver $Q_{\ell_{\infty}}^d$ is given by  $os_{\ell_{\infty}}^d$.
    \item [$\bullet$] There is an arrow $a_i(x) : x\rightarrow x + e_i$ whenever $x + e_i \in os_{\ell_{\infty}}^d$.
    \item [$\bullet$] The ideal $I$ of the path category $\Bbbk Q_{\ell_{\infty}}^d$ is generated by
    \begin{equation*}
        a_i(x + e_j)a_j(x) - a_j(x + e_i)a_i(x), \text{ with } 1\leq i,j\leq d.
    \end{equation*} 
    By convention, $a_i(x) = 0$ whenever $x$ or $x+e_i$ is not in  $os_{\ell_{\infty}}^d$, hence some of the relations are indeed zero relations.
\end{itemize}

By construction in \cite{JK19}, $\mathcal{A}_{\ell_{\infty}}^d$ has a distinguished $d\mathbb{Z}$-cluster tilting subcategory 
\begin{equation*}
    \mathcal{M}_{\ell_{\infty}}^d = \add \{M(x)\mid x\in os_{\ell_{\infty}}^{d+1}\}. 
\end{equation*}
Here as a representation $M(x)$ assigns $k$ to vertex $z\in os_{\ell_{\infty}}^{d}$ if $(x_1, \ldots, x_d) \preccurlyeq z \preccurlyeq (x_2-1, \ldots, x_{d+1}-1)$ and $0$ otherwise.
All arrows $k\rightarrow k$ act as identity, while other arrows act as zero. 
By convention, $M(x)=0$ if $x\notin os_{\ell_{\infty}}^{d+1}$.
Then
\begin{equation*}
    \Hom_{\mathcal{A}_{\ell_{\infty}}^{d}}(M(x),M(y)) \cong \left\{ \begin{array}{ll}
        kf_{yx} &  x \preccurlyeq y\\
        0 & \text{ otherwise.}
    \end{array} 
    \right.
\end{equation*}
Here $f_{yx}$ is given by $k\xrightarrow{1} k$ at vertices $z$ where $M(x)_z = M(y)_z = k$ and $0$ otherwise.
The composition of morphisms in $\mathcal{M}_{\ell_{\infty}}^{d}$ is completely determined by
\begin{equation*}
    f_{zy}\circ f_{yx} =
    \left\{ \begin{array}{ll}
        f_{zx} &  x \preccurlyeq z\\
        0 & \text{ otherwise.}
    \end{array} 
    \right.
\end{equation*}
We call $f_{yx}\neq 0$ an arrow morphism if there is no $z\in os_{\ell_{\infty}}^{d+1}$ such that $f_{yx}=f_{yz}\circ f_{zx}$.

Now assume $\ell_{\infty}$ is $m$-periodic.
Consider the group 
\begin{equation*}
    G = \langle \sigma=(m,m,\ldots, m)\rangle \subseteq \mathbb{Z}^d
\end{equation*}
and the $G$-action on $\mathcal{A}_{\ell_{\infty}}^{d}$ via $F_{\sigma}: x\mapsto x+\sigma$.
We obtain the orbit category $\mathcal{A}_{\ell_{\infty}}^d / G$.
Notice that $\mathcal{M}_{\ell_{\infty}}^d$ is $G$-equivariant. 
Thus by Theorem \ref{thm_DI_orbit_construction},
\begin{equation*}
    \mathcal{M}_{\ell_{\infty}} / G \subseteq \modn \mathcal{A}_{\ell_{\infty}}^d / G
\end{equation*}
is a $d\mathbb{Z}$-cluster tilting subcategory. 
Moreover, we have 
\begin{equation*}
    \Hom_{\mathcal{A}_{\ell_{\infty}}^d / G} (M(x), M(y)) \cong \bigoplus_{i: x\preccurlyeq F_{\sigma}^i(y)} k f_{F_{\sigma}^i(y)x}.
\end{equation*}

From now on we will instead use the notation $\underline{\ell} = (\ell_1, \ldots, \ell_m)$. 
We say $\underline{\ell}$ is of type $\mathbb{A}_m$ if $\ell_1=1$ and $\underline{\ell}$ is of type $\widetilde{\mathbb{A}}_{m-1}$ if it is $m$-periodic following the notations in \cite{JK19}.
In the latter case, we consider $\underline{\ell}$ up to rotation symmetry.

\begin{Def}
    Let $\ell_{\infty}$ be a Kupisch series with the equivalent notation $\underline{\ell} = (\ell_1, \ldots, \ell_m)$ where  $m$ is a positive integer.
    Then
    \begin{itemize}
        \item [(a)] The $d$-Nakayama algebra of type $\mathbb{A}_m$ is given by 
        \begin{equation*}
            A_{\underline{\ell}}^{d} = \mathcal{A}_{\ell_{\infty}}^d \text{ where } \ell_{\infty} = (\ldots, 0, \ell_1, \ldots, \ell_m, 0, \ldots).
        \end{equation*}
        The distinguished $d\mathbb{Z}$-cluster tilting subcategory of $\modn A_{\underline{\ell}}^d$ is given by 
        \begin{equation*}
            \mathcal{M}_{\underline{\ell}}^d = \mathcal{M}_{\ell_{\infty}}^d.
        \end{equation*}
        \item [(b)] The $d$-Nakayama algebra of type $\widetilde{\mathbb{A}}_{m-1}$ is given by 
        \begin{equation*}
            A_{\underline{\ell}}^d = \mathcal{A}_{\ell_{\infty}}^d / G \text{ where } \ell_{\infty} = (\ldots, \ell_m, \ell_1, \ldots, \ell_m, \ldots).
        \end{equation*}
        The distinguished $d\mathbb{Z}$-cluster tilting subcategory of $\modn A_{\underline{\ell}}^d$ is given by 
        \begin{equation*}
            \mathcal{M}_{\underline{\ell}}^d = \mathcal{M}_{\ell_{\infty}}^d / G.
        \end{equation*}
    \end{itemize}
\end{Def}

\begin{Rem}\label{RemCoordNota}
    \begin{itemize}
        \item [(i)] The definition of $os_{\underline{\ell}}^{d}$ is slightly different from that in \cite[Definition 1.9]{JK19}.
        We add $(1,2,\ldots,d)$ to each of the ordered sequence defined in \cite{JK19} to make it strictly increasing.
        \item [(ii)] By definition $A_{\underline{\ell}}^{d}$ is a locally bounded $k$-linear category.
        By abuse of notation,
        we still call it an algebra.
    \end{itemize}
\end{Rem}

\begin{Nota}
    Let $\underline{\ell}=(\ell_1, \ldots, \ell_n)$ be a Kupisch series. 
    We call $\underline{\ell}$ acyclic homogeneous if $\ell_i=i$ for $1\leq i\leq \ell \leq m$ and $\ell_i=\ell$ for $\ell+1\leq i\leq m$ and denote the corresponding $d$-Nakayama algebra by $A_{\ell,m}^{d}$.
    We call $\underline{\ell}$ cyclic homogeneous if $\ell_i=\ell$ for $1\leq i\leq m$ and denote the corresponding $d$-Nakayama algebra by $\widetilde{A}_{\ell, m-1}^{d}$. In this case, $\widetilde{A}_{\ell, m-1}^{d}$ is self-injective (see \cite[Theorem 4.10]{JK19}).
\end{Nota}

\begin{Rem}
    Let $A=A_{\underline{\ell}}^{d}$ be an acyclic $d$-Nakayama algebra.
    Note that $A$ has a unique source $a=(1,2,\ldots,d)$ and a unique sink $b=(m,\ldots,m+d-1)$.
\end{Rem}

Let  $A=A_{\underline{\ell}}^{d}$ be a $d$-Nakayama algebra and let $\mathcal{M} = \mathcal{M}_{\underline{\ell}}^{d}\subseteq \modn A$ be the distinguished $d\mathbb{Z}$-cluster tilting subcategory.
We summarize basic homological properties of $M(x)\in \mathcal{M}$ in the following proposition.

\begin{Prop}\label{Prop_modcomb}(\cite[Proposition 2.22, Proposition 2.25, Theorem 3.16]{JK19}\cite[Lemma 3.5(b)]{DI20})
Let $x\in os_{\ell_{\infty}}^{d+1}$. The following statements hold true.
    \begin{itemize} 
      \item [(i)] ${M}(x)$ is simple if and only if $x = (i, i+1, \ldots, i + d)$ for some integer $i$.  
      \item [(ii)] ${M}(x)$ is projective if and only if $x_1 = \min \{y \mid (y, x_2, \ldots, x_{d+1}) \in os_{\ell_{\infty}}^{d+1}\}$.
      \item [(iii)] 
      ${M}(x)$ is injective if and only if $x_{d+1} = \max\{y \mid (x_1, \ldots, x_d, y)\in os_{\ell_{\infty}}^{d+1}\}$.
      \item [(iv)] If $x_1 > x_{d+1} -\ell_{x_{d+1}-d} - d + 1$, then there exists an exact sequence
          \begin{equation*}
              \xymatrix@=1em@R=1em{
                0\ar[r] & \Omega^d({M}(x))\ar[r] & P^d\ar[r] & \cdots\ar[r] & P^1\ar[r] & {M}(x)\ar[r] & 0
              }
          \end{equation*}
        with $P^i = {M}(x_{d+1} - \ell_{x_{d+1}-d} - d + 1, x_1, \ldots, x_{i-1}, x_{i+1}, \ldots, x_{d+1})$ for $1\leq i\leq d$ and $\Omega^d({M}(x)) = {M}(x_{d+1} - \ell_{x_{d+1}-d} - d + 1, x_1, \ldots, x_d)$. 
      \item [(v)] We use the notation $\tau_d(x) = (x_1 - 1, x_2 - 1, \ldots, x_{d+1} - 1)$.
      Then $\tau_d({M}(x)) = {M}(\tau_d(x))$ for $M(x)$ not projective. 
      \item [(vi)] If $\ell_{\infty}$ is $m$-periodic, then we have 
      \begin{equation*}
          \dim_{\Bbbk}\Ext^d_A(M(y), M(x)) = |\{x\mid \sigma^ix\preccurlyeq \tau_d(y) \text{ for some } i\}|.
      \end{equation*}
      If $\ell_{\infty}$ is of width $m$,
      then 
      \begin{equation*}
          \Ext^d_A(M(y), M(x)) \cong \left\{
          \begin{array}{cc}
              k & x\preccurlyeq \tau_d(y), \\
              0 & \text{otherwise.}
          \end{array}\right.
      \end{equation*}
  \end{itemize}
\end{Prop}

Define the map
\begin{equation*}
    f: \mathbb{Z} \rightarrow \mathbb{Z}
\end{equation*}
by $f(i) = \min\{x\mid (x, \ldots, i)\in os_{\underline{\ell}}^{d+1}\}$.
Note that if $\underline{\ell}$ is acyclic, then
\begin{equation*}
    f(i) =\max\{ i  - \ell_{i-d} - d + 1, 1\}.
\end{equation*} 
And if $\underline{\ell}$ is cyclic, 
$f(i) = i  - \ell_{i-d} - d + 1$.
Dually we define
\begin{equation*}
    g: \mathbb{Z} \rightarrow \mathbb{Z}
\end{equation*}
by $g(i) = \max\{y \mid (i, \ldots, y)\in os_{\underline{\ell}}^{d+1}\}$.

\begin{Rem}\label{f_periodic} The functions $f$ and $g$ have the following properties.
    \begin{itemize}
        \item [(i)] $M(x)$ is projective if and only if $x_1 = f(x_{d+1})$.
        And $M(x)$ is injective if and only if $x_{d+1} = g(x_1)$.
        \item [(ii)] Let $x = (x_1, \ldots, x_{d+1})\in os_{\underline{\ell}}^{d+1}$. 
        If $f(x_{d+1}) < x_1$, 
        we have that $\Omega^d (M(x)) = M(f(x_{d+1}), x_1, \ldots, x_d)$.
        If $x_{d+1} < g(x_1)$, 
        then $\Omega^{-d}(M(x)) = M(x_2, \ldots, x_{d+1}, g(x_1))$.
        \item [(iii)] In the cyclic  case, we have that
        $f(i+m) = f(i)+m$ and $g(i+m) = g(i)+m$. 
        \item [(iv)] $f$ is non-decreasing. 
        We know that $i - \ell_{i-d} \geq (i - 1) - \ell_{i-d-1}$ since $\ell_{i-d} \leq \ell_{i-d-1} + 1$. 
        This implies $f(i) \geq f(i-1)$.
        \item [(v)] $g$ is non-decreasing since $M(i + 1, x_2, \ldots, x_d, g(i))$ is a quotient module of $M(i, x_2, \ldots, x_d, g(i))$ which implies $g(i+1)\geq g(i)$.
    \end{itemize}
    
\end{Rem}

Now we discuss the sink-source gluing  of acyclic $d$-Nakayama algebras. 
Since a $d$-Nakayama algebra is uniquely determined by its Kupisch series,
we firstly define the gluing of Kupisch series.

\begin{Def}
    Let $\underline{\ell_1}=(1, 2, \ldots, \ell_m)$ and $\underline{\ell_2}=(1, 2, \ldots, \ell_n)$ be two Kupisch series.
    We call $\underline{\ell} =  (1,2,\ldots, \ell_m, 2, \ldots, \ell_n)$ the gluing  of $\underline{\ell_1}$ and $\underline{\ell_2}$ and denote it by $\underline{\ell}=\underline{\ell_1}\Delta \underline{\ell_2}$.
\end{Def}

\begin{Prop}\label{prop_glue_acyclic_nak_and_form_of_simple_bridge}

The sink-source gluing of two acyclic $d$-Nakayama algebras is itself an acyclic $d$-Nakayama algebra.
Specifically,
    \begin{equation*}
        A_{\underline{\ell_1}}^{d} \Delta_S A_{\underline{\ell_2}}^{d} = A_{\underline{\ell_1}\Delta \underline{\ell_2}}^{d}
    \end{equation*}
    and the simple bridge is $S = M(m, m+1, \ldots, m+d)$ where $m=|\underline{\ell_1}|$. 
    Moreover the corresponding  distinguished $d\mathbb{Z}$-cluster tilting subcategories are related by gluing. More precisely,
    \begin{equation*}
        \mathcal{M}_{\underline{\ell_1}}^d \Delta_S \mathcal{M}_{\underline{\ell_2}}^d = \mathcal{M}_{\underline{\ell}}^d.
    \end{equation*}
\end{Prop}

\begin{proof}
    The statements follow from  a straightforward verification on combinatorial data defining related algebras and categories.
\end{proof}

Conversely we introduce the notion of degluing of an acyclic $d$-Nakayama algebra.

\begin{Def}\label{def_deglue_acyclic_d_NakAlg}
    Let $A = A_{\underline{\ell}}^{d}$ be an acyclic $d$-Nakayama algebra where $\ell_m \geq \ell_{m+1} = 2$ and $ \ell_{m+1} < \ell_{m+2}$ for some integer $m\geq 1$.
    Let $\underline{\ell_1} = (1, \ell_2, \ldots, \ell_{m})$ and $\underline{\ell_2} = (1, \ell_{m+1}, \ldots, \ell_n)$.
    Then $A = A_{\underline{\ell_1}}^{d} \Delta_S A_{\underline{\ell_2}}^{d}$. We call this a degluing of $A$ and $(m, \ldots, m+d-1)$ a deglue point of $A$.
\end{Def}

\begin{Eg}
    Let $\underline{\ell}=(1,2,3,3,2,3,4)$.
    Then $\underline{\ell} = \underline{\ell_1}\Delta\underline{\ell_2}$ where $\underline{\ell_1} = (1,2,3,3)$ and $\underline{\ell_2} = (1,2,3,4)$.
    The quivers $Q_{\underline{\ell}}^2$, $Q_{\underline{\ell_1}}^2$ and $Q_{\underline{\ell_2}}^2$ are given below. 
    The deglue point is given by $p=(4,5)$.
    \begin{equation*}
        \begin{xy}
            0;<20pt,0cm>:<20pt,20pt>::
            (0,0) *+{12} ="12",
            (0,1) *+{13} ="13",
            (0,2) *+{14} ="14",
            (2,0) *+{23} ="23",
            (2,1) *+{24} ="24",
            (2,2) *+{25} ="25",
            (4,0) *+{34} ="34",
            (4,1) *+{35} ="35",
            (6,0) *+{45} ="45",
            (6,1) *+{46} ="46",
            (6,2) *+{47} ="47",
            (6,3) *+{48} ="48",
            (8,0) *+{56} ="56",
            (8,1) *+{57} ="57",
            (8,2) *+{58} ="58",
            (10,0) *+{67} ="67",
            (10,1) *+{68} ="68",
            (12,0) *+{78} ="78",
            (6.5,-1) *+{\text{ The quiver } Q_{\underline{\ell}}^2},
            "12", {\ar "13"},
            "13", {\ar "14"},
            "23", {\ar "24"},
            "24", {\ar "25"},
            "34", {\ar "35"},
            "45", {\ar "46"},
            "46", {\ar "47"},
            "47", {\ar "48"},
            "56", {\ar "57"},
            "57", {\ar "58"},
            "67", {\ar "68"},
             "13", {\ar "23"},
             "14", {\ar "24"},
             "24", {\ar "34"},
             "25", {\ar "35"},
             "35", {\ar "45"},
             "46", {\ar "56"},
             "47", {\ar "57"},
             "48", {\ar "58"},
             "57", {\ar "67"},
             "58", {\ar "68"},
             "68", {\ar "78"},
        \end{xy}
    \end{equation*}
    \begin{equation*}
        \begin{xy}
            0;<20pt,0cm>:<20pt,20pt>::
            (0,0) *+{12} ="12",
            (0,1) *+{13} ="13",
            (0,2) *+{14} ="14",
            (2,0) *+{23} ="23",
            (2,1) *+{24} ="24",
            (2,2) *+{25} ="25",
            (4,0) *+{34} ="34",
            (4,1) *+{35} ="35",
            (6,0) *+{45} ="45a",
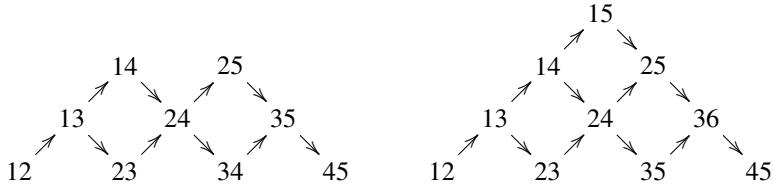
            (4,-1) *+{\text{ The quiver } Q_{\underline{\ell_1}}^2},
            (8,0) *+{12} ="45b",
            (8,1) *+{13} ="46",
            (8,2) *+{14} ="47",
            (8,3) *+{15} ="48",
            (10,0) *+{23} ="56",
            (10,1) *+{24} ="57",
            (10,2) *+{25} ="58",
            (12,0) *+{35} ="67",
            (12,1) *+{36} ="68",
            (14,0) *+{45} ="78",
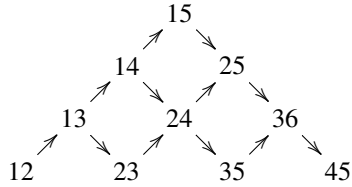
            (12,-1) *+{\text{ The quiver } Q_{\underline{\ell_2}}^2},
            "12", {\ar "13"},
            "13", {\ar "14"},
            "23", {\ar "24"},
            "24", {\ar "25"},
            "34", {\ar "35"},
            "45b", {\ar "46"},
            "46", {\ar "47"},
            "47", {\ar "48"},
            "56", {\ar "57"},
            "57", {\ar "58"},
            "67", {\ar "68"},
             "13", {\ar "23"},
             "14", {\ar "24"},
             "24", {\ar "34"},
             "25", {\ar "35"},
             "35", {\ar "45a"},
             "46", {\ar "56"},
             "47", {\ar "57"},
             "48", {\ar "58"},
             "57", {\ar "67"},
             "58", {\ar "68"},
             "68", {\ar "78"},
        \end{xy}
    \end{equation*}
\end{Eg}

For later use, we also introduce the sink-source self-gluing and self-degluing for acyclic $d$-Nakayama algebras.
Further we compare the corresponding distinguished $d\mathbb{Z}$-cluster tilting subcategories. 

\begin{Def}\label{def_self-glue_and_self-deglue}
    Let $\underline{\ell_1} = (1,2,\ldots, \ell_m)$ be a   Kupisch series. 
    We call $\underline{\ell}^c = (\ell_m, 2,\ldots, \ell_{m-1})$ the self-gluing of $\underline{\ell_1}$.
    Accordingly we call $A^c = A_{\underline{\ell_1}^c}^d$ the self-gluing of $A = A_{\underline{\ell_1}}^d$.
    Conversely if we have $\ell_2 = (\ell_n, 2, \ell_3, \ldots, \ell_{n-1})$ such that $\ell_n\neq  2$,
    then $\underline{\ell_2}^a = (1,2,\ell_3,\ldots,\ell_n)$ is called the self-degluing or $\underline{\ell_2}$.
    Accordingly, $A_2^a = A_{\underline{\ell_2}^a}^d$ is called the self-degluing of $A_2 = A_{\underline{\ell_2}}^d$. And the vertex $(1,2,\ldots,d)$ is called the self-deglue point of $A_2$.
\end{Def}

\begin{Eg}
    Let $\underline{\ell}=(1,2,2,3,4,5,2,2)$.
    Then the self-gluing of $\underline{\ell}$ is given by $\underline{\ell}^c = (2,2,2,3,4,5,2) = (5,2,2,2,2,3,4)$.
    The self-degluing of $\underline{\ell}^c$ is $\underline{\ell'}=(1,2,2,2,2,3,4,5)$.
    Note that $\underline{\ell}\neq \underline{\ell'}$.
\end{Eg}

\begin{Rem}
    
    Let $\underline{\ell}=(2, \ldots, \ell_m, 2, \ldots, \ell_n)$ such that $\ell_m \neq 2 \neq \ell_n$.
    Then
    \begin{equation*}
        \underline{\ell_1} = (1,2,\ldots, \ell_m, 2, \ldots, \ell_n)
    \end{equation*}
    and 
    \begin{equation*}
        \underline{\ell_2} = (1,2,\ldots, \ell_n, 2, \ldots, \ell_m)
    \end{equation*}
    are both self-degluing of $\underline{\ell}$ as $\underline{\ell} = (2, \ldots, \ell_n, 2, \ldots, \ell_m)$ up to rotation symmetry.
    In other words, by our definition, self-degluing of $\underline{\ell}$ are not unique. 
    However once we choose a representative of  $\underline{\ell}$, 
    the self-degluing is indeed unique.
    For this reason, we still refer to the self-degluing of $\underline{\ell}$.
    
\end{Rem}

Let
\begin{equation*}
    \underline{\ell}=(\ell_m, 2,\ell_3, \ldots, \ell_{m-1}),
\end{equation*}
be a Kupisch series such that $\ell_m\neq 2$
and 
\begin{equation*}
    \underline{\ell}^a = (1,2,\ell_3,\ldots,\ell_m)
\end{equation*}
the self-degluing of $\underline{\ell}$.
Denote by $A=A_{\underline{\ell}}^d$ and $A^a = A_{\underline{\ell}^a}^d$ the corresponding $d$-Nakayama algebras. 
We compare the distinguished $d\mathbb{Z}$-cluster tilting subcategories $\mathcal{M}\subseteq \modn A$ and $\mathcal{M}^a\subseteq \modn A^a$.

Observe that $os_{\underline{\ell}^a}^{d+1}\subseteq os_{\underline{\ell}}^{d+1}$.
For $x\in os_{\underline{\ell}}^{d+1}$,
there exists $x'\in os_{\underline{\ell}^a}^{d+1}$ such that $M(x)\cong M(x')$.
We fix the notation  
\begin{equation*}
    \lambda = (1,2,\ldots,d+1) \text{ and } \mu = (m, m+1, \ldots, m+d).
\end{equation*}
Let $x, y\in os_{\underline{\ell}^a}^{d+1}$, 
then $M(x)\cong M(y)$ if and only if $x=y$ or $\{x, y\} = \{\lambda, \mu\}$.
Note that $M^a(\lambda)\ncong M^a(\mu)$.

\begin{Prop}\label{prop_M_and_M^a}
    Let $A$ and $A^a$ be the $d$-Nakayama algebras as above and let $\mathcal{M}$ and $\mathcal{M}^a$ be the corresponding distinguished $d\mathbb{Z}$-cluster tilting subcategories. The following statements hold.
    \begin{itemize}
        \item [(a)] 
        \begin{align*}
            \proj A^a & =\add(\{M^a(x)\mid M(x)\in \proj A \text{ or } x=\lambda\})\\
            \inj A^a & =\add(\{M^a(x)\mid M(x)\in \inj A \text{ or } x=\mu\}).
        \end{align*}
        \item [(b)] For $x\neq \lambda$, 
        \begin{equation*}
            \xymatrix@C=1em@R=1em{
            0\ar[r] & \Omega^d_{A^a}M^a(x)\ar[r] & M^a(x_d)\ar[r] & \cdots\ar[r] & M^a(x_1)\ar[r] & M^a(x)\ar[r] & 0
            }
        \end{equation*}
        is the minimal projective resolution of $M^a(x)$ if 
        \begin{equation*}
            \xymatrix@C=1em@R=1em{
            0\ar[r] & \Omega^d_{A}M(x)\ar[r] & M(x_d)\ar[r] & \cdots\ar[r] & M(x_1)\ar[r] & M(x)\ar[r] & 0
            }
        \end{equation*}
        is the minimal projective resolution of $M(x)$.
        \item [(c)] For $x\neq \lambda$ and $y\neq \mu$, we have 
        \begin{equation*}
            \Ext_{A}^i(M(x), M(y)) = 0, \text{ for all } 0<i<nd, 
        \end{equation*}
        implies that 
        \begin{equation*}
            \Ext_{A^a}^i(M^a(x), M^a(y)) = 0, \text{ for all } 0<i<nd.
        \end{equation*}  
        
        \item [(d)] For $x\neq \lambda$ and $y\neq \mu$, we have 
        \begin{equation*}
            \tau_{nd,A}M(x) \cong M(x') \text{ implies } \tau_{nd,A^a}M^a(x) \cong M^a(x'), 
        \end{equation*}
        and
        \begin{equation*}
            \tau_{nd,A}^{-1}M(y) \cong M(y') \text{ implies } \tau_{nd,A^a}^{-1}M^a(y) \cong M^a(y'),
        \end{equation*}
        for $x', y'\in os_{\underline{\ell}^a}^{d+1}$.
    \end{itemize}
\end{Prop}

\begin{proof}
    Let $x=(x_1, \ldots, x_{d+1})\in os_{\underline{\ell}^a}^{d+1}$.
    We have $M(x)\in \proj A$ if and only if $f(x_{d+1}) = x_1$.
    By comparing $\underline{\ell}$ with $\underline{\ell}^a$,
    we find that the above condition is equivalent to $f^a(x_{d+1}) = x_1$,
    which precisely means that $M^a(x)\in \proj A^a$.
    Moreover, $M^a(\lambda)\in \proj A^a$ thus we obtain $\proj A^a$ as in (a).
    Arguing similarly we get $\inj A^a$ is of the given form.

    To see (b), comparing the minimal projective resolution of $M(x)$ and  $M^a(x)$ given in Proposition \ref{Prop_modcomb} (iv) together with (a),
    the statement follows.
    In particular,
    \begin{equation*}
        \Omega_{A}^dM(x) = M(x') \text{ implies } \Omega_{A^a}^dM^a(x) = M^a(x').
    \end{equation*}

    Now we prove (c).
    As $M^a(x), M^a(y)\in \mathcal{M}^a$,
    it suffices to show
    \begin{equation*}
        \Ext_{A^a}^{kd}(M^a(x), M^a(y)) = 0 \text{ for } 1\leq j\leq n-1.
    \end{equation*}
    By Proposition \ref{Prop_modcomb} (vi), 
    $\Ext_{A^a}^d(M^a(x), M^a(y))\neq 0$ implies $y\preccurlyeq \tau_d(x)$.
    This gives 
    \begin{equation*}
        \Ext_A^d(M(x), M(y))\neq 0,
    \end{equation*}
    a contradiction.

    For $1<k<n$, notice that  
    \begin{equation*}
        \Ext_{A^a}^{kd}(M^a(x), M^a(y)) \cong 
        \Ext_{A^a}^{d}(\Omega_{A^a}^{(k-1)d}M^a(x), M^a(y))
    \end{equation*}
    together with (b),
    a similar argument as above applies and we conclude 
    $\Ext_{A^a}^{kd}(M^a(x), M^a(y)) = 0$.

    Lastly we prove (d). 
    Assume $\Omega_A^{(n-1)d}M(x)\cong M(z)$ then 
    $\Omega_{A^a}^{(n-1)d}M^a(x)\cong M^a(z)$ by (b).
    Thus
    \begin{equation*}
        \tau_{nd,A}M(x) = \tau_{d,A}M(z) \cong M(\tau_{d,A}(z)) = M(x')
    \end{equation*}
    implies 
    \begin{equation*}
        \tau_{nd,A^a}M^a(x) = \tau_{d,A^a}M^a(z) \cong M(\tau_{d,A^a}(z)) = M^a(x').
    \end{equation*}
    Arguing similarly we obtain the statement for $y$.
\end{proof}

\begin{Lem}\label{lem_partial_ndZ_ct_of_A_gives_that_of_A^a}
    Let $A$ and $A^a$ be $d$-Nakayama algebras as above.
    Assume $\modn A$ admits a partial $nd\mathbb{Z}$-cluster tilting subcategory. 
    Then $\modn A^a$ also admits a partial  $nd\mathbb{Z}$-cluster tilting subcategory.
\end{Lem}
\begin{proof}
    Assume $\mathcal{C}\subseteq \modn A$ is partial $nd\mathbb{Z}$-cluster tilting. 
    It follows from Proposition \ref{prop_M_and_M^a} that 
    \begin{equation*}
        \mathcal{C}^a = \add \{M^a(x)\mid M(x)\in \mathcal{C}\cap \mathcal{M} \text{ or } x=\lambda, \mu\}
    \end{equation*}
    is partial $nd\mathbb{Z}$-cluster tilting in $\modn A^a$.
\end{proof}

\begin{Eg}
    Let $\underline{\ell}=(4,2,3,3,2,3)$ be a periodic  Kupisch series.
    Then $\underline{\ell}^a = (1,2,3,3,2,3,4)$. 
    Let $A = A_{\underline{\ell}}^2$ and $A^a = A_{\underline{\ell}^a}^2$ be the correponding $2$-Nakayama algebras.
    The quivers $Q_{\underline{\ell}}^2$, $Q_{\underline{\ell}^a}^2$ are shown below. 
    Notice that the vertices $12$ on the left and on the right in $Q_{\underline{\ell}}^2$ should be identified.
    
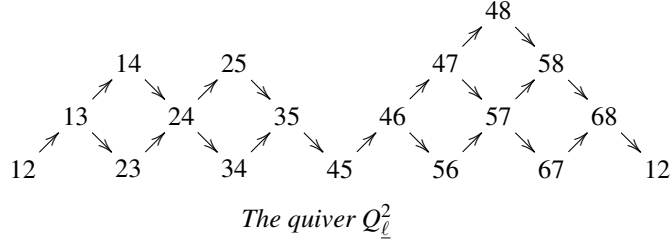
\begin{equation*}
        \begin{xy}
            0;<20pt,0cm>:<20pt,20pt>::
            (0,0) *+{12} ="12",
            (0,1) *+{13} ="13",
            (0,2) *+{14} ="14",
            (2,0) *+{23} ="23",
            (2,1) *+{24} ="24",
            (2,2) *+{25} ="25",
            (4,0) *+{34} ="34",
            (4,1) *+{35} ="35",
            (6,0) *+{45} ="45",
            (6,1) *+{46} ="46",
            (6,2) *+{47} ="47",
            (6,3) *+{48} ="48",
            (8,0) *+{56} ="56",
            (8,1) *+{57} ="57",
            (8,2) *+{58} ="58",
            (10,0) *+{67} ="67",
            (10,1) *+{68} ="68",
            (12,0) *+{12} ="78",
            (6.5,-1) *+{\text{ The quiver } Q_{\underline{\ell}}^2},
            "12", {\ar "13"},
            "13", {\ar "14"},
            "23", {\ar "24"},
            "24", {\ar "25"},
            "34", {\ar "35"},
            "45", {\ar "46"},
            "46", {\ar "47"},
            "47", {\ar "48"},
            "56", {\ar "57"},
            "57", {\ar "58"},
            "67", {\ar "68"},
             "13", {\ar "23"},
             "14", {\ar "24"},
             "24", {\ar "34"},
             "25", {\ar "35"},
             "35", {\ar "45"},
             "46", {\ar "56"},
             "47", {\ar "57"},
             "48", {\ar "58"},
             "57", {\ar "67"},
             "58", {\ar "68"},
             "68", {\ar "78"},
        \end{xy}
    \end{equation*}
    
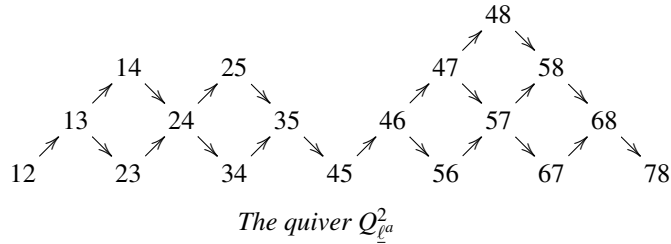
\begin{equation*}
        \begin{xy}
            0;<20pt,0cm>:<20pt,20pt>::
            (0,0) *+{12} ="12",
            (0,1) *+{13} ="13",
            (0,2) *+{14} ="14",
            (2,0) *+{23} ="23",
            (2,1) *+{24} ="24",
            (2,2) *+{25} ="25",
            (4,0) *+{34} ="34",
            (4,1) *+{35} ="35",
            (6,0) *+{45} ="45",
            (6,1) *+{46} ="46",
            (6,2) *+{47} ="47",
            (6,3) *+{48} ="48",
            (8,0) *+{56} ="56",
            (8,1) *+{57} ="57",
            (8,2) *+{58} ="58",
            (10,0) *+{67} ="67",
            (10,1) *+{68} ="68",
            (12,0) *+{78} ="78",
            (6.5,-1) *+{\text{ The quiver } Q_{\underline{\ell}^a}^2},
            "12", {\ar "13"},
            "13", {\ar "14"},
            "23", {\ar "24"},
            "24", {\ar "25"},
            "34", {\ar "35"},
            "45", {\ar "46"},
            "46", {\ar "47"},
            "47", {\ar "48"},
            "56", {\ar "57"},
            "57", {\ar "58"},
            "67", {\ar "68"},
             "13", {\ar "23"},
             "14", {\ar "24"},
             "24", {\ar "34"},
             "25", {\ar "35"},
             "35", {\ar "45"},
             "46", {\ar "56"},
             "47", {\ar "57"},
             "48", {\ar "58"},
             "57", {\ar "67"},
             "58", {\ar "68"},
             "68", {\ar "78"},
        \end{xy}
    \end{equation*}
\end{Eg}

\section{Classification}

In this section, we study the existence and classification of $nd\mathbb{Z}$-cluster tilting subcategories for each $d$-Nakayama algebra $A=A_{\underline{\ell}}^d$ and $n>1$.
In Section \ref{sec_necessary_conditions}, we obtain certain necessary conditions on $\underline{\ell}$ such that there exists a partial $nd\mathbb{Z}$-cluster tilting subcategory for some $n>1$.
In the case when $A$ is non-self-injective, we obtain a classification of its $nd\mathbb{Z}$-cluster tilting subcategories in Section \ref{sec_non-self-inj_case}. 
The classification is completely determined by the arithmetic properties of $\underline{\ell}$.
In fact, except for the case where $\rad^2 A = 0$, there exists at most one $nd\mathbb{Z}$-cluster tilting subcategory for a suitable integer $n$.
In the case when $A = \widetilde{A}_{m-1, \ell}^d$ is self-injective, 
we show that an $nd\mathbb{Z}$-cluster tilting subcategory is only possible if $n|m$ and $n|(\ell-2)$.
In case $n=\ell-2$, we show that such subcategory does exist by constructing an explicit example.

\subsection{Necessary conditions}\label{sec_necessary_conditions}

Let $A=A_{\underline{\ell}}^d$ be a non-self-injective $d$-Nakayama algebra.
In this section, we show that $\modn A$ admits a partial $nd\mathbb{Z}$-cluster tilting subcategory only if $\underline{\ell}$ is piecewise homogeneous, i.e. $\underline{\ell}= \underline{\ell_1} \Delta \cdots \Delta \underline{\ell_t}$ where each $\underline{\ell_i}$ is homogenous for $1\leq i\leq t$, or $\underline{\ell}$ is the self-glue of a piecewise homogenous Kupisch series. 

We show by the following Lemma that if $P\hookrightarrow Q$ for indecomposable $P, Q\in \add A$, 
then more modules have to be included  in any partial $nd\mathbb{Z}$-cluster tilting subcategory $\mathcal{C}\subseteq \modn A$.
In particular we show $\soc P\in \mathcal{C}$. 
Such simple modules will play the role of simple bridges in the degluing procedure later on.
A dual statement for injective modules hold true as well.

\begin{Lem}\label{lem_moremodules}
   Assume $\ell\geq 3$.
   Let $\mathcal{C}\subseteq \modn A$ be a partial $nd\mathbb{Z}$-cluster tilting subcategory and $x = (x_1, \ldots, x_{d+1})\in os_{\underline{\ell}}^{d+1}$.
    \begin{itemize}
        \item [(a)] If $f(x_{d+1}) = f(x_{d+1}+1) = x_1$,
        then $M(y)\in \mathcal{C}$ where $y = (x_1, \ldots, x_{d+1}+i)$ with $i = -1, 0, 1$ and $M(z)\in \mathcal{C}$ where $z=(x_1, x_1+1, \ldots, x_d+1)$ for $x_d < x_{d+1}$ whenever $y, z\in os_{\underline{\ell}}^{d+1}$.
        \item [(b)] If $g(x_1 - 1) = g(x_1) = x_{d+1}$,
        then $M(y)\in \mathcal{C}$ where $y = (x_1 + i, \ldots, x_{d+1})$ with $i = -1, 0, 1$ and $M(z)\in \mathcal{C}$ where $z=(x_1, x_1+1, \ldots, x_d+1)$ for $x_d < x_{d+1}$ whenever $y, z\in os_{\underline{\ell}}^{d+1}$.
    \end{itemize}
\end{Lem}
\begin{proof}
    We only prove (a) as (b) is similar.
    Since 
    \begin{equation*}
        f(x_{d+1}+1) = f(x_{d+1}) = x_1,
    \end{equation*}
    we have that $M(x_1, \ldots, x_{d+1}+1)$ and $M(x_1, \ldots, x_{d+1})$ are projective, hence in $\mathcal{C}$.
    For the same reason, $M(x_1, \ldots, x_{d+1})$ is non-injective, so $M(x_1, \ldots, x_{d+1})\in \mathcal{C}_I\cap \mathcal{M}_I$.
    By Proposition \ref{prop_more_modules_from_tau_d_Omega_-1} (b),
    \begin{equation*}
        \Omega^d\tau_d^{-1}M(x_1, \ldots,x_d,  x_{d+1}) \cong \Omega^dM(x_1+1, \ldots, x_d+1, x_{d+1}+1) \cong M(x_1, x_1+1, \ldots, x_d + 1) \in \mathcal{C}.
    \end{equation*}
    
    We need to prove that $M(x_1, \ldots, x_d, x_{d+1}-1)\in \mathcal{C}$. 
    We may assume $x_d<x_{d+1}-1$ or else $M(x_1,\ldots, x_d, x_{d+1}-1)=0$ and there is nothing to show.
    If $f(x_{d+1}-1) = x_1$, 
    then $M(x_1, \ldots, x_{d+1}-1)$ is projective hence in $\mathcal{C}$.
    Otherwise, $f(x_{d+1}-1)\leq x_1 - 1$ and $M(x_1-1, \ldots, x_{d+1}-1)\in \mathcal{M}$.
    Note that $M(x_1-1, \ldots, x_{d+1})\notin \mathcal{M}$ since $f(x_{d+1})=x_1>x_1-1$.
    Thus we have $g(x_1-1)=x_{d+1}-1$.
    
    Since $f(x_d+1)\leq f(x_{d+1}-1)\leq x_1-1$,  we have that $M(x_1, x_1+1, \ldots, x_d + 1)\in \mathcal{C}_P$.
    So Proposition \ref{prop_more_modules_from_tau_d_Omega_-1} (a) yields that 
    \begin{equation*}
        \Omega^{-d}\tau_dM(x_1, x_1+1, \ldots, x_d + 1) \cong \Omega^{-d}M(x_1-1, x_1, \ldots, x_d) \cong M(x_1, \ldots, x_d, x_{d+1}-1)\in \mathcal{C}.
    \end{equation*}
\end{proof}

As a corollary of the above Lemma, the following Lemma poses certain restrictions on the shape of $\underline{\ell}$ if $\modn A$ admits a partial $nd\mathbb{Z}$-cluster tilting subcategory.

\begin{Lem}\label{lem_necessaryshape}
    The algebra 
    $A = A_{\underline{\ell}}^{d}$ does not admit a partial  $nd\mathbb{Z}$-cluster tilting subcategory for $n\geq 2$ in the following cases.
    \begin{itemize}
        \item [(a)] $2 < \ell_{j-1} = \ell_j < \ell_{j+1}$ for some $j$.
        \item [(b)] $\ell_j > \ell_{j+1} > 2$ for some $j$.
    \end{itemize}
\end{Lem}
\begin{proof}
    Assume towards a contradiction that $\mathcal{C}$ is a partial  $nd\mathbb{Z}$-cluster tilting subcategory of $\modn A$ for some $n\geq 2$.

    Suppose (a) holds true.
    Since by definition $\ell_{j+1} \leq \ell_j + 1$,
    we have $\ell_{j+1} = \ell_j + 1$.
    So $f(j') = f(j'+1) = i$ for some integer $i$ and $j'$. 
    Indeed, $j' = j+d$. 
    Moreover, $g(i-1) = j'-1$ since $(i-1, \ldots, j')\notin os_{\underline{\ell}}^{d+1}$.
    The condition $\ell_j > 2$ implies that $i+d+1 < j'$.
    So instead, we construct a contradiction under the following condition
    \begin{equation*}
        (a')~~~ f(j') = f(j'+1) = i \text{ and } g(i-1) = j'-1 \text{ for some } i, j'\in \mathbb{Z} \text{ such that } i+d+1 < j'.
    \end{equation*}
    Since $i+d+1 < j'$, 
    \begin{equation*}
        z = (i, i+2, i+3, \ldots, i+d, j'-1)\in os_{\underline{\ell}}^{d+1}.
    \end{equation*}
    By Lemma \ref{lem_moremodules} (a), $ M(z)\in\mathcal{C}$.
    Since $g(i - 1) = j' - 1$ implies $f(j'-1) < i$, we have $M(z)$ is not projective.
    By Proposition \ref{prop_more_modules_from_tau_d_Omega_-1} (a),
    \begin{align*}
        \Omega^{-d}\tau_d M(z) & \cong M(i + 1, i + 2, \ldots, i+d-1, j'-2, j'-1) = M(w)\in \mathcal{C}.
    \end{align*}

Let $w' = (i, i+1, \ldots, i+d-2, i+d-1, j'-2)$.
By Lemma \ref{lem_moremodules} (a), $M(w')\in \mathcal{C}$.
Notice that $w'\preccurlyeq \tau_d(w)$.
By Proposition \ref{Prop_modcomb} (vi),
$\Ext_A^d(M(w), M(w')) \neq 0$
which contradicts with $\mathcal{C}$ being partial $nd\mathbb{Z}$-cluster tilting.
    
Dually, assumption (b) implies the following condition 
\begin{equation*}
    (b')~~~ g(i) = g(i+1) = j \text{ and } f(j+1) = i+2 \text{ for some } i,j\in\mathbb{Z} \text{ such that } i+d+1 < j.
\end{equation*}
We can prove similarly that $\modn A$ does not admit an $nd\mathbb{Z}$-cluster tilting subcategory under condition $(b')$.
\end{proof}

Recall that we have defined a deglue point for an acyclic Kupisch series in Definition \ref{def_deglue_acyclic_d_NakAlg} as well as a self-deglue point for a cyclic Kupisch series in Definition \ref{def_self-glue_and_self-deglue}.

\begin{Cor}\label{cor_necessary_shape_of_Kupisch_series}
    Suppose $\modn A$ admits a partial  $nd\mathbb{Z}$-cluster tilting subcategory. 
    Then $\underline{\ell}$ must be one  of the following shapes.
    \begin{itemize}
        \item [(a)] $\underline{\ell}$ is cyclic homogeneous.
        \item [(b)] $\underline{\ell}$ is acyclic homogeneous.
        \item [(c)] $\underline{\ell}$ is cyclic with a self-deglue point.
        \item [(d)] 
        $\underline{\ell}$ is acyclic with a deglue point.
    \end{itemize}
\end{Cor}

\begin{proof} 
    Assume $\underline{\ell}$ is not homogeneous,
    then there exists 
    \begin{equation*}
        \ell_{i-1}=\ell_i<\ell_{i+1} \text{ or } \ell_{i-1}=\ell_i>\ell_{i+1}
    \end{equation*}
    for some $i>1$.
    By excluding the cases in Lemma \ref{lem_necessaryshape},
    this gives a deglue point.
\end{proof}

\begin{Prop}\label{prop_A_is_piecewise_homogenous_simple_bridges_in_C}
    Let $A=A_{\underline{\ell}}^{d}$ be a non-self-injective $d$-Nakayama algebra.
    Suppose $\modn A$ admits a partial  $nd\mathbb{Z}$-cluster tilting subcategory $\mathcal{
C}$. The following statements hold true.
    \begin{itemize}
        \item [(a)] 
        $\underline{\ell} = \underline{\ell_1}\Delta \cdots \Delta \underline{\ell_t}$ 
        or
        $\underline{\ell} = (\underline{\ell_1}\Delta \cdots \Delta \underline{\ell_t})^c$ such that $\underline{\ell_i}$ is acyclic homogeneous for all $1\leq i\leq t$.
        \item [(b)] The simple bridges  
        \begin{equation*}
             S_i = M(1 + \sum_{k=1}^i m_k, 2+ \sum_{k=1}^i m_k, \ldots, d+1+\sum_{k=1}^i m_k), m_i = |\underline{\ell_i}|-1,  1\leq i\leq t-1,
        \end{equation*}
        are all included in $\mathcal{C}$.
    \end{itemize}
\end{Prop}
\begin{proof}
    As $A$ is non-self-injective, 
    $\underline{\ell}$ is not cyclic homogeneous.
    The statements hold true trivially if $\underline{\ell}$ is acyclic homogeneous.
    So by Corollary \ref{cor_necessary_shape_of_Kupisch_series}, 
    we only need to consider the following cases.
    \begin{itemize}
        \item [(1)] $\underline{\ell}$ is acyclic with a deglue point.
        \item [(2)] $\underline{\ell}$ is cyclic with a self-deglue point.
    \end{itemize}
    In case (1), we may write 
    \begin{equation*}
         \underline{\ell} = \underline{\ell_1}\Delta_S \underline{\ell_2}
    \end{equation*}
    with 
    \begin{equation*}
        S=M(m, m+1, \ldots, m+d)
    \end{equation*}
     where $m=|\underline{\ell_1}|$.
     By the definition of degluing, 
     we have 
     \begin{equation*}
         \ell_m\geq \ell_{m+1} = 2,  \text{ and }
         \ell_{m+1}<\ell_{m+2}.
     \end{equation*}
     This implies $\ell_{m+2}=\ell_{m+1}+1$.
     By Lemma \ref{lem_moremodules} (a),
     $S\in \mathcal{C}$.
     So Proposition \ref{prop_partial_d_ct_of_lambda_obtained_from_gluing} applies and 
     \begin{equation*}
         \mathcal{C}_i = \mathcal{C}\cap \modn A_i \subseteq \modn A_i
     \end{equation*}
     is partial $nd\mathbb{Z}$-cluster tilting where $A_i = A_{\underline{\ell_i}}^{d}$ for $i=1,2$.

     Iterating the degluing procedure, we conclude that statement (a) and (b) hold true.

     In case (2), let $\underline{\ell}^a$ be the self-degluing of $\underline{\ell}$ and $A^a = A_{\underline{\ell}^a}^d$ the corresponding acyclic $d$-Nakayama algebra.
     Let $\mathcal{M}\subseteq \modn A$ and $\mathcal{M}^a\subseteq \modn A^a$ be the distinguished $d\mathbb{Z}$-cluster tilting subcategories.
     By Lemma \ref{lem_partial_ndZ_ct_of_A_gives_that_of_A^a}, 
     $\modn A^a$ admits a partial $nd\mathbb{Z}$-cluster tilting subcategory given by 
     \begin{equation*}
         \mathcal{C}^a = \add \{M^a(x)\mid M(x)\in \mathcal{C}\cap \mathcal{M} \text{ or } x \in  \{\lambda, \mu\}\}
     \end{equation*}
     where $\lambda = (1,2,\ldots, d+1)$ and $\mu = (m, m+1, \ldots, m+d)$ with $m = |\underline{\ell}|$.
     In particular $\underline{\ell}^a$ is of the form as in case (1). 
     Applying statements (a) and (b), we have  
     \begin{equation*}
         \underline{\ell}^a = \underline{\ell_1}^a \Delta_{S_1}\cdots\Delta_{S_{t-1}}\underline{\ell_t}^a
     \end{equation*}
     such that 
     $\underline{\ell_i}^a$ is acyclic homogeneous for each $1\leq i\leq t$ and 
     $S_i\in \mathcal{C}^a$ for all $1\leq i\leq t-1$.

     Therefore, $\underline{\ell} = (\underline{\ell}^a)^c$ and all $S_i \in \mathcal{C}$ for $1\leq i\leq t-1$.
\end{proof}

\begin{Def}
    Let $A=A_{\underline{\ell}}^{d}$.
    We call $A$ piecewise homogeneous if $\underline{\ell} = \underline{\ell_1}\Delta \cdots \Delta \underline{\ell_t}$ such that $\underline{\ell_i}$ is homogeneous for $1\leq i\leq t$.
    And we call $A$ cyclic piecewise homogeneous if $\underline{\ell} = (\underline{\ell_1}\Delta \cdots \Delta \underline{\ell_t})^c$ with $\underline{\ell_i}$ being homogeneous for $1\leq i\leq t$.
\end{Def}

\subsection{Non-selfinjective case} \label{sec_non-self-inj_case}

In this section, we give a classification of non-self-injective  $d$-Nakayama algebras which admit an  $nd\mathbb{Z}$-cluster tilting subcategory.
More precisely, such algebras are given by piecewise homogeneous or cyclic piecewise homogeneous ones of particular shape.
We start with the homogeneous acyclic case.
We will obtain the classification result for the piecewise homogeneous case by applying gluing techniques.
Finally the classification result for cyclic piecewise homogeneous case will be obtained by applying the orbit construction.

\subsubsection{Homogeneous acyclic case}\label{subsec_homo_acyclic_case}

Let $A = A^{d}_{\ell, m}$ be an acyclic homogeneous $d$-Nakayama algebra and $\mathcal{M}$ the distinguished $d\mathbb{Z}$-cluster tilting subcategory.
The map $f$ and $g$ simplify as follows.

\begin{Prop}\label{prop_f_g_acyclic_homo}
    We have 
    \begin{equation*}
    f(i) = \left\{\begin{array}{cc}
        1 & d+1\leq i\leq \ell + d \\
        i - \ell - d + 1 & \ell + d + 1 \leq i \leq m 
    \end{array}\right.
\end{equation*}
and 
\begin{equation*}
    g(i) = \left\{\begin{array}{cc}
       i + \ell + d - 1  & 1\leq i\leq m - \ell + 1  \\
        m+d  &  m - \ell + 2 \leq i \leq m.
    \end{array}\right.
\end{equation*}
\end{Prop}

The global dimension of $A=A_{\ell, m}^d$ was determined in \cite[Theorem 5.28]{Ber20} which will be useful for us. We recall the result by the following Lemma.

\begin{Lem}\cite[Theorem 5.28]{Ber20}\label{lem_global_dim_acyclic_homo}
    Let $A=A_{\ell, m}^d$ be an acyclic homogeneous $d$-Nakayama algebra.
    Consider 
    \begin{equation*}
        q = \lfloor \frac{m-1}{\ell+d-1}\rfloor
    \end{equation*}
    and 
    \begin{equation*}
        r = \left\{
        \begin{array}{cc}
            1 & \text{ if } m-1-q(\ell+d-1) < \ell \\
            m-1-q(\ell+d-1)-\ell+2 & \text{ otherwise. }
        \end{array}\right.
    \end{equation*}
    The the global dimension of $A$ is given as follows.
    \begin{equation*}
        \gldim A = \left\{
        \begin{array}{cc}
           d(d+1)q  & \text{ if } (\ell+d-1)|(m-1)  \\
            d((d+1)q+r) & \text{ otherwise.} 
        \end{array}\right.
    \end{equation*}
\end{Lem}

Note that when $\ell=2$, the algebra $A$ is a classical Nakayama algebra.
A classification result for $nd\mathbb{Z}$-cluster tilting subcategories of $\modn A$ was obtained in \cite{HKV25} as we recall by the following proposition.

\begin{Prop}\label{prop_classification_homo_acyclic_l=2}
    Let $A = A_{2,m}^{d}$.
    The following statements are equivalent.
    \begin{itemize}
        \item [(a)] 
        $\modn A$ admits an $nd\mathbb{Z}$-cluster tilting subcategory $\mathcal{C}$.
        \item [(b)] $\modn A$ admits a partial  $nd\mathbb{Z}$-cluster tilting subcategory $\mathcal{C}$.
        \item [(c)] $n|(m-1)$.
    \end{itemize}
    If one of the above conditions hold true, then the subcategory in (a) and (b) is unique and 
     \begin{equation*}
        \mathcal{C} = \add(\{A\}\cup \{\tau_{nd}^{-k}(M(1, 2, \ldots, d+1))\}) = \add(\{A\}\cup \{M(1 + kn, \ldots, d+1+kn)\mid 1\leq k\leq \frac{m-1}{n}\}).
    \end{equation*}
\end{Prop}
\begin{proof}
    By \cite[Theorem 1]{Vas19}, (a) and (b) are equivalent as $A$ is representation-directed.
    By \cite[Proposition 4.9]{HKV25}, (a) and (c) are equivalent and $\mathcal{C}$ is of the given form.
\end{proof}

Next we address the case $\ell\geq 3$.
For $A = A_{\ell, m}^d$, Proposition \ref{prop_homo_acyclic_equiv_conditions} shows that $\modn A$ admits an $nd\mathbb{Z}$-cluster tilting subcategory if and only if certain numerical conditions on $n, d, \ell$, and $m$ are satisfied. In particular, when these conditions hold, such an $nd\mathbb{Z}$-cluster tilting subcategory is unique and is given by $\add(A \oplus DA)$.
After the completion of this paper, the author became aware that the implication (c) $\Rightarrow$ (a) in Proposition \ref{prop_homo_acyclic_equiv_conditions} had already been proved in \cite[Proposition 6.14]{Sen23}. To keep this paper self-contained, we include full proofs here.

\begin{Prop}\label{prop_homo_acyclic_equiv_conditions}
    Let $A = A_{\ell, m}^d$ be an acyclic homogeneous $d$-Nakayama algebra with $\ell\geq 3$.
    The following conditions are equivalent.
    \begin{itemize}
        \item [(a)] $\modn A$ admits an $nd\mathbb{Z}$-cluster tilting subcategory $\mathcal{C}$.
        \item [(b)] $\modn A$ admits a partial $nd\mathbb{Z}$-cluster tilting subcategory $\mathcal{C}$.
        \item [(c)] $(d+1) | (n-2)$ and $m=\frac{n-2}{d+1}(\ell+d-1)+\ell+1$.
    \end{itemize}
    If one of the above conditions holds, then $\mathcal{C}=\add (A\oplus DA)$. 
\end{Prop}
\begin{proof}
    By Proposition \ref{prop_d_ct_implies_weakly_d_ct}, (a) implies (b).

    Now we prove (b) implies (c).
    As a consequence of Proposition \ref{prop_f_g_acyclic_homo}, 
    we get that  
    \begin{equation*}
        \mathcal{P} = \{M(x)\mid x= (1, x_2, \ldots, x_{d+1}) \text{ and } 1<x_2< \cdots < x_{d+1} < \ell + d\}
    \end{equation*}
    the set of isomorphism classes of indecomposable projective noninjective $A$-modules and 
    \begin{equation*}
        \mathcal{I} = \{M(y)\mid y = (y_1, \ldots, y_d, m+d) \text{ and } m-\ell +  1 < y_1 < y_2 < \cdots < y_d < m+d\}
    \end{equation*}
    the set of isomorphism classes of indecomposable injective nonprojective $A$-modules.
    Further, we denote by 
    \begin{equation*}
        \mathcal{F} = \{M(i, \ldots, j)\mid i + \ell + d -1 = j\}
    \end{equation*}
    the set of isomorphism classes of indecomposable projective-injective $A$-modules.
    
    Firstly we show that $\tau_{nd}^{-k}\mathcal{P}\subseteq \mathcal{I}$ for some positive integer $k$.

    Let $M(x), M(y)\in \mathcal{P}$ such that $x_i<y_i$ for some $1\leq i\leq d+1$ and $x_j=y_j$ for $j\neq i$.
    We claim that $\tau_{nd}^{-k}M(x)\in \mathcal{I}$ if and only if $\tau_{nd}^{-k}M(y)\in \mathcal{I}$.

    Let $k,k'$ be the minimal integers such that 
    \begin{equation*}
        M(x')=\tau_{nd}^{-k}M(x)\in \mathcal{I} \text{ and } M(y')=\tau_{nd}^{-k'}M(y)\in \mathcal{I}.
    \end{equation*}
    
    Assume $k>k'$. 
    Then $M(x'')=\tau_{nd}^{-k'}M(x)\notin \mathcal{I}$.
    So $x''_{d+1}<m+d=y'_{d+1}$ and $x''_i=y'_i$ for $1\leq i\leq d$.
    But then $g(x''_1)=g(y'_1)=m+d$ which implies $\Omega^{-d}M(x'')\in \mathcal{I}$.
    So $\tau_{nd}^{-k}M(x) = 0\notin \mathcal{I}$, a contradiction.
    Arguing similarly, we can prove that $k<k'$ is impossible as well.
    Thus $k=k'$ and our claim follows.

    Next we show $k=1$.
    Assume towards a contradiction that $k>1$. 
    Let 
    \begin{equation*}
        x=(1,2,\ldots,d+1) \text{ and } y=(1,3,\ldots,d+2).
    \end{equation*}
    We have $M(x), M(y)\in \mathcal{P}$. 

    Since $\tau_{nd}^{-1}M(y)\notin \mathcal{I}$, we have $\Omega^{-nd}M(y)\in \mathcal{C}_I$. 
    So 
    \begin{equation*}
        M(z) =\tau_{nd}\Omega^{-nd}M(y) \cong \tau_{d}\Omega^{-d}M(y)\cong M(2, \ldots, d+1, d+\ell-1)\in \mathcal{C}.
    \end{equation*}
    But then $\Ext_A^d(M(z), M(x))\neq 0$ since $x\preccurlyeq \tau_d(z)$, a contradiction.
    Therefore $k=1$.

    Now we write $n-1=a(d+1)+b$ with $0\leq b\leq d$. Observe that for a positive integer $t$, we have 
    \begin{align*}
        \Omega^{-d(d+1)t}M(x) & =  M(g^t(x_1), \ldots, g^t(x_{d+1}))\\
        & = M(x_1 + t(\ell + d - 1), \ldots,x_{d+1} + t(\ell + d - 1))\\
        & = \tau_d^{-t(\ell + d - 1)}M(x).
    \end{align*}
    Thus 
    \begin{align*}
        \tau_{nd}^{-1}M(x) & = \tau_d^{-1}\Omega^{-d(n-1)}M(x) \\
        & = \tau_d^{-1}\Omega^{-d(d+1)a - db}M(x)\\
        & = \tau_d^{-1-a(\ell+d-1)}\Omega^{-db}M(x)\\
        & = \tau_d^{-1-a(\ell+d-1)}M(x_{b+1}, \ldots, x_{d+1}, g(x_1), \ldots, g(x_b))\\
        & = M(x_{b+1}+1+a(\ell+d-1), \ldots, g(x_b)+1+a(\ell+d-1)) \in \mathcal{I}.
    \end{align*}
    Similarly,
    \begin{equation*}
        \tau_{nd}^{-1}M(y) = M(y_{b+1}+1+a(\ell+d-1), \ldots, g(y_b)+1+a(\ell+d-1)) \in \mathcal{I}.
    \end{equation*}
    Then we have 
    \begin{equation*}
        g(x_b)+1+a(\ell+d-1)) = g(y_b)+1+a(\ell+d-1)) = m + d
    \end{equation*}
    which implies $b = 1$ and thus $\ell+a(\ell+d-1)+1 = m$.
    Therefore, $(d+1) | (n-2)$ and $m = \frac{n-2}{d+1}(\ell+d-1)+\ell+1$.

    So (b) implies (c).

    It will follow from Lemma \ref{lem_gldim_auslander_alg} that (c) implies (a).
\end{proof}

\begin{Lem}\label{lem_gldim_auslander_alg}
    Let $A=A_{\ell, m}^d$ be an acyclic homogeneous $d$-Nakayama algebra.
    Assume $\ell\geq 3$ and
    \begin{equation*}
        n-2=k(d+1), m=k(\ell+d-1)+\ell+1
    \end{equation*}
    for some integer $k\geq 1$. 
    Then $\mathcal{C}=\add(A\oplus DA)\subseteq \modn A$ is an $nd\mathbb{Z}$-cluster tilting subcategory.
\end{Lem}
\begin{proof}
    Firstly we prove that  $\Ext_A^j(\mathcal{C}, \mathcal{C}) = 0$ for all $0<j<nd$.
    since $\mathcal{C}\subseteq \mathcal{M}$, 
    it suffices to show 
    \begin{equation*}
        \Ext_A^{(i+1)d}(M(y), M(x)) =   0, \text{ for all } 0\leq i\leq n-2
    \end{equation*}
    where $x=(1,x_1,\ldots,x_{d+1})$ with $x_{d+1}\leq \ell+d-1$ and $y=(y_1, \ldots, y_d, m+d)$ with $y_1\geq m-\ell+2$.
    Note that 
    \begin{equation*}
        \Ext_A^{(i+1)d}(M(y), M(x)) \cong \Ext_A^d(\Omega^{id}M(y), M(x)) \cong D\overline{\Hom}_A(M(x), \tau_d\Omega^{id}M(y)).
    \end{equation*}
    Write $\Omega^{id}M(y) = M(y')$ where $y'=(y'_1, \ldots, y'_{d+1})$.
    We claim that $y'_{d+1}\geq \ell+d+1$.
    If not, 
    then
    \begin{equation*}
        \Omega^{(i+1)d}M(y) = M(1, y''_2,\ldots, y_{d+1}'')
    \end{equation*}
    which is projective.
    This contradicts with the fact that $\tau_{nd}M(y)$ is projective and non-injective as $i+1\leq n-1$.

    We have 
    \begin{equation*}
        \tau_d\Omega^{id}M(y) = M(y_1'-1, \ldots, y_{d+1}'-1)
    \end{equation*}
    and $y'_{d+1}-1\geq \ell+d$.
    If $x\preccurlyeq (y_1'-1, \ldots, y_{d+1}'-1)$ then
    \begin{equation*}
        x\preccurlyeq (1, x_2, \ldots, x_d, \ell+d) \text{ and } (1, x_2, \ldots, x_d, \ell+d)\preccurlyeq (y_1'-1, \ldots, y_{d+1}'-1).
    \end{equation*}
    Note that $M(1,x_2,\ldots,x_d,\ell+d)$ is injective.
    This implies that any nonzero morphism from $M(x)$ to $\tau_d\Omega^{id}M(y)$ would factor through an injective module.
    Thus 
    \begin{equation*}
        \Ext_A^{(i+1)d}(M(y), M(x)) = 0
    \end{equation*}
    for all $0\leq i\leq n-2$ with $M(x)\in \add DA$ and $M(y)\in \add A$.
    
    Let $C=P\oplus I$ be the basic additive generator of $\mathcal{C}$ such that an indecomposable module $Q\in \add I$ if and only if $Q\cong M(x_1,\ldots,x_d,m+d)$.
    Denote by $\Lambda = \End_A(C)$ the endomorphism algebra of $C$.
    We prove instead $\gldim \Lambda = nd+1$ by applying Lemma \ref{lem_glodim_endomorphism_algebra}.
    Then by Lemma \ref{lem_characterisation_of_d_ct_module} (b), we have that 
    $\mathcal{C} = \add C$ is $nd$-cluster tilting as $C$ is an $nd$-rigid generator-cogenerator.
    
    Firstly we show that $\resdim_C Q\leq nd$ for $Q$ an indecomposable direct summand of $I$.

    Write $Q=M(y)$ where $y=(y_1, \ldots, y_d, m+d)\in os_{m,\ell}^{d+1}$ such that $y_1\geq m-\ell+2$.
    Then $\tau_d Q \cong M(\tau_d(y))$.
    As $\dim_{\Bbbk}\Ext_A^d(Q, \tau_d Q) = 1$,
    arguing similarly as in \cite[Theorem 3.8]{OT12},
    we have the $d$-almost split sequence of $Q$ as follows. 
    \begin{equation*}
        \xymatrix@C=1em@R=1em{
        0\ar[r] & \tau_d Q\ar[r] & E_d\ar[r] & \cdots\ar[r] & E_1\ar[r] & Q\ar[r] & 0 & (\ast)
        }
    \end{equation*}
    where $E_k = \bigoplus _z M(z)$ such that $z_i\in \{y_i, y_i-1\}$ for all $i$ and $|\{i\mid z_i=y_i-1\}|=k$.

    We write $E_k = E'_k\oplus E''_k$ where $E'_k\in \add I$ and $E''_k\notin \add I$. Set
    \begin{equation*}
        \sigma_k = \{z\mid z_i\in \{y_i, y_i-1\}, z_{d+1}=m+d  \text{ and } |\{i\mid z_i=y_i-1\}|=k\}.
    \end{equation*}
    Then $E_k' = \bigoplus_{z\in \sigma_k} M(z)$.
    In particular, 
    \begin{equation*}
        E_d' = M(y_1-1, \ldots, y_d-1, m+d)
    \end{equation*}
    is indecomposable.
    Denote by $S = S_{y^\ast}$ the simple socle of $E_d'$ i.e.  $y^{\ast}=(y_1-1,\ldots,y_d-1)$.
    We claim that the following sequence is exact.
    \begin{equation*}
        \xymatrix@C=1em@R=1em{
        0\ar[r] & S\ar[r] & E'_d\ar[r] & E'_{d-1}\ar[r] & \cdots\ar[r] & E'_1\ar[r] & Q\ar[r] & 0. & (\ast\ast)
        }
    \end{equation*}
     Consider the $d$-Auslander algebra $\Lambda = A_{\ell}^d$ and denote by $\mathcal{M}_{\Lambda} \subseteq \modn \Lambda$ the $d\mathbb{Z}$-cluster tilting subcategory.
     Let
     \begin{equation*}
         \lambda = \{x\in os_{\ell,m}^d\mid x_1\leq m-\ell\} \text{ and } f = \sum_{x\in \lambda} e_x.
     \end{equation*}
     We may view $\Lambda \cong A/AfA$.

     Arguing by coordinates, 
     we may view $(\ast\ast)$ in $\modn \Lambda$ and prove it is exact.

     To see this we view $\Lambda = \End_{\Lambda'}(M')$ where $\Lambda' = A_{\ell}^{d-1}$ and $\mathcal{M'} = \add M' \subseteq \modn \Lambda'$ is the $(d-1)\mathbb{Z}$-cluster tilting subcategory.

     For each $1\leq k\leq d$ and each $x\in \sigma_k$,
     let $z_x = (x_1, \ldots, x_d)$. 
     Moreover let $z_Q = (y_1, \ldots, y_d)$.
     Notice that $M(z_x)\in \mathcal{M}'$.
     Let $F_k = \bigoplus_{x\in \sigma_k}M(z_x)$.
     In particular $F_d = M(y^{\ast})$.
     Again by arguing with coordinates, we conclude the following sequence gives the $(d-1)$-almost split sequence in $\mathcal{M}'$ that starts from $F_d$.
     \begin{equation*}
         \xymatrix@C=1em@R=1em{
         0\ar[r] & F_d\ar[r] & F_{d-1}\ar[r] & \cdots\ar[r] & F_1\ar[r] & M(z_Q)\ar[r] & 0. & (\dag)
        }
     \end{equation*}
     By applying $D\Hom_{\Lambda'}(-,M')$ to $(\dag)$ and attaching the kernel of the first map,
     we obtain an exact sequence which coincides with $(\ast\ast)$.
   
     In fact $(\ast\ast)$ is the minimal $\add C/Q$-approximation sequence of $Q$ and   $\Omega_C^d(Q) = S_{y^{\ast}}$. 
    
    Note that $\Hom_A(I, S_{y^{\ast}})=0$. 
    Indeed, $S_{(u_1,\ldots,u_d)}\in \add (\top I)$ implies $u_d = m+d-1$ but $y_d-1<m+d-1$. 
    Therefore, $\Omega_C^jS_{y^{\ast}} \cong \Omega_A^jS_{y^{\ast}}$ for $j\geq 0$.

    Next we calculate $\projdim_A S_{y^{\ast}}$. 

    We consider $A' = A_{\ell, m-1}^{d}$.
    Indeed, $A'\cong eAe\cong A/\langle 1-e\rangle$ where $e=\sum_{v}e_v$ and $v=(v_1, \ldots, v_d)\in os_{\ell, m}^d$ with $v_d\leq m+d-1$.
    Consider the adjoint pair $(G, H)$ defined as follows. 
    \begin{equation*}
        \xymatrix@R=6em{
        \modn A \ar@/_/[rr]_{H=\Hom_{A}(eA, -)} && \modn eAe \cong \modn A'\ar@/_/[ll]_{G = -\otimes_{eAe} eA}
        }
    \end{equation*}
    Since $eA \cong eAe\oplus eA(1-e) = eAe$ as $eAe$-modules, $G$ is exact. 
    Moreover, $G$ preserves projective modules since it admits a right adjoint $H$ which is exact. 
    We may view $G$ as an exact embedding and conclude that $\projdim GU = \projdim U$ for $U\in \modn A'$. 

    So we may view $S_{y^{\ast}}\in \modn A'$ and calculate $\projdim_{A'}S_{y^{\ast}}$. 

    By Lemma \ref{lem_global_dim_acyclic_homo}, we have $\gldim A'=(n-1)d$. So $\projdim_{A'}S_{y^{\ast}}\leq (n-1)d$. 
    This implies $\resdim_C S_{y^{\ast}}\leq (n-1)d$.
    Thus $\resdim_C Q\leq nd$.

    Now we consider $T=M(m,m+1,\ldots,m+d)\in \add I$. 
    In the $d$-almost split sequence $(\ast)$ for $T$, we have that $E_i = E_i'$ for all $i$ thus $\Omega_C^d(T) = \tau_dT = M(m-1, m, \ldots, m+d-1)$.
    Viewed as an $A'$-module, we have $\projdim_{A'}\Omega_C^d(T) = (n-1)d$.
    Thus $\resdim_C T = nd$.

    Lastly we need to calculate $\resdim_C \rad P$ for $P\in \add A$ incomposable. 
    We may view $\rad P$ as an $A$-module and conclude that $\resdim_C \rad P = \projdim_{A}\rad P \leq nd$ as $\gldim A = nd$ by Lemma \ref{lem_global_dim_acyclic_homo}. 

    Altogether we have $\gldim \Lambda = nd+1$ by Lemma \ref{lem_glodim_endomorphism_algebra}.
\end{proof}

\begin{Thm}\label{thm_classification_homo_acyclic}
    Let $A = A^{d}_{\ell, m}$ be a homogeneous acyclic $d$-Nakayama algebra. 
    There exists an $nd\mathbb{Z}$-cluster tilting subcategory $\mathcal{C}\subseteq \modn A$ if and only if one of the following conditions is satisfied.
    \begin{itemize}
        \item [(a)] $\ell = 2$ and $n\mid (m-1)$.
        \item [(b)] $\ell\geq 3$, $(d+1) | (n- 2)$ and $m = \frac{n-2}{d+1}(\ell +d -1) + \ell + 1$.
    \end{itemize}
    Moreover, we have in case (a) that 
    \begin{equation*}
        \mathcal{C} = \add(\{A\}\cup \{\tau_{nd}^{-k}(M(1, 2, \ldots, d+1))\}) = \add(\{A\}\cup \{M(1 + kn, \ldots, d+1+kn)\mid 1\leq k\leq \frac{m-1}{n}\})
    \end{equation*}
    and in case (b) that $\mathcal{C} = \add(A\oplus DA)$.
\end{Thm}

\begin{proof}
    The statement follows from Proposition \ref{prop_classification_homo_acyclic_l=2} and Proposition \ref{prop_homo_acyclic_equiv_conditions}.
\end{proof}

\begin{Rem}\label{rem_finer_deglue}
    Let $A=A_{2,m}^{d}$ such that $\mathcal{C}\subseteq \modn A$ is $nd\mathbb{Z}$-cluster tilting. 
    As shown in Theorem \ref{thm_classification_homo_acyclic}, 
    $\mathcal{C} = \add(\{A\} \cup \{M_k\mid 1\leq k\leq \frac{m-1}{n}\})$.
    We may deglue $A$ further such that these $M_k$'s are simple bridges.
    From now on, we consider this finer degluing of $A$ whenever $A$ admits an $nd\mathbb{Z}$-cluster tilting subcategory.
    Thus if a homogeneous piece $A$ of a degluing has an $nd\mathbb{Z}$-cluster tilting subcategory $\mathcal{C}\subseteq \modn A$, 
    we may assume $\mathcal{C} = \add\{A, DA\}$.
\end{Rem}

\begin{Rem}
    In case (b) in the above Theorem, 
    we get $2$-subhomogeneous $nd$-representation finite algebras, see \cite{Xin25} for more details.
\end{Rem}

\begin{Rem}
    When $d = 1$, the above Theorem recovers the corresponding result for classical Nakayama algebras, see \cite[Proposition 4.2]{HKV25}.
\end{Rem}

\subsubsection{Non-homogeneous acyclic case.}\label{subsec_non_homo_acyclic_case}

Let $A=A_{\underline{\ell}}^{d}$ be an acyclic $d$-Nakayama algebra. 
We give a classification for such an $A$ which admits an $nd\mathbb{Z}$-cluster tilting subcategory in its module category.

\begin{Thm}\label{thm_classification_nonhomo_acyclic}
   Let $A=A_{\underline{\ell}}^{d}$ be an acyclic $d$-Nakayama algebra.
   Then $\modn A$ has an $nd\mathbb{Z}$-cluster tilting subcategory $\mathcal{C}$ if and only if the following conditions hold.
    \begin{itemize}
        \item [(i)] $A = A_1\Delta A_2\Delta \cdots \Delta A_t$ with $A_i = A_{\ell_i, m_i+1}^{d}$ for $1\leq i\leq t$.
        \item [(ii)] For $1\leq i\leq t$,
        one of the following conditions holds.
        \begin{itemize}
            \item [(a)] $\ell_i=2$ and $n=m_i$.
            \item [(b)] $(d+1) | (n-2)$ and $m_i = \frac{n-2}{d+1}(\ell_i + d - 1) + \ell_i$.
        \end{itemize}
    \end{itemize}
    Moreover, in that case $\mathcal{C} = \add (\{A \oplus DA\} \cup \{S_i\mid 1\leq i\leq t-1\})$ where 
    \begin{equation*}
        S_i = M(1 + \sum_{k=1}^i m_i, 2+ \sum_{k=1}^i m_i, \ldots, d+1+\sum_{k=1}^i m_i).
    \end{equation*}
\end{Thm}
\begin{proof}
    Firstly we show the necessity.
    
    By Proposition \ref{prop_A_is_piecewise_homogenous_simple_bridges_in_C},  $A = A_1\Delta A_2\Delta \cdots \Delta A_s$ where $A_i = A_{\ell_i, m_i+1}^{d}$  for $1\leq i\leq s$ and the simple bridges 
    \begin{equation*}
        S_i = M(1 + \sum_{k=1}^i m_k, 2+ \sum_{k=1}^i m_k, \ldots, d+1+\sum_{k=1}^i m_k), 1\leq i\leq s-1
    \end{equation*}
    are included in $\mathcal{C}$.
    
    Let $\mathcal{M}_i\subseteq \modn A_i$ be the distinguished $d\mathbb{Z}$-cluster tilting subcategory.
    Then 
    \begin{equation*}
        \mathcal{M}_i = \mathcal{M}\cap \modn A_i
    \end{equation*}
    by Proposition \ref{prop_glue_acyclic_nak_and_form_of_simple_bridge}.
    Hence by Proposition \ref{prop_C_intersecting_M_gives_weakly_ndZ_ct}, for each $1\leq i\leq s$, 
    \begin{equation*}
        \mathcal{C}_i = \mathcal{C}\cap \mathcal{M}_i\subseteq \modn A_i
    \end{equation*}
    is partial $nd\mathbb{Z}$-cluster tilting.

    For each $1\leq i\leq s$ such that $\ell_i=2$, 
    we may apply Proposition \ref{prop_classification_homo_acyclic_l=2} and consider a finer degluing as in Remark \ref{rem_finer_deglue} to get 
    \begin{equation*}
        A=A_1\Delta \cdots\Delta A_t
    \end{equation*}
    for some positive integer $t$.
    
    So (ii) holds true by 
    Proposition \ref{prop_classification_homo_acyclic_l=2} and 
    Proposition \ref{prop_homo_acyclic_equiv_conditions}.
    In particular, 
    \begin{equation*}
        \mathcal{C}_j = \add(A_j\oplus DA_j) \subseteq \modn A_j
    \end{equation*}
    is $nd\mathbb{Z}$-cluster tilting.

    By Proposition \ref{prop_glue_d_ct_2_d_ct}, 
    \begin{equation*}
        \mathcal{C}' = \mathcal{C}_1\Delta \mathcal{C}_2\Delta \cdots \Delta \mathcal{C}_t \subseteq \modn A
    \end{equation*}
    is $nd\mathbb{Z}$-cluster tilting.
    As $\mathcal{C}'\subseteq \mathcal{C}$ and $\mathcal{C}$ is $nd\mathbb{Z}$-cluster tilting, 
    we have
    \begin{equation*}
        \mathcal{C} = \mathcal{C}_1\Delta\cdots\Delta \mathcal{C}_t
    \end{equation*}
    and
    \begin{equation*}
        \mathcal{C} = \add(\{A\oplus DA\}\cup \{S_i\mid 1\leq i\leq t\}).
    \end{equation*}

    Now we prove the sufficiency.
    Given (i) and (ii), we have 
    $\mathcal{C}_i = \add(A_i\oplus DA_i)\subseteq \modn A_i$ is $nd\mathbb{Z}$-cluster tilting for $1\leq i\leq t$.
    By Proposition \ref{prop_glue_d_ct_2_d_ct},
    $\mathcal{C} = \mathcal{C}_1\Delta \cdots \Delta \mathcal{C}_t\subseteq \modn A$ is $nd\mathbb{Z}$-cluster tilting and 
    $\mathcal{C} = \add(\{A\oplus DA\}\cup \{S_i\mid 1\leq i\leq t\})$.
\end{proof}

For a positive integer $s$, recall that an algebra $\Lambda$ is called $s$-representation finite if $\gldim \Lambda=s$ and there exists an $s$-cluster tilting subcategory  $\mathcal{C} = \add C\subseteq \modn \Lambda$ with $C\in \modn \Lambda$.  
Notice that in this case, $\mathcal{C}$ is trivially $s\mathbb{Z}$-cluster tilting.
In \cite[Example 7.4]{Sen23}, the following question was posed: whether every $s$-representation finite $d$-Nakayama algebra is given by the acyclic homogeneous ones satisfying the numerical condition in Proposition \ref{prop_homo_acyclic_equiv_conditions} (c). We give an affirmative answer to this question as a consequence of Theorem \ref{thm_classification_nonhomo_acyclic}. 

\begin{Cor}
    A $d$-Nakayama algebra $A$ is $s$-representation finite if and only if $A = A_{\ell,m}^d$ and $s=nd$ for some integer $n>1$ such that 
    \begin{equation*}
        (d+1)|(n-2) \text{ and } m = \frac{n-2}{d+1}(\ell+d-1)+\ell+1.
    \end{equation*}
\end{Cor}
\begin{proof}
    Since $\modn A$ admits a $d\mathbb{Z}$-cluster tilting subcategory, $\gldim A < \infty$ implies $\gldim A\in d\mathbb{Z}$.
    Hence $s=nd$ for some integer $n>1$.

    The sufficiency part follows from Proposition \ref{prop_homo_acyclic_equiv_conditions}, and applying Lemma \ref{lem_global_dim_acyclic_homo} yields $\gldim A = nd$; see \cite[Proposition 6.14]{Sen23} for an alternative proof.

    Now we prove necessity. Assume that $\modn A$ admits an $nd\mathbb{Z}$-cluster tilting subcategory. Then, by Theorem \ref{thm_classification_nonhomo_acyclic} and Proposition \ref{prop_necessary_condition_of_cyclic_non_selfinj}, we have either $A = A_1\Delta \cdots \Delta A_t$ or $A = (A_1\Delta \cdots \Delta A_t)^c$, where each $A_i$ is an acyclic homogeneous $d$-Nakayama algebra satisfying the stated numerical condition for $1\leq i\leq t$. In either case,
    \begin{equation*}
        \gldim A > \max\{\gldim A_i \mid 1\leq i \leq t\} \geq nd.
    \end{equation*}
    It follows that $t = 1$ and $A = A_{\ell, m}^d$ with
    \begin{equation*}
        (d+1)|(n-2) \text{ and } m = \frac{n-2}{d+1}(\ell+d-1)+\ell+1.
    \end{equation*}
\end{proof}

\begin{Cor}\label{cor_acyclic_ndZ_ct_iff_weakly_ndZ_ct}
    Let $A$ be an acyclic $d$-Nakayama algebra.
    Then a subcategory  $\mathcal{C}\subseteq \modn A$ is $nd\mathbb{Z}$-cluster tilting if and only if it is partial  $nd\mathbb{Z}$-cluster tilting. 
\end{Cor}

\begin{Rem}
    By definition, the set $\mathbb{P}$ of  partial $d\mathbb{Z}$-cluster tilting subcategories is closed under intersections.
    However, for $d\mathbb{Z}$-cluster tilting subcategories, 
    inclusion implies equality.
    Thus if partial $d\mathbb{Z}$-cluster tilting coincides with $d\mathbb{Z}$-cluster tilting,
    the only possibility is that either $\mathbb{P}$ is empty or a singleton. 
\end{Rem}

We end this section with an example.

\begin{Eg}\label{example_glued_2nakalg}
    The Gabriel quiver of the $2$-Nakayama algebra $A = A^{2}_{3,8}\Delta A^{2}_{2,6}$ is given by
        
    
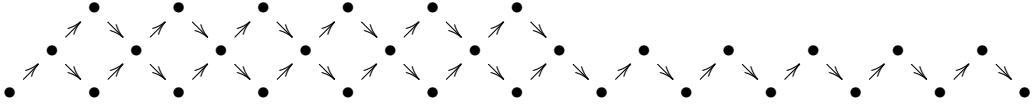
\begin{equation*}
        \begin{xy}
            0;<16pt,0cm>:<16pt,16pt>::
            (0,0) *+{\bullet} ="12",
            (0,1) *+{\bullet} ="13",
            (0,2) *+{\bullet} ="14",
            (2,0) *+{\bullet} ="23",
            (2,1) *+{\bullet} ="24",
            (2,2) *+{\bullet} ="25",
            (4,0) *+{\bullet} ="34",
            (4,1) *+{\bullet} ="35",
            (4,2) *+{\bullet} ="36",
            (6,0) *+{\bullet} ="45",
            (6,1) *+{\bullet} ="46",
            (6,2) *+{\bullet} ="47",
            (8,0) *+{\bullet} ="56",
            (8,1) *+{\bullet} ="57",
            (8,2) *+{\bullet} ="58",
            (10,0) *+{\bullet} ="67",
            (10,1) *+{\bullet} ="68",
            (10,2) *+{\bullet} ="69",
            (12,0) *+{\bullet} ="78",
            (12,1) *+{\bullet} ="79",
            (14,0) *+{\bullet} ="89",
            (14,1) *+{\bullet} ="8X",
            (16,0) *+{\bullet} ="9X",
            (16,1) *+{\bullet} ="9Y",
            (18,0) *+{\bullet} ="XY",
            (18,1) *+{\bullet} ="XZ",
            (20,0) *+{\bullet} ="YZ",
            (20,1) *+{\bullet} ="YU",
            (22,0) *+{\bullet} ="ZU",
            (22,1) *+{\bullet} ="ZV",
            (24,0) *+{\bullet} ="UV",
            "12", {\ar "13"},
            "13", {\ar "14"},
            "23", {\ar "24"},
            "24", {\ar "25"},
            "34", {\ar "35"},
            "35", {\ar "36"},
            "45", {\ar "46"},
            "46", {\ar "47"},
            "56", {\ar "57"},
            "57", {\ar "58"},
            "67", {\ar "68"},
            "68", {\ar "69"},
            "78", {\ar "79"},
            "89", {\ar "8X"},
            "9X", {\ar "9Y"},
            "XY", {\ar "XZ"},
            "YZ", {\ar "YU"},
            "ZU", {\ar "ZV"},
             "13", {\ar "23"},
             "14", {\ar "24"},
             "24", {\ar "34"},
             "25", {\ar "35"},
             "35", {\ar "45"},
             "36", {\ar "46"},
             "46", {\ar "56"},
             "47", {\ar "57"},
             "57", {\ar "67"},
             "58", {\ar "68"},
             "68", {\ar "78"},
             "69", {\ar "79"},
             "79", {\ar "89"},
             "8X", {\ar "9X"},
             "9Y", {\ar "XY"},
             "XZ", {\ar "YZ"},
             "YU", {\ar "ZU"},
             "ZV", {\ar "UV"},
        \end{xy}
    \end{equation*}

    The Auslander-Reiten quiver of a distinguished $2\mathbb{Z}$-cluster tilting subcategory $\mathcal{M} \subseteq\modn A$ is
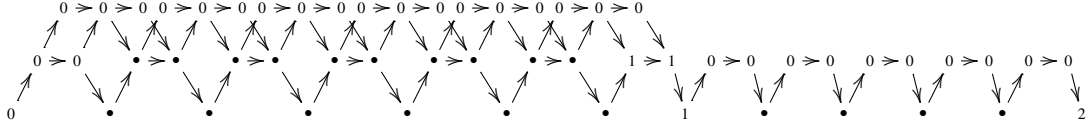
    \begin{equation*}
{\tiny
\begin{xy}
0;<15pt,0cm>:<10pt,20pt>:: 
(0,0) *+{0} ="123",
(0,1) *+{0} ="124",
(0,2) *+{0} ="125",
(1,1) *+{0} ="134",
(1,2) *+{0} ="135",
(2,2) *+{0} ="145",
(2.5,0) *+{\bullet} ="234",
(2.5,1) *+{\bullet} ="235",
(2.5,2) *+{0} ="236",
(3.5,1) *+{\bullet} ="245",
(3.5,2) *+{0} ="246",
(4.5,2) *+{0} ="256",
(5,0) *+{\bullet} ="345",
(5,1) *+{\bullet} ="346",
(5,2) *+{0} ="347",
(6,1) *+{\bullet} ="356",
(6,2) *+{0} ="357",
(7,2) *+{0} ="367",
(7.5,0) *+{\bullet} ="456",
(7.5,1) *+{\bullet} ="457",
(7.5,2) *+{0} ="458",
(8.5,1) *+{\bullet} ="467",
(8.5,2) *+{0} ="468",
(9.5,2) *+{0} ="478",
(10,0) *+{\bullet} ="567",
(10,1) *+{\bullet} ="568",
(10,2) *+{0} ="569",
(11,1) *+{\bullet} ="578",
(11,2) *+{0} ="579",
(12,2) *+{0} ="589",
(12.5,0) *+{\bullet} ="678",
(12.5,1) *+{\bullet} ="679",
(12.5,2) *+{0} ="67X",
(13.5,1) *+{\bullet} ="689",
(13.5,2) *+{0} ="68X",
(14.5,2) *+{0} ="69X",
(15,0) *+{\bullet} ="789",
(15,1) *+{1} ="78X",
(16,1) *+{1} ="79X",
(17,0) *+{1} ="89X",
(17,1) *+{0} ="89Y",
(18,1) *+{0} ="8XY",
(19,0) *+{\bullet} ="9XY",
(19,1) *+{0} ="9XZ",
(20,1) *+{0} ="9YZ",
(21,0) *+{\bullet} ="XYZ",
(21,1) *+{0} ="XYU",
(22,1) *+{0} ="XZU",
(23,0) *+{\bullet} ="YZU",
(23,1) *+{0} ="YZV",
(24,1) *+{0} ="YUV",
(25,0) *+{\bullet} ="ZUV",
(25,1) *+{0} ="ZUW",
(26,1) *+{0} ="ZVW",
(27,0) *+{2} ="UVW",
"123", {\ar"124"},
"124", {\ar"125"},
"124", {\ar"134"},
"125", {\ar"135"},
"134", {\ar"135"},
"134", {\ar"234"},
"135", {\ar"145"},
"135", {\ar"235"},
"145", {\ar"245"},
"234", {\ar"235"},
"235", {\ar"236"},
"245", {\ar"246"},
"235", {\ar"245"},
"236", {\ar"246"},
"246", {\ar"256"},
"245", {\ar"345"},
"246", {\ar"346"},
"256", {\ar"356"},
"345", {\ar"346"},
"346", {\ar"347"},
"356", {\ar"357"},
"346", {\ar"356"},
"347", {\ar"357"},
"357", {\ar"367"},
"356", {\ar"456"},
"357", {\ar"457"},
"367", {\ar"467"},
"456", {\ar"457"},
"457", {\ar"458"},
"467", {\ar"468"},
"457", {\ar"467"},
"458", {\ar"468"},
"468", {\ar"478"},
"467", {\ar"567"},
"468", {\ar"568"},
"478", {\ar"578"},
"567", {\ar"568"},
"568", {\ar"569"},
"578", {\ar"579"},
"568", {\ar"578"},
"569", {\ar"579"},
"579", {\ar"589"},
"578", {\ar"678"},
"579", {\ar"679"},
"589", {\ar"689"},
"678", {\ar"679"},
"679", {\ar"67X"},
"689", {\ar"68X"},
"679", {\ar"689"},
"67X", {\ar"68X"},
"68X", {\ar"69X"},
"689", {\ar"789"},
"68X", {\ar"78X"},
"69X", {\ar"79X"},
"789", {\ar"78X"},
"78X", {\ar"79X"},
"79X", {\ar"89X"},
"89X", {\ar"89Y"},
"89Y", {\ar"8XY"},
"8XY", {\ar"9XY"},
"9XY", {\ar"9XZ"},
"9XZ", {\ar"9YZ"},
"9YZ", {\ar"XYZ"},
"XYZ", {\ar"XYU"},
"XYU", {\ar"XZU"},
"XZU", {\ar"YZU"},
"YZU", {\ar"YZV"},
"YZV", {\ar"YUV"},
"YUV", {\ar"ZUV"},
"ZUV", {\ar"ZUW"},
"ZUW", {\ar"ZVW"},
"ZVW", {\ar"UVW"},
\end{xy}
}
\end{equation*}

 There is a $10\mathbb{Z}$-cluster tilting subcategory $\mathcal{C} \subseteq\mathcal{M}$. 
 Here the label $i$ refers to $\add \tau_{10}^{-i}A \in \mathcal{C}$.
\end{Eg}

\subsubsection{Cyclic case}\label{sec_cyclic_nonself-inj_case}

For a non-self-injective cyclic $d$-Nakayama algebra $A=A_{\underline{\ell}}^d$, 
we firstly deduce a necessary condition for $\modn A$ to admit an $nd\mathbb{Z}$-cluster tilting subcategory from the corresponding result on its self-degluing.
Further, we show that the necessary condition is indeed sufficient by constructing a gluing system as discussed in \cite{Vas20} and applying an orbit construction as in Section \ref{preliminary_orbit_construction}.

\begin{Prop}\label{prop_necessary_condition_of_cyclic_non_selfinj}
    Let $A=A_{\underline{\ell}}^{d}$ be a cyclic non-self-injective $d$-Nakayama algebra. 
    Assume that $\modn A$ admits a partial $nd\mathbb{Z}$-cluster tilting subcategory.
    Then the following conditions hold true.
    \begin{itemize}
        \item [(i)] $A= (A_1\Delta A_2\Delta \cdots \Delta A_t)^c$ with $A_i=A_{\ell_i,m_i}^{d}$ for $1\leq i\leq t$.
        \item [(ii)] For $1\leq i\leq t$, one of the following conditions holds true.
        \begin{itemize}
            \item [(a)] $\ell_i=2$ and $n=m_i$.
            \item [(b)] $(d+1)|(n-2)$ and $m_i=\frac{n-2}{d+1}(\ell_i+d-1)+\ell_i$.
        \end{itemize}
    \end{itemize}
\end{Prop}
\begin{proof}
    By Corollary \ref{cor_necessary_shape_of_Kupisch_series}, 
    there is a self-deglue point of $A$.
    Let $A^a = A_{\underline{\ell}^a}^d$ be the self-degluing of $A$.
    Then by Lemma \ref{lem_partial_ndZ_ct_of_A_gives_that_of_A^a},
    $\modn A^a$ admits a partial $nd\mathbb{Z}$-cluster tilting subcategory $\mathcal{C}$.
    By Corollary \ref{cor_acyclic_ndZ_ct_iff_weakly_ndZ_ct},
    $\mathcal{C}\subseteq \modn A^a$ is indeed $nd\mathbb{Z}$-cluster tilting. 
    Hence the statements follow from Theorem \ref{thm_classification_nonhomo_acyclic}.
\end{proof}

Next we show, using gluing systems \cite{Vas20} and the orbit construction recalled in Section \ref{preliminary_orbit_construction}, that the necessary condition given in Proposition  \ref{prop_necessary_condition_of_cyclic_non_selfinj} is in fact sufficient for $nd\mathbb{Z}$-cluster tilting.
This will allow us to obtain a classification of $nd\mathbb{Z}$-cluster tilting subcategories of the cyclic non-self-injective $d$-Nakayama algebra $A$.
For the rest of the section,
we assume $A$ satisfies (i) and (ii) in Proposition \ref{prop_necessary_condition_of_cyclic_non_selfinj} and denote by $A^a$ the self-degluing of $A$.
By Theorem \ref{thm_classification_nonhomo_acyclic},
$\modn A^a$ admits a unique $nd\mathbb{Z}$-cluster tilting subcategory $\mathcal{C}^a$.

We begin with the definition of infinite sink-source self-gluing of $A^a$ following \cite[Section 5.2]{Vas20}.
Note that $A^a = \Bbbk Q/\mathcal{R}$ where $Q$ has a unique sink and a unique source as shown below.
\begin{figure}[h]
    \centering
    \begin{tikzpicture}
        \draw (0,0) ellipse (0.5cm and 0.7cm);
        \draw (-1,0) node [scale=0.2, circle, draw,fill=black]{};
        \draw (-1.3,0) node {$a$};
        \draw[->] (-0.9,0.05)--(-0.5, 0.4);
        \draw (1,0) node [scale=0.2, circle, draw,fill=black]{};
        \draw (1.3,0) node {$b$};
        \draw[->] (0.5, 0.4)--(0.9,0.05);
        \node[] at (0,0) {$Q'$};
    \end{tikzpicture}
    \caption{quiver $Q$}
    \label{fig:placeholder}
\end{figure}
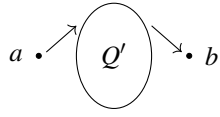

For each integer $z\in \mathbb{Z}$, let $Q[z]$ be a copy of $Q$ with vertices and arrows labelled as follows.
\begin{align*}
    (Q[z])_0 & = \{i[z]\mid i\in Q_0\}\\
    (Q[z])_1 & = \{\alpha[z]:i[z]\rightarrow j[z]\mid \alpha\in Q_1 \text{ with } \alpha: i\rightarrow j\}.
\end{align*}
Denote by $\mathcal{R}[z]$ the ideal of $kQ[z]$ corresponding to $\mathcal{R}$ of $kQ$.
Let 
\begin{equation*}
    \Lambda_z = \Bbbk Q[z]/\mathcal{R}[z]
\end{equation*}
and 
\begin{equation*}
    S_z = S_a[z], \text{ and } T_z = S_b[z]
\end{equation*}
the simple $\Lambda_z$ modules at vertices $a[z]$ and $b[z]$ correspondingly.
Denote by $\mathcal{C}_z\subseteq \modn \Lambda_z$ the $nd\mathbb{Z}$-cluster tilting subcategory corresponding to $\mathcal{C}^{a}$ of $\modn A^{a}$.

Let $G=\overrightarrow{A_{\infty}^{\infty}}$ be the following graph.
\begin{equation*}
    \xymatrix@=1em@R=1em{\cdots\ar[r] & -2\ar[r]^{\alpha_{-2}} & -1\ar[r]^{\alpha_{-1}} & 0\ar[r]^{\alpha_{0}} & 1\ar[r]^{\alpha_{1}} & 2\ar[r]^{\alpha_{-1}} & \cdots}.
\end{equation*}
Following the notations in \cite[Definition 4.8]{Vas20}, we have that $(\Lambda_z, S_z, T_z)_{z\in \mathbb{Z}}$ forms a gluing system on $G$.

Denote by $I=\{s, s+1, \ldots, t\}\subseteq \mathbb{Z}$ the finite interval subset of $\mathbb{Z}$. 
Let 
\begin{equation*}
    \Lambda_I = \Lambda_s\Delta\cdots \Delta \Lambda_t = \Bbbk Q_I/\mathcal{R}_I
\end{equation*}
be the sink-source gluing along $I$.
For $I\subseteq J$ we have a ring  epimorphism 
\begin{equation*}
    \Lambda_J\rightarrow \Lambda_J/\langle 1-e_I\rangle = \Lambda_I
\end{equation*}
where $e_I=\sum_{i\in (Q_I)_0}e_i$ which induces a fully faithful embedding 
\begin{equation*}
    \iota_{IJ}: \modn \Lambda_I\rightarrow \modn \Lambda_J.
\end{equation*}
Similarly as discussed in Section \ref{preliminary_gluing}, $\iota_{IJ}$ admits a left adjoint $L_{IJ}$ and a right adjoint $R_{IJ}$.
Moreover, we denote by 
\begin{equation*}
    \theta_{IJK}: \iota_{IJ}\circ \iota_{JK}\Rightarrow \iota_{IK}
\end{equation*}
the natural isomorphism for $I,J,K\subseteq \mathbb{Z}$.

Let $\Lambda=\Bbbk \dot{Q}/\dot{\mathcal{R}}$ where $\dot{Q}$ is given in Figure \ref{fig:quiver Q dot} and $\dot{\mathcal{R}}$ is generated by $\mathcal{R}[z]$ and the paths from $Q'[z]$ to $Q'[z+1]$ for $z\in \mathbb{Z}$. 

\begin{figure}[h]
    \centering
    \begin{tikzpicture}
        \draw (0,0) ellipse (0.5cm and 0.7cm);
        \draw (-1,0) node [scale=0.2, circle, draw,fill=black]{};
        \draw (-1.3,0) node {$\cdots$};
        \draw[->] (-0.9,0.05)--(-0.5, 0.4);
        \draw (1,0) node [scale=0.2, circle, draw,fill=black]{};
        \draw[->] (0.5, 0.4)--(0.9,0.05);
        \node[] at (0,0) {$Q'[-1]$};
        \draw (2,0) ellipse (0.5cm and 0.7cm);
        \node[] at (2,0) {$Q'[0]$};
        \draw (1,0) node [scale=0.2, circle, draw,fill=black]{};
        \draw[->] (1.1,0.05)--(1.5, 0.4);
        \draw (3,0) node [scale=0.2, circle, draw,fill=black]{};
        \draw[->] (2.5, 0.4)--(2.9,0.05);
        \draw[->] (3.1,0.05)--(3.5, 0.4);
        \draw (4,0) ellipse (0.5cm and 0.7cm);
        \node[] at (4,0) {$Q'[1]$};
        \draw (5,0) node [scale=0.2, circle, draw,fill=black]{};
        \draw[->] (4.5, 0.4)--(4.9,0.05);
        \node[] at (5.3,0) {$\cdots$};
    \end{tikzpicture}
    \caption{quiver $\dot{Q}$}
    \label{fig:quiver Q dot}
\end{figure}

Following the steps in \cite[Section 4.1.2]{Vas20}, $(\proj\Lambda, \Phi_I, \Theta_{IJ})$ is a firm source of the \textbf{Cat}-inverse system $(\proj \Lambda_I, \iota_{IJ}, \theta_{IJK})$, see \cite[Proposition 4.11, Corollary 4.15]{Vas20},
where $\Phi_I: \proj \Lambda\rightarrow \proj \Lambda_I$ corresponds to the canonical projection $ \Lambda\rightarrow \Lambda_I$ and 
$\Theta_{IJ}$ is the canonical natural isomorphism from $ \iota_{IJ}\circ \Phi_J$ to $\Phi_I$.

Consequently $(\modn \Lambda, \Phi_{I\ast}, \Theta_{IJ\ast})$ is an admissible target \cite[Definition 3.6]{Vas20} of $(\modn \Lambda_I, \iota_{IJ\ast}, \Theta_{IJK\ast})$ by \cite[Proposition 4.19]{Vas20} where $\Phi_{I\ast}: \modn \Lambda_I\rightarrow \modn \Lambda$ is the induced embedding.
Notice that representation directedness is assumed in \cite[Section 4]{Vas20} however the proofs still work without this assumption.

Recall that $\mathcal{C}^a\subseteq \modn A^a$ is the $nd\mathbb{Z}$-cluster tilting subcategory obtained from Theorem \ref{thm_classification_nonhomo_acyclic}.
Applying Proposition \ref{prop_glue_d_ct_2_d_ct}, 
we have 
\begin{equation*}
    \mathcal{C}_I = \mathcal{C}^a_s \Delta \cdots \Delta \mathcal{C}^a_t \subseteq \modn \Lambda_I
\end{equation*}
is an $nd\mathbb{Z}$-cluster tilting subcategory.
Since each $\mathcal{C}_J\subseteq \modn \Lambda_J$ is $nd$-cluster tilting and $\iota_{IJ\ast}(\modn \Lambda_I) \cap \mathcal{C}_J = \iota_{IJ\ast}(\mathcal{C}_I)$ for all $I\subseteq J$,
we get that 
$(\mathcal{C}_I)_I$ is an asympotically weakly $nd$-cluster tilting system (see \cite[Definition 3.18]{Vas20}).
As a consequence, we have the following proposition.

\begin{Prop}
    Let 
    \begin{equation*}
        \mathcal{C} = \add \{\Phi_{I\ast}(\mathcal{C}_I)\mid I\subseteq \mathbb{Z}\} \subseteq \modn \Lambda
    \end{equation*}
    is an $nd\mathbb{Z}$-cluster tilting subcategory.
\end{Prop}
\begin{proof}
    Note that $\mathcal{C}_I$ is functorially finite as it admits an additive generator by Theorem \ref{thm_classification_nonhomo_acyclic}.
    Applying \cite[Theorem 3.19, Corollary 3.21]{Vas20},
    we get that $\mathcal{C}\subseteq \modn \Lambda$ is $nd$-cluster tilting.

    Now we show $\Omega_{\Lambda}^{nd}\mathcal{C}\subseteq \mathcal{C}$ to conclude that $\mathcal{C}$ is $nd\mathbb{Z}$-cluster tilting.

    Let $M\in \mathcal{C}$.
    We may view $M\in \mathcal{C}_I$ for some sufficiently large interval $I\subseteq \mathbb{Z}$ such that $\Omega_{\Lambda}^{nd}M\cong \Omega_{\Lambda_I}^{nd}M$.
    Since $\mathcal{C}_I\subseteq \modn \Lambda_I$ is $nd\mathbb{Z}$-cluster tilting, 
    $\Omega_{\Lambda_I}^{nd}M\in \mathcal{C}_I$.
    Hence $\Omega_{\Lambda}^{nd}M\in \mathcal{C}$.
\end{proof}

We view $\Lambda = \Bbbk\dot{Q}/\dot{\mathcal{R}}$ as a path category. 
For $k\in \mathbb{Z}$, the automorphism 
\begin{align*}
    [k]: \Bbbk \dot{Q}/\dot{\mathcal{R}} & \rightarrow \Bbbk \dot{Q}/\dot{\mathcal{R}}\\
    p & \mapsto p[k]
\end{align*}
induces a $\mathbb{Z}$-action on $\Lambda$. 

Let $F: \Lambda \rightarrow \Lambda /\mathbb{Z}$ and let $F_{\ast}: \modn \Lambda \rightarrow \modn (\Lambda/\mathbb{Z})$ be the push-down functor. 

\begin{Prop}
    We have that $F_{\ast}(\mathcal{C})\subseteq \modn (\Lambda/\mathbb{Z})$ is an $nd\mathbb{Z}$-cluster tilting subcategory.
\end{Prop}
\begin{proof}
    Viewing objects in $\mathcal{C}$ as representations of the quiver $\dot{Q}$ bounded by $\dot{R}$,
    it is straightforward to check that $\mathcal{C}$ is $\mathbb{Z}$-equivariant.
    The proof is parallel to \cite[Lemma 5.9]{Vas20}.
    
    Applying Theorem \ref{thm_DI_orbit_construction}, we have $F_{\ast}(\mathcal{C})\subseteq \modn (\Lambda/\mathbb{Z})$ is $nd$-cluster tilting.
    Since $F_{\ast}$ admits an exact right adjoint, 
    $F_{\ast}$ preserves projective objects. 
    We have the following commutative diagram.
    \begin{equation*}
        \xymatrix{
        \modn \Lambda \ar[r]^{\Omega_{\Lambda}^{nd}}\ar[d]^{F_{\ast}} & \modn \Lambda\ar[d]^{F_{\ast}} \\
        \modn (\Lambda/\mathbb{Z})
        \ar[r]^{\Omega_{\Lambda/\mathbb{Z}}^{nd}} & \modn (\Lambda/\mathbb{Z})}
    \end{equation*}
    Since $\Omega_{\Lambda}^{nd}\mathcal{C} \subseteq \mathcal{C}$,
    we have 
    \begin{equation*}
        \Omega_{\Lambda/\mathbb{Z}}^{nd}F_{\ast}\mathcal{C} \cong  F_{\ast}\Omega_{\Lambda}^{nd} \mathcal{C} \subseteq F_{\ast}\mathcal{C}
    \end{equation*}
    which implies 
    $F_{\ast}\mathcal{C}\subseteq \modn (\Lambda/\mathbb{Z})$ is $nd\mathbb{Z}$-cluster tilting.
\end{proof}

Following a similar argument as in \cite[Section 5.2]{Vas20}, $\modn (\Lambda/\mathbb{Z})$ can be realized as the module category of a finite dimensional algebra which is $A$ in our case. 

\begin{Cor}\label{cor_existence_ndZ_ct_for_non_selfinj_cyclic_case}
    There is an equivalence $\modn A\simeq \modn (\Lambda/\mathbb{Z})$. 
    In particular, $\modn A$ admits an $nd\mathbb{Z}$-cluster tilting subcategory given by $ F_{\ast}(\mathcal{C})$.
\end{Cor}

Now we are ready to classify $nd\mathbb{Z}$-cluster tilting subcategories for cyclic non-self-injective $d$-Nakayama algebras.

\begin{Thm}\label{thm_classification_non_selfinj_cyclic_case}
    Let $A=A_{\underline{\ell}}^{d}$ be a cyclic non-self-injective $d$-Nakayama algebra.
    Then $\modn A$ admits an $nd\mathbb{Z}$-cluster tilting subcategory $\mathcal{C}$ if and only if 
    the following conditions hold true.
    \begin{itemize}
        \item [(i)] $A= (A_1\Delta A_2\Delta \cdots \Delta A_t)^c$ with $A_i=A_{\ell_i,m_i}^{d}$ for $1\leq i\leq t$.
        \item [(ii)] For $1\leq i\leq t$, one of the following conditions holds true.
        \begin{itemize}
            \item [(a)] $\ell_i=2$ and $n=m_i$.
            \item [(b)] $(d+1)|(n-2)$ and $m_i=\frac{n-2}{d+1}(\ell_i+d-1)+\ell_i$.
        \end{itemize}
    \end{itemize}
    Moreover, in that case  
    $\mathcal{C} = \add (\{A \oplus DA\} \cup \{S_i\mid 1\leq i\leq t\})$ where 
    \begin{equation*}
        S_i = M(1 + \sum_{k=1}^i m_i, 2+ \sum_{k=1}^i m_i, \ldots, d+1+\sum_{k=1}^i m_i).
    \end{equation*}
\end{Thm}
\begin{proof}
    Condition (i) and (ii) follow by 
    Proposition \ref{prop_necessary_condition_of_cyclic_non_selfinj}.
    Assuming (i) and (ii), $\mathcal{C}$ is 
   of the given form by Corollary \ref{cor_existence_ndZ_ct_for_non_selfinj_cyclic_case}.
    
\end{proof}

\subsection{Self-injective case}

Recall that a $d$-Nakayama algebra $A_{\underline{\ell}}^{d}$ of type $\Tilde{\mathbb{A}}_{m-1}$ is self-injective if and only if  $\underline{\ell} = (\ell_1, \ldots, \ell_m)$ with $\ell_i=\ell$ for some positive integer $\ell\geq 2$.
In this case, we use the notation $A = \Tilde{A}_{\ell, m-1}^{d}$ following \cite[Definition 4.9]{JK19}.
We denote by $\mathcal{M}\subseteq \modn A$ the distinguished $d\mathbb{Z}$-cluster tilting subcategory of $A$.

In this section, we firstly give a classification of $nd\mathbb{Z}$-cluster tilting subcategories of $A$  when $\ell=2$.
In this case, $A$ is a classical Nakayama algebra so the result in \cite{HKV25} applies.
If $\ell\geq 3$,
we obtain necessary conditions for $A$ to admit an $nd\mathbb{Z}$-cluster tilting subcategory.
However, to get a classification result in this case seems rather difficult. 
Instead, we give an explicit example of an $(\ell-2)d\mathbb{Z}$-cluster tilting subcategory for each $A$ with $\gcd(\ell-2, d)=1$.

\begin{Prop}\label{prop_classification_selfinj_ell_2}
    Let $A = \widetilde{A}_{2, m-1}^d$ be a self-injective $d$-Nakayama algebra. 
    Then $\modn A$ admits an $nd\mathbb{Z}$-cluster tilting subcategory $\mathcal{C}$ if and only if $n|(m-1)$.
    In this case, 
    \begin{align*}
        \mathcal{C} & = \add(\{A\}\cup \{\tau_{nd}^kM(i, \ldots, i+d)\mid k\in\mathbb{Z}\})\\
        & = \add(\{A\}\cup \{M(i+knd, \ldots, i+(kn+1)d)\mid k\in \mathbb{Z}\}) \text{ for some } i\in \mathbb{Z}.
    \end{align*}
\end{Prop}
\begin{proof}
    Note that $A$ is a classical Nakayama algebra with $\rad^2 A=0$.
    The statement follows from  \cite[Proposition 4.9]{HKV25}.
\end{proof}

From now on, we assume $\ell\geq 3$. 
First we recall the orbit construction that underlies $A$. 
Denote by $A'= A_{\ell-1}^{d}$ the $(d-1)$-Auslander algebra of type $\mathbb{A}_{\ell-1}$.
Note that $A'$ is the acyclic $d$-Nakayama algebra defined by the Kupisch series $(1,2,\ldots, \ell-1)$.
Denote by 
\begin{equation*}
    \mathcal{M}' = \add\{\tau_d^{-i}A'\mid i\geq 0\} \subseteq \modn A'
\end{equation*}
the distinguished, indeed unique $d\mathbb{Z}$-cluster tilting subcategory of $\modn A'$.
We refer to \cite{Iya11}, \cite[Section 3]{OT12} and \cite[Section 2.1]{JK19} for more details.

Let
\begin{equation*}
    \nu = D\circ \mathbb{R}\Hom_{A'}(-,A') \cong -\otimes^{\mathbb{L}}_{A'} DA' : \mathcal{D}^b(A') \rightarrow \mathcal{D}^b(A') 
\end{equation*}
be the Nakayama functor of $\mathcal{D}^b(A')$. 
Denote the $d$-Nakayama functor by 
\begin{equation*}
    \nu_d = \nu\circ [-d]: \mathcal{D}^b(A')\rightarrow \mathcal{D}^b(A').
\end{equation*}
Then 
\begin{equation*}
    \mathfrak{M} = \add \{M[di]\mid M\in \mathcal{M}',  i\in \mathbb{Z}\} = \add \{\nu_d^i(A)\mid i\in \mathbb{Z}\} \subseteq \mathcal{D}^b(A')
\end{equation*}
is a $d\mathbb{Z}$-cluster tilting subcategory induced by $\mathcal{M}'$, see \cite{Iya11}.

Moreover, we denote by $\widehat{A'}$ the repetitive algebra (see \cite[II.2]{Hap88}) of $A'$. By \cite[Theorem II.4.9]{Hap88}, there is a triangle equivalence 
\begin{equation*}
    \mathcal{D}^b(A') \cong \underline{\modn } \widehat{A'}.
\end{equation*}

It was shown in \cite[Proposition 3.20]{JK19} that
\begin{equation*}
    \widehat{A_{\ell-1}^{d}} \cong A_{\mathbb{Z}\ell}^{d}
\end{equation*}
where $A_{\mathbb{Z}\ell}^{d}$ is the $d$-Nakayama algebra of type $\mathbb{A}_{\infty}^{\infty}$ defined by the Kupisch series $\mathbb{Z}\ell = (\ldots, \ell, \ell, \ldots)$. 
We denote by 
\begin{equation*}
    \mathcal{M}_{\mathbb{Z}\ell}^d\subseteq \modn A_{\mathbb{Z}\ell}^d
\end{equation*}
the distinguished $d\mathbb{Z}$-cluster tilting subcategory which corresponds to $\mathfrak{M}$.
We may view $A_{\mathbb{Z}\ell}^d$ as a category. Note that 
\begin{equation*}
    \sigma: A_{\mathbb{Z}\ell}^d\rightarrow A_{\mathbb{Z}\ell}^d, (x_1, \ldots, x_d)\mapsto (x_1+m, \ldots, x_d+m)
\end{equation*}
defines an $H$-action on $A_{\mathbb{Z}\ell}^d$ where $H=\langle \sigma\rangle$ is the group generated by $\sigma$.
Denote by 
\begin{equation*}
    F_{\ast}: \modn A_{\mathbb{Z}\ell}^d\rightarrow  \modn A_{\mathbb{Z}\ell}^d/H
\end{equation*}
the push-down functor.
Moreover $H(\mathcal{M}_{\mathbb{Z}\ell}^d)\subseteq \mathcal{M}_{\mathbb{Z}\ell}^d$ as $\sigma M(x) \cong \tau_d^m M(x)$.

\begin{Def}(\cite[Definition 4.9]{JK19}) We have 
    $A = A_{\mathbb{Z}\ell}^{d}/H$
    and 
    $\mathcal{M} = \mathcal{M}_{\mathbb{Z}\ell}^d/H$.
\end{Def}

In general, assume that $\mathfrak{U}\subseteq \mathcal{D}^b(A')$ is a $q$-cluster tilting (respectively $q\mathbb{Z}$-cluster tilting) subcategory which satisfies  $H(\mathfrak{U})\subseteq \mathfrak{U}$,
then the above construction gives rise to a $q$-cluster tilting (respectively $q\mathbb{Z}$-cluster tilting) subcategory $\underline{\mathcal{U}} = \mathfrak{U}/H \subseteq \underline{\modn} A$ as illustrated in the following diagram \cite{DI20}
\begin{equation}\label{diagram}
    \xymatrix@R=.4em@C=1em{
    \mathcal{D}^b(A')\ar[r]^{\sim} & \underline{\modn } \widehat{A'}\ar[r]^{\sim} &  \underline{\modn} A_{\mathbb{Z}\ell}^{d}\ar[r]^{F_{\ast}} & \underline{\modn} A\\
    \cup &&& \cup\\
    \mathfrak{U}\ar@{-->}[rrr] &&& \underline{\mathcal{U}}
    }
\end{equation}

Recall from \cite[Chapter IV, Proposition 3.7]{ARS95} that  
for a self-injective algebra $A$,
we have 
\begin{equation*}
    \tau \cong \mathcal{N} \circ \Omega^2 \cong \Omega^2 \circ \mathcal{N},
\end{equation*}
on $\underline{\modn} A$ 
where $\mathcal{N} = D\Hom_A(-, A)$ denotes the Nakayama automorphism on $\modn A$.
Therefore 
\begin{equation*}
   \tau\circ \Omega \cong \Omega \circ \tau 
\end{equation*}
on $\underline{\modn } A$.

\begin{Lem}\label{lem_cyclic_restriction_n}
    Let $\mathcal{C}\subseteq \modn A$ be an $nd\mathbb{Z}$-cluster tilting subcategory.
    Let $k\in\mathbb{Z}$ be an integer, and assume $M\in \mathcal{C}_P$.
    Then $\tau_d^{-k}M\in \mathcal{C}_P$ if and only if $n\mid k$.
\end{Lem}
\begin{proof}
    First note that $\mathcal{C}$ is closed under $\Omega^d\circ \tau_d^{-1} \cong \Omega^{nd}\circ \tau_{nd}^{-1}$ and 
    \begin{equation*}
        \tau_d^{-k}M \cong \Omega^{-dk}(\Omega^d\tau_d^{-1})^kM
    \end{equation*}
    as $\tau$ and $\Omega$ commute.
    So $M' = (\Omega^d\tau_d^{-1})^kM\in \mathcal{C}$.

    For $k = n$ we get $\tau_d^{-k}M \cong \Omega^{-nd}M'\in \mathcal{C}_P$ by definition.

    For $1\leq k\leq n-1$, we get 
    \begin{equation*}
        \Ext_A^{dk}(\tau_d^{-k}M, M')\cong \underline{\Hom}_A(\Omega^{dk}\tau_d^{-k}M, M') \cong \underline{\Hom}_A(M', M') \neq 0
    \end{equation*}
    which implies $\tau_d^{-k}M\notin \mathcal{C}_P$.

    For $k\geq 0$ the claim follows by induction. 
    The case $k\leq 0$ is similar.
    
\end{proof}

Recall that the Serre functor on $\underline{\modn} A$ is given by $\mathbb{S} = \tau_d\Omega^{-d}$.
Moreover, $\underline{\modn} A$ is $\frac{(\ell-2)d}{\ell+d-1}$-Calabi-Yau as $\mathcal{D}^b(A')$ is $\frac{(\ell-2)d}{\ell+d-1}$-Calabi-Yau (see \cite[Remark 2.29]{DJW19}, \cite[Theorem 6.2]{Gra23}).
Thus we have  $\mathbb{S}^{\ell+d-1}\cong \Omega^{-(\ell-2)d}$ which implies 
\begin{equation*}
    \tau_d^{\ell+d-1} \cong \Omega^{-d(d+1)}.
\end{equation*}
on $\underline{\modn} A$.
We obtain the following necessary condition for $\modn A$ to admit an $nd\mathbb{Z}$-cluster tilting subcategory.

\begin{Cor}\label{cor_necessary_condition_selfinj}
    If $\modn A$ admits an $nd\mathbb{Z}$-cluster tilting subcategory, then $n\mid m$ and $n\mid (\ell-2)$.
\end{Cor}
\begin{proof}
    Note that $\tau_d^{-m}X \cong X$ for all $X\in \underline{\modn} A$.
    In fact we have $\tau_d^{-m}M(x)=M(x')\cong M(x)$ as $(x_i \equiv x'_i) \mod m$ for $1\leq i\leq d+1$.
    Since $\tau_d^{-m}$ is a triangle functor, the claim follows.
    
    Let $\mathcal{C}$ be an $nd\mathbb{Z}$-cluster tilting subcategory of $\modn A$.
    Since $\mathcal{C}_P\neq \emptyset$, there is $M\in \mathcal{C}_P$.
    Then $\tau_d^{-m}M = M\in \mathcal{C}_P$ together with Lemma \ref{lem_cyclic_restriction_n} implies $n\mid m$.

    As we have seen before that $\mathcal{C}$ is closed under $\Omega^{-d}\circ \tau_d$, 
    we have $M' = (\Omega^{-d}\tau_d)^{d+1}M \in \mathcal{C}_P$.
    Moreover, 
    \begin{equation*}
        M' = \tau_d^{d+1}\Omega^{-d(d+1)}M = \tau_d^{d+1}\tau_d^{-(\ell+d-1)}M = \tau_d^{2-\ell}M.
    \end{equation*}
    Again by Lemma \ref{lem_cyclic_restriction_n}, we have $n\mid (\ell - 2)$.
\end{proof}

We will not give a full classification of $nd\mathbb{Z}$-cluster tilting subcategories of $\modn A$.
Instead we focus on a special case appearing in Corollary \ref{cor_necessary_condition_selfinj}.
Specifically we assume additionally that $n = \ell-2$ and $m=kn=k(\ell-2)$ for some positive integer $k$.
We give an example of an $(\ell-2)d\mathbb{Z}$-cluster tilting subcategory of $A$ in the case when $\gcd(\ell-2, d)=1$.

We recall the following proposition which shows that there exists an $(\ell-2)d\mathbb{Z}$-cluster tilting subcategory of $\mathcal{D}^b(A')$ in the case when $\gcd(\ell-2, d)=1$.

\begin{Prop}\cite[Theorem 1.3]{Xin25}
    Assume $\gcd(\ell-2, d) = 1$. 
    There is a tilting complex $T \in \mathcal{D}^b(A')$  such that 
    \begin{equation*}
        \mathfrak{U} = \add \{\nu_{(\ell-2)d}^i(T)\mid i\in \mathbb{Z}\} \subseteq \mathcal{D}^b(A')
    \end{equation*}
    is an $(\ell-2)d\mathbb{Z}$-cluster tilting subcategory.
\end{Prop}

Next we show how $\mathfrak{U}$ gives an $(\ell-2)d\mathbb{Z}$-cluster tilting subcategory of $\modn A$. 

\begin{Prop}
    Assume $\gcd(\ell-2, d)=1$. We have that $\underline{\mathcal{U}} = F_{\ast}(\mathfrak{U})$ is an $(\ell-2)d\mathbb{Z}$-cluster tilting subcategory of $\underline{\modn} A$.
    Consequently, its preimage under $\modn A\rightarrow \underline{\modn} A$ 
    \begin{equation*}
        \mathcal{U} \subseteq\modn A
    \end{equation*}
    is an $(\ell-2)d\mathbb{Z}$-cluster tilting subcategory.
\end{Prop}
\begin{proof}
    We have the following commutative diagram, see \cite[Theorem 3.22]{JK19}.
\begin{equation*}
    \xymatrix{
    \mathcal{D}^b(A')\ar[r]^{\sim}\ar[d]^{\nu_d} & \underline{\modn} A_{\mathbb{Z}\ell}^{d}\ar[d]^{\tau_d} \\
    \mathcal{D}^b(A')\ar[r]^{\sim} & \underline{\modn} A_{\mathbb{Z}\ell}^{d}
    }
\end{equation*}
    To apply Theorem \ref{thm_DI_orbit_construction},  it suffices to show that $\nu_d^{m}(\mathfrak{U})\subseteq\mathfrak{U}$. 
    We claim that $\nu_d^{\ell-2}(\mathfrak{U})\subseteq\mathfrak{U}$. Then the statement follows since $(\ell-2)|m$.

    We have 
    \begin{equation*}
        \nu_d^{\ell-2} = (\nu\circ [-d])^{\ell-2} \cong  \nu^{\ell-2}\circ [-(\ell-2)d].
    \end{equation*}
    Since $\mathfrak{U}$ is $(\ell-2)d\mathbb{Z}$-cluster tilting, by \cite[Proposition 3.6]{IO13}, $\nu(\mathfrak{U})=\mathfrak{U}$.
    For the same reason, $\mathfrak{U}[(\ell-2)d]\subseteq \mathfrak{U}$.
    Thus $\nu_d^{\ell-2}(\mathfrak{U})= \mathfrak{U}$. 
    Therefore, $\underline{\mathcal{U}} = F_{\ast}(\mathfrak{U})$ is an $(\ell-2)d\mathbb{Z}$-cluster tilting subcategory of $\underline{\modn} A$.
    Since there is a one-to-one correspondence between the set of $(\ell-2)d\mathbb{Z}$-cluster tilting subcategories of $\underline{\modn}A$ and that of $\modn A$, we obtain the $(\ell-2)d\mathbb{Z}$-cluster tilting subcategory $\mathcal{U}$ of $\modn A$.
\end{proof}

\begin{Rem}
    The $(\ell-2)d\mathbb{Z}$-cluster tilting subcategory $\mathcal{U}\subseteq \modn A$ constructed above is included in $\mathcal{M}$. 
    Indeed, the tilting complex $T\in \mathcal{D}^b(A')$ is of the form 
    \begin{equation*}
        T = \bigoplus_{i=1}^{\ell+d-2}\nu^iP
    \end{equation*}
    where $P$ is a certain basic projective $A'$-module \cite[Definition 4.18]{Xin25}.
    As $\mathfrak{M}\subseteq\mathcal{D}^b(A')$ is $d\mathbb{Z}$-cluster tilting, 
    by \cite[Proposition 3.6]{IO13}, 
    $\nu(\mathfrak{M})=\mathfrak{M}$ which implies $\mathfrak{U}\subseteq \mathfrak{M}$.
    Applying the push-down functor, 
    we have $\mathcal{U}\subseteq \mathcal{M}$.
\end{Rem}

We illustrate the above construction in an example.

\begin{Eg}\label{ex_selfinj}
    Let $A = \widetilde{A}_{5,2}^2$. 
    The Gabriel quiver of $A$ is given as follows. 
    Here the leftmost line and rightmost line are identified.
    
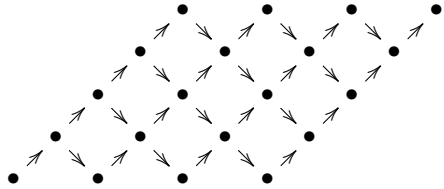
\begin{equation*}
        \begin{xy}
            0;<16pt,0cm>:<16pt,16pt>::
            (0,0) *+{\bullet} ="12",
            (0,1) *+{\bullet} ="13",
            (0,2) *+{\bullet} ="14",
            (0,3) *+{\bullet} ="15",
            (0,4) *+{\bullet} ="16",
            (2,0) *+{\bullet} ="23",
            (2,1) *+{\bullet} ="24",
            (2,2) *+{\bullet} ="25",
            (2,3) *+{\bullet} ="26",
            (2,4) *+{\bullet} ="27",
            (4,0) *+{\bullet} ="34",
            (4,1) *+{\bullet} ="35",
            (4,2) *+{\bullet} ="36",
            (4,3) *+{\bullet} ="37",
            (4,4) *+{\bullet} ="38",
            (6,0) *+{\bullet} ="45",
            (6,1) *+{\bullet} ="46",
            (6,2) *+{\bullet} ="47",
            (6,3) *+{\bullet} ="48",
            (6,4) *+{\bullet} ="49",
            "12", {\ar "13"},
            "13", {\ar "14"},
            "14", {\ar "15"},
            "15", {\ar "16"},
            "23", {\ar "24"},
            "24", {\ar "25"},
            "25", {\ar "26"},
            "26", {\ar "27"},
            "34", {\ar "35"},
            "35", {\ar "36"},
            "36", {\ar "37"},
            "37", {\ar "38"},
            "45", {\ar "46"},
            "46", {\ar "47"},
            "47", {\ar "48"},
            "48", {\ar "49"},
             "13", {\ar "23"},
             "14", {\ar "24"},
             "15", {\ar "25"},
             "16", {\ar "26"},
             "24", {\ar "34"},
             "25", {\ar "35"},
             "26", {\ar "36"},
             "27", {\ar "37"},
             "35", {\ar "45"},
             "36", {\ar "46"},
             "37", {\ar "47"},
             "38", {\ar "48"},
        \end{xy}
    \end{equation*}
    The Auslander-Reiten quiver of the distinguished $2\mathbb{Z}$-cluster tilting subcategory $\mathcal{M}\subseteq \modn A$ is given below. 
    Here the leftmost slice of triangle and the rightmost slice of triangle are identified.
    Note that the modules denoted by purple color are projective-injective modules. 
    There is a $6\mathbb{Z}$-cluster tilting subcategory $\mathcal{C} = \add C\subseteq \modn A$ where 
    \begin{equation*}
        C = A\oplus \bigoplus_{i=0}^2 \tau_6^{-i}(M(123)\oplus M(124)).
    \end{equation*}
    In the following quiver, for $0\leq i\leq 2$, red $i$ refers to $\tau_6^{-i}M(123)$ and blue $i$ refers to $\tau_6^{-i}M(124)$.
    \begin{equation*}
        \begin{xy}
            0;<20pt,0cm>:<10pt,20pt>:: 
            (0,0) *+{\textcolor{red}{0}} ="123",
            (0,1) *+{\textcolor{blue}{0}} ="124",
            (0,2) *+{\bullet} ="125",
            (0,3) *+{\textcolor{red}{2}} ="126",
            (0,4) *+{\textcolor{purple}{\bullet}} ="127",
            (1,1) *+{\bullet} ="134",
            (1,2) *+{\bullet} ="135",
            (1,3) *+{\textcolor{blue}{2}} ="136",
            (1,4) *+{\textcolor{purple}{\bullet}} ="137",
            (2,2) *+{\bullet} ="145",
            (2,3) *+{\bullet} ="146",
            (2,4) *+{\textcolor{purple}{\bullet}} ="147",
            (3,3) *+{\textcolor{red}{1}} ="156",
            (3,4) *+{\textcolor{purple}{\bullet}} ="157",
            (4,4) *+{\textcolor{purple}{\bullet}} ="167",
            (5,0) *+{\bullet} ="234",
            (5,1) *+{\bullet} ="235",
            (5,2) *+{\bullet} ="236",
            (5,3) *+{\bullet} ="237",
            (5,4) *+{\textcolor{purple}{\bullet}} ="238",
            (6,1) *+{\bullet} ="245",
            (6,2) *+{\bullet} ="246",
            (6,3) *+{\bullet} ="247",
            (6,4) *+{\textcolor{purple}{\bullet}} ="248",
            (7,2) *+{\textcolor{blue}{1}} ="256",
            (7,3) *+{\bullet} ="257",
            (7,4) *+{\textcolor{purple}{\bullet}} ="258",
            (8,3) *+{\bullet} ="267",
            (8,4) *+{\textcolor{purple}{\bullet}} ="268",
            (9,4) *+{\textcolor{purple}{\bullet}} ="278",
            (10,0) *+{\bullet} ="345",
            (10,1) *+{\bullet} ="346",
            (10,2) *+{\bullet} ="347",
            (10,3) *+{\bullet} ="348",
            (10,4) *+{\textcolor{purple}{\bullet}} ="349",
            (11,1) *+{\bullet} ="356",
            (11,2) *+{\bullet} ="357",
            (11,3) *+{\bullet} ="358",
            (11,4) *+{\textcolor{purple}{\bullet}} ="359",
            (12,2) *+{\bullet} ="367",
            (12,3) *+{\bullet} ="368",
            (12,4) *+{\textcolor{purple}{\bullet}} ="369",
            (13,3) *+{\bullet} ="378",
            (13,4) *+{\textcolor{purple}{\bullet}} ="379",
            (14,4) *+{\textcolor{purple}{\bullet}} ="389",
            (15,0) *+{\textcolor{red}{0}} ="456",
            (15,1) *+{\textcolor{blue}{0}} ="457",
            (15,2) *+{\bullet} ="458",
            (15,3) *+{\textcolor{red}{2}} ="459",
            (15,4) *+{\textcolor{purple}{\bullet}} ="45X",
            (16,1) *+{\bullet} ="467",
            (16,2) *+{\bullet} ="468",
            (16,3) *+{\textcolor{blue}{2}} ="469",
            (16,4) *+{\textcolor{purple}{\bullet}} ="46X",
            (17,2) *+{\bullet} ="478",
            (17,3) *+{\bullet} ="479",
            (17,4) *+{\textcolor{purple}{\bullet}} ="47X",
            (18,3) *+{\textcolor{red}{1}} ="489",
            (18,4) *+{\textcolor{purple}{\bullet}} ="48X",
            (19,4) *+{\textcolor{purple}{\bullet}} ="49X",
            "123", {\ar"124"},
            "124", {\ar"125"},
            "125", {\ar"126"},
            "126", {\ar"127"},
            "134", {\ar"135"},
            "135", {\ar"136"},
            "136", {\ar"137"},
            "145", {\ar"146"},
            "146", {\ar"147"},
            "156", {\ar"157"},
            "124", {\ar"134"},
            "125", {\ar"135"},
            "135", {\ar"145"},
            "126", {\ar"136"},
            "136", {\ar"146"},
            "146", {\ar"156"},
            "127", {\ar"137"},
            "137", {\ar"147"},
            "147", {\ar"157"},
            "157", {\ar"167"},
            "134", {\ar"234"},
            "135", {\ar"235"},
            "136", {\ar"236"},
            "145", {\ar"245"},
            "146", {\ar"246"},
            "156", {\ar"256"},
            "137", {\ar"237"},
            "147", {\ar"247"},
            "157", {\ar"257"},
            "167", {\ar"267"},
            "234", {\ar"235"},
            "235", {\ar"236"},
            "236", {\ar"237"},
            "237", {\ar"238"},
            "245", {\ar"246"},
            "246", {\ar"247"},
            "247", {\ar"248"},
            "256", {\ar"257"},
            "257", {\ar"258"},
            "267", {\ar"268"},
            "235", {\ar"245"},
            "236", {\ar"246"},
            "246", {\ar"256"},
            "237", {\ar"247"},
            "247", {\ar"257"},
            "257", {\ar"267"},
            "238", {\ar"248"},
            "248", {\ar"258"},
            "258", {\ar"268"},
            "268", {\ar"278"},
            "245", {\ar"345"},
            "246", {\ar"346"},
            "256", {\ar"356"},
            "245", {\ar"345"},
            "246", {\ar"346"},
            "247", {\ar"347"},
            "248", {\ar"348"},
            "256", {\ar"356"},
            "257", {\ar"357"},
            "258", {\ar"358"},
            "267", {\ar"367"},
            "268", {\ar"368"},
            "278", {\ar"378"},
            "345", {\ar"346"},
            "346", {\ar"347"},
            "347", {\ar"348"},
            "348", {\ar"349"},
            "356", {\ar"357"},
            "357", {\ar"358"},
            "358", {\ar"359"},
            "367", {\ar"368"},
            "368", {\ar"369"},
            "378", {\ar"379"},
            "346", {\ar"356"},
            "347", {\ar"357"},
            "357", {\ar"367"},
            "348", {\ar"358"},
            "358", {\ar"368"},
            "368", {\ar"378"},
            "349", {\ar"359"},
            "359", {\ar"369"},
            "369", {\ar"379"},
            "379", {\ar"389"},
            "356", {\ar"456"},
            "357", {\ar"457"},
            "358", {\ar"458"},
            "359", {\ar"459"},
            "367", {\ar"467"},
            "368", {\ar"468"},
            "369", {\ar"469"},
            "378", {\ar"478"},
            "379", {\ar"479"},
            "389", {\ar"489"},
            "456", {\ar"457"},
            "457", {\ar"458"},
            "458", {\ar"459"},
            "459", {\ar"45X"},
            "467", {\ar"468"},
            "468", {\ar"469"},
            "469", {\ar"46X"},
            "478", {\ar"479"},
            "479", {\ar"47X"},
            "489", {\ar"48X"},
            "457", {\ar"467"},
            "458", {\ar"468"},
            "468", {\ar"478"},
            "459", {\ar"469"},
            "469", {\ar"479"},
            "479", {\ar"489"},
            "45X", {\ar"46X"},
            "46X", {\ar"47X"},
            "47X", {\ar"48X"},
            "48X", {\ar"49X"},
        \end{xy}
    \end{equation*}
\end{Eg}

\section*{Acknowledgments}
The author would like to thank Hongrui Wei for pointing out that the previous definition of partial $d\mathbb{Z}$-cluster tilting subcategories did not, in fact, fit the purpose of this article, as explained in Remark \ref{rem_partial_dZ_ct_def_implies_dZ_rigidity}. The author would also like to thank Emre Sen for drawing attention to the relevance of \cite[Section 6.3]{Sen23} to the results in Section \ref{subsec_homo_acyclic_case}.

\bibliographystyle{amsplain}

\providecommand{\bysame}{\leavevmode\hbox to3em{\hrulefill}\thinspace}
\renewcommand{\MRhref}[2]{%
  \href{http://www.ams.org/mathscinet-getitem?mr=#1}{#2}
}
\renewcommand\MR[1]{\relax\ifhmode\unskip\space\fi MR~\MRhref{#1}{#1}}


\end{document}